%% file: tese.tex
\documentclass[twoside,openright,titlepage,numbers=noenddot,headinclude,  lineheaders footinclude=true,cleardoublepage=empty,
                                BCOR=5mm,paper=a4,fontsize=12pt ]{scrbook} 
\usepackage[usenames,dvipsnames,svgnames,table]{xcolor}
\usepackage[unicode]{hyperref}
\usepackage[utf8]{inputenc}
\usepackage[eulerchapternumbers,beramono]{classicthesis}
\usepackage{estilo}
\usepackage[T1]{fontenc}                       
\usepackage{graphicx}
\usepackage[english]{babel}
\usepackage{hyperref}
\usepackage{amsmath}
\usepackage{amsfonts}
\usepackage{amssymb}
\usepackage{amsthm}
\usepackage{breqn}
\usepackage{pdfpages}
\usepackage{makeidx}
\usepackage{blindtext}
\usepackage{mathtools}
\usepackage{bm}
\usepackage{tikz}
\usepackage{wrapfig}
\usepackage{relsize}
\usepackage{float}
\usepackage{cite}

\newtheorem{theorem}{Theorem}[chapter]
\newtheorem{lemma}[theorem]{Lemma}
\newtheorem{proposition}[theorem]{Proposition}
\newtheorem{corollary}[theorem]{Corollary}
\newtheorem{definition}[theorem]{Definition}
\newtheorem{example}[theorem]{Example}
\newtheorem{preremark}[theorem]{Remark}

\newcommand{\iprod}{\mathbin{\lrcorner}}

\newcommand{\cl}{{\mathcal{C}}\ell}

\newcommand{\oct}{\mathbb{O}}

\newcommand{\ga}{\gamma}

\newcommand{\re}{\mathbb{R}}

\newcommand{\com}{\mathbb{C}}
\newcommand{\quat}{\mathbb{H}}

\newcommand{\bege}{\begin{equation}}
\newcommand{\enge}{\end{equation}}
\newcommand{\vv}{{\textbf v}}
\newcommand{\uu}{{\textbf u}}

\newcommand{\ee}{{\textbf e}}

\newcommand{\A}{\mathbb{D}}

\newcommand{\beq}{\begin{eqnarray}}
\newcommand{\eeq}{\end{eqnarray}}

\newcommand{\benu}{\begin{enumerate}}
\newcommand{\enu}{\end{enumerate}}
\newcommand{\define}{\vcentcolon=}
\newcommand{\ree}{\text{Re}}
\newcommand{\imm}{\text{Im}}
\newcommand{\ad}{\text{Ad}}
\newcommand{\I}{\mathbb{I}}
\newcommand{\spb}{\mathcal{S}}
\newcommand{\tr}{\text{Tr}}
\newcommand{\lp}{\left(}
\newcommand{\rp}{\right)}
\numberwithin{equation}{chapter}

\newcommand{\autor}{Aquerman Yanes}

\newcommand{\titulo}{Affinely Connected Spaces, Geodesic Loops, $G_2$-Structures and Deformations}
\def\mestrado{}
\newcommand{\orientador}{Roldão da Rocha}

\def\versaofinal{}

\newcommand{\ano}{2020}
\newcommand{\centro}{Centro de Matemática, Computação e Cognição \xspace}
\newcommand{\titulacao}{Mestre em Matemática \xspace}
\newcommand{\palavraschaves}{conexões afins, \textit{loops} geodésicos, octonions, estruturas G2,
álgebras não-associativas, torção}
\newcommand{\keywords}{affine connections, geodesic loops, octonions, G2-structures, non-associative algebras, torsion} 
\makeindex

\begin{document}

\input{capa}  
\thispagestyle{empty}\newpage\mbox{}\thispagestyle{empty}\newpage\thispagestyle{empty} 
\frontmatter

\input{folha-de-rosto}

\includepdf{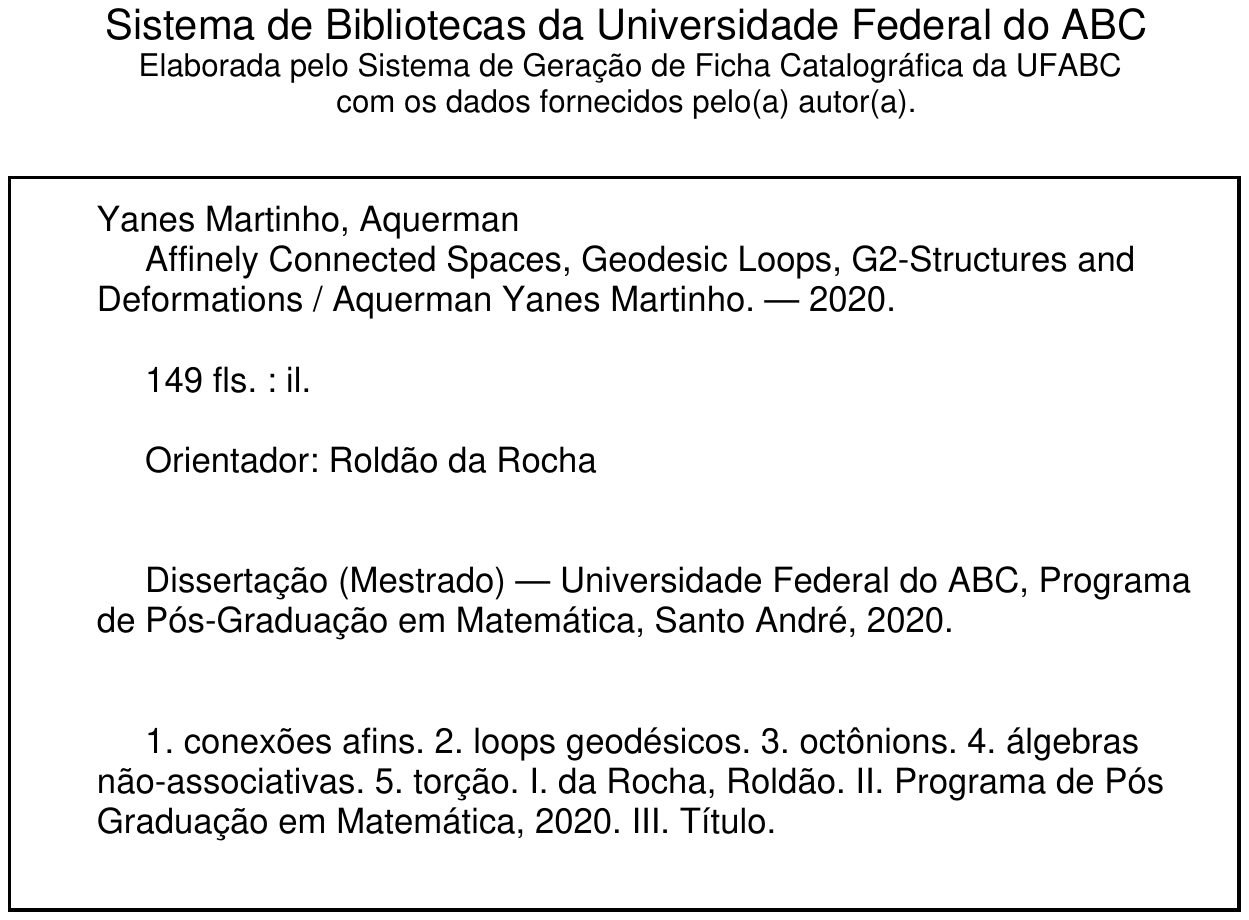}
%
\includepdf{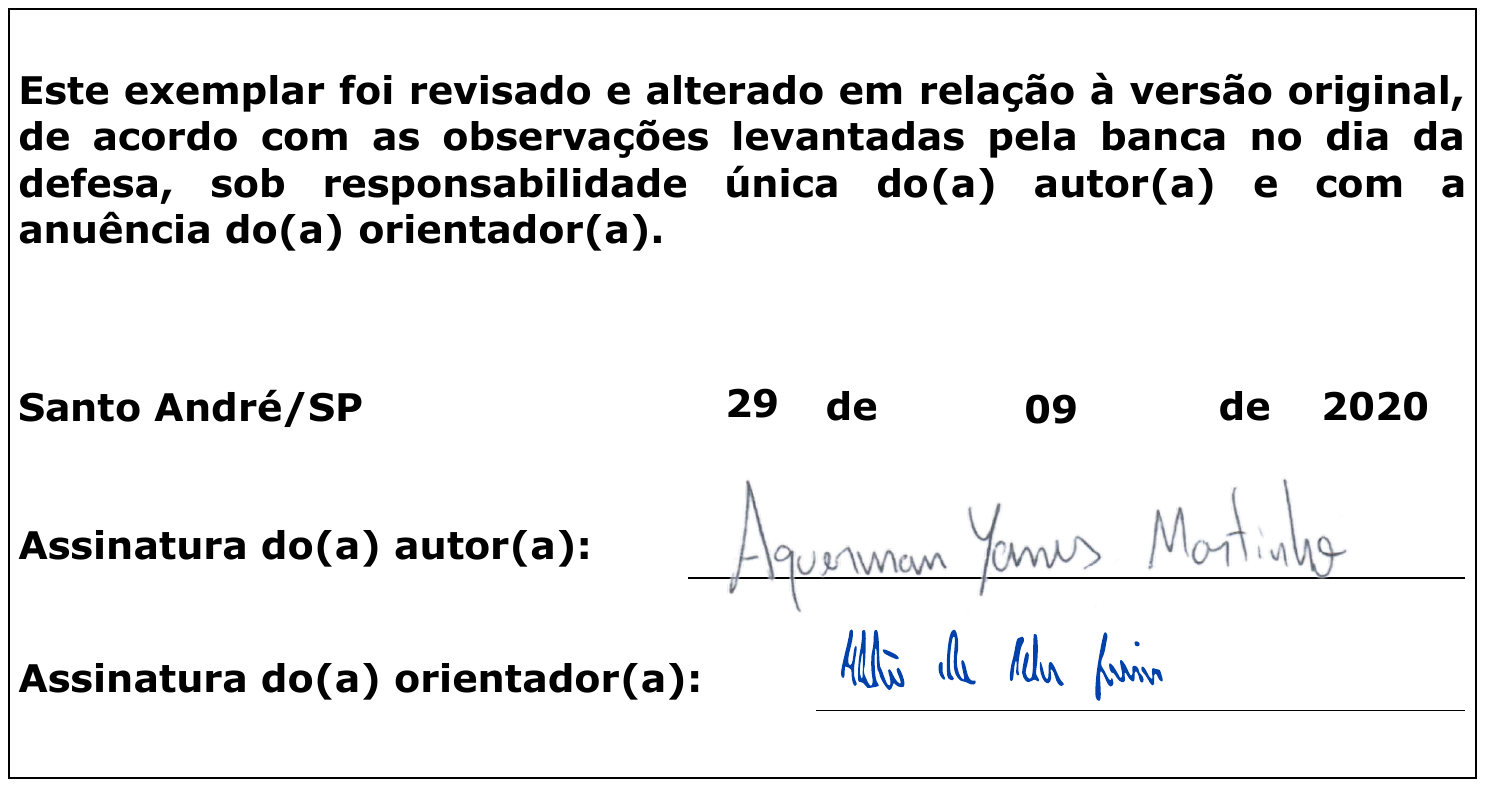}
\thispagestyle{empty}\newpage\mbox{}\thispagestyle{empty}\newpage\thispagestyle{empty}
\includepdf{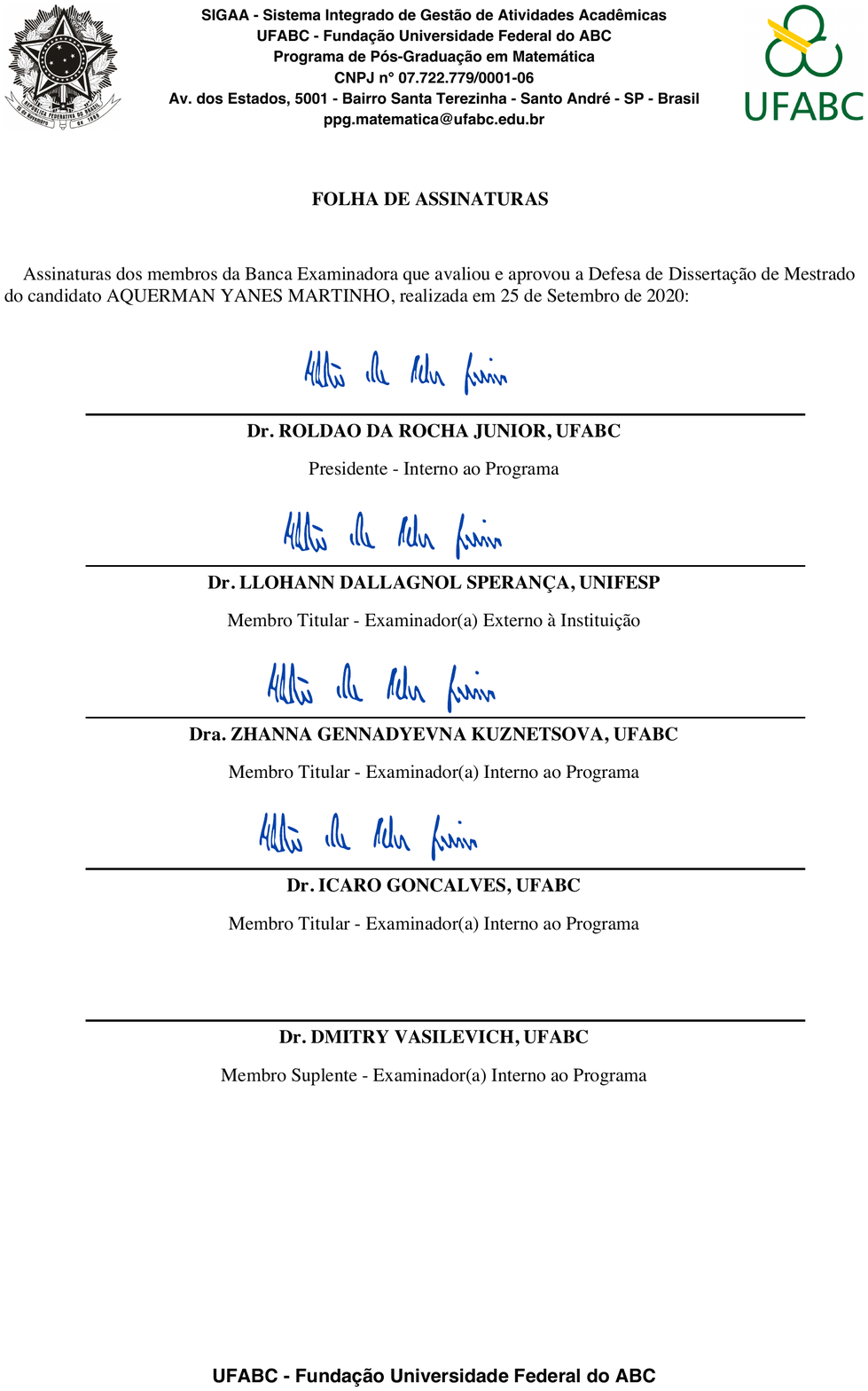}
%
\input{dedicatoria}

%
\input{agradecimentos}

%
\input{epigrafe}
\input{resumo}

\input{abstract}
\tableofcontents
\mainmatter

\newpage
\phantomsection
\addcontentsline{toc}{chapter}{\;\;\;\;\;\textsc{introduction}}
\chapter*{\textsc{introduction}}
\vspace{1cm}

The emergence of modern physics in the last century with the foundations of general relativity and quantum mechanics has enabled the coming of an unprecedented alliance between the search of physical descriptions of nature and the development of mathematical theories. The problem of unification between the two aforementioned theories is of great relevance nowadays and has been tackled from several different points of view over the last decades.

In the above-mentioned celebrated theory of gravity, the tools of Riemannian manifolds are employed in order to describe the interactions between matter and the underlying geometry, yielding then novel unexpected phenomena. Besides, the development of more general connections over manifolds introduced by Élie Cartan \cite{cartan1} has enabled the emergence of broader theories of gravity. This can be perceived by relaxing the torsionless connection requirement which is usually present in Riemannian geometry, accommodating more geometric interpretation to such theories. For instance, the non-vanishing of torsion when analyzing geometries over the $7$-sphere $S^7$ is well-known to be related to the nonassociative normed division algebra of the octonions $\oct$ and its properties. Accordingly, such algebra has been prominently studied and its physical interpretation is considered \cite{Kugo:1982bn, baez, Gunaydin:1973rs, Gunaydin:1974fb, Gunaydin:2000xr, Borsten:2008wd, Gunaydin:1995ku, Bernevig:2003yz, Kallosh:2006zs, Gunaydin:1978jq}.

On these grounds, the aforementioned relations between the geometry of a more general affinely connected manifold and related algebraic structures are examined and connections between their respective properties are scrutinized in this work. Furthermore, the so-called $G_2$-structures are considered in $7$-dimensional manifolds and their properties studied and related to the geometric information they extend to. In this configuration, octonion fields can be defined upon the space and be seen to intrinsically relate to spinor fields and its covariant derivative, a notion which extends to a vast literature in mathematical-physics \cite{HoffdaSilva:2019xvd, Lopes:2018cvu, daRocha:2005ti, daRocha:2011yr, Arcodia:2019flm, daRocha:2013qhu, daRocha:2007pz, daRocha:2008we, Bernardini:2012sc, Villalobos:2015xca,  Cavalcanti:2014uta, daRocha:2014dla, daRocha:2007sd, Rodrigues:2005yz, Bonora:2015ppa, Lopes:2018tsw}.

Chapter 1 is devoted to establishing preliminaries and notation, mainly on the theory of vector bundles. In this context, affinely connected spaces are defined, yielding the notions of geodesics, torsion and curvature, which are seen to characterize the underlying geometry of such spaces. 

In Chapter 2 Riemannian metrics are introduced and their relations with connections may be tackled. The inclusion of a metric in this discussion has as main goal the emergence of normal coordinates around a point, which will be proven to be a robust tool in what follows.
The Levi-Civita connection is also considered and its relation with a more general connection in the presence of a metric is discussed in the light of the contorsion tensor. 

Chapter 3 is devoted to developing the theory of geodesic loops. The formal definitions of local loops as originally presented by Kikkawa \cite{kikkawa} are given and the geodesic loop construction over an affinely connected manifold can be perceived. Their fundamental tensors and $W$-algebras are also discussed and their relation with the underlying geometry defined by the connection is given, as proved by Akivis \cite{akivis152}.

In Chapter 4 the Kaluza-Klein mechanism of spontaneous compactification in $d=11$ dimensions may be considered and geodesic loops may be employed to yield information about the geometry of the base space in such theory \cite{loginov2,loginov3}. Besides, solutions in the $7$-sphere $S^7$ with torsion as the Englert solution \cite{loginov2ref8} may be considered in the light of the so-called Cartan-Schouten geometries, which combined with the aforementioned techniques yield an one-parameter family of geometries over $S^7$.

In Chapter 5 normed division algebras are exposed so that the octonion algebra $\oct$ may be considered. Then, a brief discussion on the exceptional Lie group $G_2$ takes place, leading to the final section in which the linear configuration of $G_2$-structures are discussed.

Finally, in Chapter 6 the $G_2$-structures over $7$-dimensional manifolds are fully considered and are seen to yield an associated Riemannian metric. With the aid of such structures an octonion product can be defined upon the $7$-dimensional manifold \cite{grigorian}. The relation between this octonion product and the Levi-Civita connection of the associated metric is then seen to relate to the $G_2$-structure notion of torsion. Ultimately, an one-to-one correspondence between octonions and spinor fields in such space can be made, prospecting the possibility of generalizing these structures in physical applications, such as the ones in \cite{ced}.

\part{Affinely Connected Spaces}

\chapter{Vector Bundles}

This chapter is devoted to setting the preliminary results and structures that shall be used throughout this composition. In order to establish some notation, smooth manifold theory is first presented. Then, one may investigate the notion of vector bundles, which can be perceived as a generalization of the more usual tangent bundles. Following up, rudiments on affinely connected spaces and Riemannian manifolds are developed, all of which shall be extensively used subsequently. The following results have been taken from refs. \cite{leeriemann,leemanifolds,crainic,tu,michor} and omitted proofs can be also found therein.

\section{Smooth Manifolds}
Let $M$ be an $n$-dimensional manifold\footnote{In this text, every manifold considered is, in fact, a \textit{smooth} manifold. One also writes $\dim M=n$ whenever $M$ is $n$-dimensional.}. The ring of smooth real valued functions over $M$ shall be denoted by $C^{\infty}(M)$, and its elements will be simply called \textbf{functions} over $M$. The \textbf{tangent space} at $p\in M$ shall be denoted $T_p M$, and $X_p\in T_pM$ will be called a \textbf{tangent vector}, or just \textbf{vector}. One may then perceive the tangent space $T_pM$ as the space of derivations of functions over $M$ and its elements will be said to differentiate or derive a function $f\in C^{\infty}(M)$ over their direction.

The functions $r^i:\re^n\rightarrow \re$ for $i\in\{1,\ldots,n\}$ are called the \textbf{standard coordinates} on $\re^n$ and are defined by 
\begin{equation}r^i(a_1,\ldots,a_n)=a_i.\end{equation}
A local neighborhood $(U,\phi)$ around $p\in M$ shall be commonly denoted by $(U;x^1,\ldots,x^n)$, where $x^i=r^i\circ\phi:U\rightarrow \re$ are called the \textbf{local coordinates} over $U$. These coordinates define a vector (derivation) $\partial / \partial x^i\in T_p M$, which for $f\in C^{\infty}$ is given by
\begin{equation}
\frac{\partial}{\partial x^i}\Bigg{|}_p(f)\define \frac{\partial}{\partial r^i}\Bigg{|}_{\phi(p)}(f\circ\phi^{-1})\in\re,
\end{equation}
and is called a \textbf{coordinate vector} with respect to $(U,\phi)$. The point $p$ may sometimes be omitted and therefore $\partial /\partial x^i$ may be written whenever it is clear which point is being considered. The abbreviation \begin{equation}\partial_i=\partial/\partial x^i\end{equation}
shall be extensively used as well and it can then be seen that the set $\{\partial_1,\ldots,\partial_n\}$ is a basis for $T_pM$, called a \textbf{coordinate basis}. Also, if $M$ is an 1-dimensional manifold, then its local coordinates shall be denoted by $(U,t)$ with 
\begin{equation}\frac{d}{dt}\Bigg{|}_{t_0}\in T_{t_0} M\end{equation} 
its coordinate vector at $t_0\in U$. 

\begin{example}\upshape
The Euclidean $n$-dimensional space $M=\re^n$ is a manifold and its tangent space at each $p\in M$ may be naturally perceived as itself $T_pM\simeq M= \re^n$
\end{example}
\begin{example}\upshape
The circle $S^1=\{(x,y)\in\re^2\;:\;x^2+y^2=1\}$, which may also be perceived as the subset of the complex numbers $\com$ given by $S^1=\{e^{2\pi i\theta}\;:\;\theta\in[0,1]\}$ is a manifold. More generally, one may consider the $n$-sphere $S^n=\{(x^1,\ldots,x^{n+1})\in\re^n\;:\;(x^1)^2+\cdots(x^{n+1})^2=1\}$.
\end{example}
\begin{example}\upshape
Given two manifolds $M$ and $N$, respectively $n$- and $k$-dimensional, then the direct product $M\oplus N=M\times N$ is also a manifold of dimension $m+n$. Using the last example, one has the $n$-torus $T^n=S^1\times\cdots\times S^1$ which is a product of $n$ copies of $S^1$ circles.
\end{example}
\begin{example}\upshape
Denote by $\text{M}(n,\re)$ the space of $n\times n$ real matrices. The (real) general linear group $\text{GL}(n,\re)$ of invertible $n\times n$ matrices is an $n^2$-dimensional manifold. If $\det:\text{M}(n,\re)\rightarrow \re$ denotes the usual determinant function, it follows by continuity that $\text{GL}(n,\re)=\re\backslash\det^{-1}(\{0\})$ is open in $\text{M}(n,\re)\simeq\re^{n^2}$. Since every subset $U\subset M$ of a manifold $M$ is a manifold itself of same dimension if and only if it is open, then it follows that $\text{GL}(n,\re)$ is an $n^2$-dimensional manifold.
\end{example}

\section{Tangent Bundle}
The \textbf{tangent bundle} associated to $M$ is denoted by $TM$ and is defined by
\begin{equation}TM=\bigcup_{p\in M}(\{p\}\times T_pM).\end{equation}
A general element of the tangent bundle is usually written as $v\in TM$ and one can endow $TM$ with a natural projection $\pi:TM\rightarrow M$ given by $\pi(v)=p$ if $v\in T_p M$. One may write $v=(p,v)\in TM$ to explicitly show that $\pi(v)=p$.

For every open set $U\subset M$ one can define $TU=\bigcup_{p\in U}T_p U=\bigcup_{p\in U} T_pM$. Since $\{\partial_1,\ldots,\partial_n\}$ is a basis for $T_pM$, a vector $v\in TU$ is locally given by
\begin{equation}v=\sum_{i=1}^n a^i\frac{\partial}{\partial x^i}\Bigg{|}_p,\end{equation}
where $a^i:TU\rightarrow \re$ are smooth functions. 

Let now $\tilde{x}^i=x^i\circ\pi$, and define the map $\tilde{\phi}:TU\rightarrow\phi(U)\times\re^n$ by the relation
\begin{equation}\tilde{\phi}(v)=(\tilde{x}^1(v),\ldots,\tilde{x}^n(v),a^1(v),\ldots,a^n(v)),\end{equation}
which has as an inverse given by
\begin{equation}\tilde{\phi}^{-1}(\phi(p),a^1,\ldots,a^n)=\sum_{i=1}^n a^i\frac{\partial}{\partial x^i}\Bigg{|}_p.\end{equation}
Hence, $\phi$ is a bijection and one may transfer the topology from $\phi(U)\times\re^n\subset \re^{2n}$ to $TU=\pi^{-1}(U)$ by saying that a set $A\subset TU$ is open in $TM$ whenever $\tilde{\phi}(A)$ is open in $\phi(U)\times\re^n$. 

Now, suppose that $(V,\psi)=(V;y^1,\ldots,y^n)$ are other local coordinates of $M$ with $U\cap V\neq\emptyset$. Then a vector $v\in T(U\cap V)$ has
\begin{equation}v=\sum_{j=1}^na^j\frac{\partial}{\partial x^j}\Bigg{|}_p=\sum_{i=1}^n b^i\frac{\partial}{\partial y^i}\Bigg{|}_p,\end{equation}
in such a way that
\begin{equation}
a^k=\Bigg{(}\sum_{j=1}^na^j\frac{\partial}{\partial x^j}\Bigg{|}_p\Bigg{)}x^k=\Bigg{(}\sum_{i=1}^n b^i\frac{\partial}{\partial y^i}\Bigg{|}_p\Bigg{)}x^k=\sum_{i=1}^n b^i\frac{\partial x^k}{\partial y^i}\Bigg{|}_p,\end{equation}
and analogously there holds
\begin{equation}b^k=\sum_{j=1}^n a^j\frac{\partial y^k}{\partial x^j}.\end{equation}
Then, the map \begin{equation}\tilde{\psi}\circ\tilde{\phi}^{-1}:\phi(U\cap V)\times\re^n\rightarrow \psi(U\cap V)\times\re^n\end{equation}
for $\tilde{\psi}$ and $\tilde{\phi}$ as defined before is given by
\begin{equation}\tilde{\psi}\circ\tilde{\phi}^{-1}(x,a^1,\ldots,a^n)=(\psi\circ\phi^{-1}(p),b^1,\ldots,b^n),\end{equation}
with
\begin{equation}b^k=\sum_{j=1}^n=a^j\frac{\partial y^k}{\partial x^j}\Bigg{|}_p=\sum_{j=1}^na^j\frac{\partial(r^i\circ \psi\circ\phi^{-1})}{\partial r^j}(\phi(p)).\end{equation}
Since $\psi\circ\phi^{-1}$ is smooth, so is $\tilde{\psi}\circ\tilde{\phi}^{-1}$. Therefore, the atlas $\{T(U_{\alpha}),\tilde{\phi}_{\alpha}\}=\{(\pi^{-1}(U_{\alpha}),\tilde{\phi}_{\alpha})\}$, inherited from the smooth atlas $\{(U_{\alpha},\phi_{\alpha}\}$ from $M$ is indeed smooth and it follows that $TM$ is a $2n$-dimensional manifold itself.

\begin{definition}\label{vectorfields}
An application $X:M\rightarrow TM$ is called a \textbf{vector field} if $X$ is smooth and $\pi\circ X=\emph{\text{Id}}$, that is \begin{equation}X(p)\in T_pM,\end{equation} for each $p\in M$.
\end{definition}
Sometimes one writes $X(p)=X_p$ or just $X=X_p$ whenever there is no ambiguity about the point $p$. If $(U;x^1,\ldots,x^n)$ are local coordinates, then there are $n$ smooth functions $X^i:U\rightarrow \re$, with $i\in\{1,\ldots,n\}$, such that\footnote{The Einstein summation convention is considered throughout the text, in which one suppresses the summation symbol and sum over identical indices in different positions, where the indices can be on the upper ($X^i$) or bottom ($\partial_i$) of the terms.  }
\bege\label{vectorx}X=\sum_{i=1}^n X^i\partial_i=X^i\partial_i,\enge
and the smoothness condition for $X$ is equivalent to each $X^i$ being smooth for every $i\in\{1,\ldots,n\}$. It follows that the derivation of a function $f$ in the direction of $X$ is locally given by
\begin{equation}X(f)(p)=\sum_{i=1}^n X^i(p)\frac{\partial f}{\partial x^i}(p),\end{equation}
for each $p\in U$.
 
If $F:N\rightarrow M$ is a smooth map between two manifolds, then for each $p\in N$ one can define a linear map induced by $F$ which generalizes for manifolds the notion of derivative of an application.
\begin{definition}
Let $F:N\rightarrow M$ be a smooth map between the manifolds $M$ and $N$. Then, the \textbf{differential} of $F$ at $p\in N$ is a linear map
\begin{equation}F_{*,p}:T_pN\rightarrow T_{F(p)}M,\end{equation}
which is given by 
\begin{equation}F_{*,p}(X_p)(f)=X_p(f\circ F)\in \re,\end{equation}
For each $X_p\in T_pN$ and $f\in C^{\infty}(M)$. The image of a vector $X_p\in T_p N$ under the differential $F_{*,p}$ is called the \textbf{push-forward} of the vector $X_p$ by $F$.
\end{definition}
 Let $p\in N$ and consider local coordinates $(U;x^1,\ldots,x^k)$ around $p$ and $(V,y^1,\ldots,y^n)$ around $F(p)\in M$. Since $\{\partial\backslash\partial x^1,\ldots,\partial\backslash\partial x^k\}$ is a basis for $T_pN$ and $\{\partial\backslash\partial y^1,\ldots,\partial\backslash\partial y^n\}$ is a basis for $T_{F(p)}M$ there are, for each $j\in\{1,\ldots,k\}$, $n$ real numbers $a^i_j\in\re$ with $i\in\{1,\ldots,n\}$ such that
\begin{equation}F_{*,p}\Big{(}\frac{\partial}{\partial x^j}\Big{)}=a^l_j\frac{\partial}{\partial y^l}.\end{equation}
One can see that
\begin{equation}a^l_j\frac{\partial}{\partial y^l}(y^i)=a^l_j\delta^i_l=a^i_j.\end{equation}
On the other hand, writing $F^i=y^i\circ F$ there holds
\begin{equation}F_{*,p}\Big{(}\frac{\partial}{\partial x^j}\Big{)}(y^i)=\frac{\partial}{\partial x^j}(y^i\circ F)=\frac{\partial F^i}{\partial x^j}(p).\end{equation}
Therefore, in relation to a choice of local coordinates, the matrix $(a^i_j)$ looks like the Jacobian matrix for $F$, showing that the presented differential is a generalization of the differential of applications in $\re^n$. One may then write
\begin{equation}
    F_{*,p}=dF_p.
\end{equation}

A \textbf{parameterized smooth curve} in $M$ is a smooth application $\ga:I\rightarrow M$, where $I$ is a real open interval. For simplicity, here they are just called \textbf{curves}. 
\begin{definition}
Let $\gamma:I\rightarrow M$ be a curve in $M$. Then, its \textbf{velocity vector} $\ga'(t_0) $ at $t_0\in I$ is defined as
\begin{equation}\gamma'(t_0)\define =\gamma_{*,t_0}\Big{(}\frac{d}{dt}\Big{)}\in T_{\gamma(t_0)}M.\end{equation}
\end{definition}
Let $(U,x^1,\ldots,x^n)$ be local coordinates around $\gamma(t_0)$ in $M$. One can then define the \textbf{components} of $\ga$, namely \begin{equation}\ga^i=x^i\circ\ga:I\rightarrow \re.\end{equation} Then, the expression $\dot{\ga}^i(t_0)$ may be used in order to denote the usual (calculus) real derivative of $\ga^i$ at $t_0$. It then follows that
\begin{equation}\gamma'(t)=\sum_{i=1}^n\dot{\gamma}^i(t)\partial_i,\end{equation}
for every $t\in I$ with $\gamma(t)\in U$. 

\begin{proposition}
Let $M$ be a manifold, $p\in M$ and $X_p\in T_p M$. Then, there is a curve $\gamma:(-\varepsilon,\varepsilon)\rightarrow M$ such that $\gamma(0)=p$ and $\gamma'(0)=X_p$. Moreover, if $f\in C^{\infty}(M)$ then
\begin{equation}X_p(f)=\frac{d}{dt}\Bigg{|}_{t=0}(f\circ \gamma).\end{equation}
More generally, if $F:N\rightarrow M$ is a smooth map, $p\in N$ and $X_p\in T_p N$, then taking a curve $\gamma:(-\varepsilon,\varepsilon)\rightarrow N$ with $\ga(0)=p$ and $\ga'(0)=X_p$ it follows that
\begin{equation}dF_p(X_p)=\frac{d}{dt}\Bigg{|}_{t=0}(F\circ c)(t).\end{equation}
\end{proposition}

\section{Vector Bundles}
One may now define the more general concept of vector bundles, which is a generalization of the tangent bundle over a manifold $M$, allowing more general vector spaces in the fibers. From now on, fix the manifold dimension as $\dim M=n$.
\begin{definition}
A (real) \textbf{rank} $\bm{k}$ \textbf{vector bundle} over a manifold $M$ is a triple $(E,\pi, M)$ such that
\begin{enumerate}
    \item The set $E$ is a manifold, called the \textbf{total space};
    \item The application $\pi: E \rightarrow M$ is a smooth surjective map: for each $p\in M$, the set $\pi^{-1}(\{p\})=E_p$ is denoted the \textbf{fiber} at $p$ and is endowed with a vector space structure;
    \item There are \textbf{local trivializations}: for every $p_0\in M$ there is an open neighborhood $U\subset M$ of $p_0$ and a diffeomorphism \begin{equation}\Phi:E{\big|}_U=\pi^{-1}(U)\rightarrow U\times \re^k,\end{equation}
    called the local trivialization, with the property that $\pi_1\circ \Phi=\pi$, where $\pi_1:U\times\re^k\rightarrow U$ is the natural projection in $U$. Moreover, for each $p\in U$ the restriction $\Phi{\big|}_{E_p}$ is an isomorphism (of vectors spaces) from $E_p$ to $\{p\}\times\re^k\simeq\re^k$.
\end{enumerate}
\end{definition}
In an intuitive way, a vector bundle is a collection of vector spaces $\{ E_p\}_{p\in M}$, smoothly parameterized over $M$. The manifold $M$ is called the \textbf{base space} and $\pi$ is called the \textbf{projection}. Instead of using the triple notation, one shall sometimes just say that $\pi:E\rightarrow M$ is a vector bundle or that $E$ is a vector bundle over $M$, whenever it is clear or unnecessary to depict the projection $\pi$.

\begin{example}
For each natural number $k$ the product $M\times \re^k$ endowed with the first projection $\pi_1:M\times\re^k\rightarrow M$ and the usual vector space structure on the fibers $\{x\}\times\re^k$ is a vector bundle, called the \textbf{product bundle} of rank $k$ over $M$.
\end{example}
If there is a local trivialization of $E$ defined all over $M$, then such trivialization is called a \textbf{global trivialization} and $E$ is called a \textbf{trivial bundle}. It follows that $E$ is diffeomorphic to the product bundle $M\times\re^k$. Notice that if $U$ is an arbitrary open set in $M$, then $E{\big|}_U\vcentcolon=\pi^{-1}(U)$ gives rise to a vector bundle $\pi_U:E{\big|}_U\rightarrow U$ over $U$ which is trivial by definition.
\begin{proposition}
Let $M$ be an $n$-dimensional manifold and let $TM$ be its tangent bundle. Then, endowed with its natural projection $\pi:TM\rightarrow M$ and the vector space structure of $T_p M$ on each fiber, $TM$ is a rank $n$ vector bundle.
\end{proposition}
\begin{proof}
Let $(U,\phi)=(U;x^1,\ldots,x^n)$ be local coordinates on $M$ and $\pi:TM\rightarrow M$ the natural projection. Then, define the map $\Phi:\pi^{-1}(U)\rightarrow U\times \re^n$ by
\begin{equation}
\Phi\Bigg{(}v^i\frac{\partial}{\partial x^i}\Bigg{|}_p\Bigg{)}=\lp p,(v^1,\ldots,v^n)\rp,
\end{equation}
which is clearly linear on each fiber $\pi^{-1}(\{p\})=T_p M$ and satisfies $\pi_1\circ\Phi=\pi$. Moreover, notice that
\begin{equation}\tilde{\phi}=\lp\phi\times \text{Id}_{\re^n}\rp\circ\Phi,
\end{equation}
and since $\tilde{\phi}$ and $(\phi\times \text{Id}_{\re^n})$ are diffeomorphisms, so is $\Phi$. Therefore, $TM$ is a vector bundle of rank $n$.
\end{proof}
One may investigate what happens in the overlap of two trivializations on a vector bundle $E$ over $M$. In fact, this is given by the very useful
\begin{lemma}\label{tauzin}
Let $\pi:E\rightarrow M$ be rank $k$ vector bundle and let $\Phi:\pi^{-1}(U)\rightarrow U\times\re^k$ and $\Psi:\pi^{-1}(V)\rightarrow V\times\re^k$ be two local trivializations with $U\cap V\neq\emptyset$. Then, there is a smooth map $\tau:U\cap V\rightarrow \text{GL}(k,\re)$ such that 
\begin{equation}\Phi\circ\Psi^{-1}\lp p,v\rp=\lp p,\tau\lp p\rp v\rp.\end{equation}
\end{lemma}

The application $\tau$ depicted in Lemma \ref{tauzin} is called the \textbf{transition map} between the local trivializations $\Phi$ and $\Psi$. In the example of the tangent bundle $TM$ over $M$, the transition map associated with two charts is the Jacobian matrix of the coordinate transition map.

It is possible to construct examples of vector bundles from the gluing of a collection of vector spaces indexed by points in $M$. In the following results, the properties that such indexation must satisfy so that such gluing should indeed be a vector bundle are presented. 
\begin{definition}A \textbf{rank} ${\bm k}$ \textbf{discrete vector bundle} $E$ over $M$ is a collection of ($k$-dimensional) vector spaces $E_p$ indexed by $p\in M$, endowed with a projection $\pi:E\rightarrow M$. Namely 
\begin{equation}E=\{E_p\}_{p\in M}=\{(p,v_p)\;:\;p\in M,\;v_p\in E_p\},\end{equation}
for which $\pi(p,v_p)=p$. 
\end{definition}

\begin{lemma}[Vector Bundle Chart Lemma]\label{canetavoadora}
Let $E=\{E_p\}_{p\in M}$ be a rank $k$ discrete vector bundle over $M$ and assume there is an open cover $\{U_{\alpha}\}_{\alpha\in A}$ of $M$ for which there holds:
\begin{enumerate}
    \item For each $\alpha\in A$ there is a bijective map $\Phi_{\alpha}:\pi(U_{\alpha})\rightarrow U_{\alpha}\times\re^k$ which restricts to each $E_p$ as a vector space isomorphism into $\{p\}\times\re^k\simeq\re^k$.
    \item For each $\alpha,\beta\in A$ such that $U_{\alpha}\cap U_{\beta}\neq\emptyset$ there is a smooth map $\tau_{\alpha\beta}:U_{\alpha}\cap U_{\beta}\rightarrow\text{GL}(k,\re)$ for which the map $\Phi_{\alpha}\circ\Phi_{\beta}^{-1}$ has the form
    \begin{equation}\Phi_{\alpha}\circ\Phi_{\beta}^{-1}\lp p,v\rp =\lp p,\tau_{\alpha\beta}\lp p\rp v\rp .\end{equation}
\end{enumerate}
Then, $E$ has an unique topology and smooth structure for which it is a manifold and a rank $k$ vector bundle over $M$, with $\pi$ as projection and $\{(U_{\alpha},\Phi_{\alpha}\}$ its atlas.
\end{lemma}

Lemma \ref{canetavoadora} guarantees that operations with vector bundles such as the ones that can be done with respect to vector spaces produce new vector bundles, as long as overlaps of local descriptions are smooth. Since vector bundles are gluing of vector spaces together, this proposition formalizes the process under which one must proceed in order to uniquely define a smooth structure over it. Namely, for $E$ and $F$ vector bundles over $M$ one may define the following vector bundles:

\begin{example}[Direct sums]
Suppose $\pi':E\rightarrow M$ and $\pi'':F\rightarrow M$ are ranks $k'$ and $k''$ vector bundles, respectively. Then, one can construct the \textbf{direct sum bundle} between $E$ and $F$ with fibers at $p\in M$ equal to $E_{p}\oplus F_{p}$. The total space is $E\oplus F=\bigcup_{p\in M}(\{p\}\times(E_p\oplus F_p))$ endowed with the obvious projection $\pi:E\rightarrow M$. Consider a neighborhood $U$ of $p\in M$ and the local trivializations $(U,\Phi')$ for $E$ and $(U,\Phi'')$ for $F$ and define the map $\Phi:\pi^{-1}(U)\rightarrow U\times\re^{k'+k''}$ by
\begin{equation}\Phi(v',v'')=\lp\pi'(v'),\lp\pi_{\re^{k'}}\circ\Phi'(v'),\pi_{\re^{k''}}\circ\Phi''(v'')\rp\rp.\end{equation}
If $(\tilde{U},\tilde{\Phi}')$ and $(\tilde{U},\tilde{\Phi}'')$ are two other local trivializations for $E$ and $F$, respectively, then one can similarly define the map $\tilde{\Phi}$ using them. By Lemma \ref{tauzin}, there are two transition maps $\tau':U\cap\tilde{U}\rightarrow\text{GL}(k',\re)$ between $\phi'$ and $\tilde{\phi}'$ and $\tau'':U\cap\tilde{U}\rightarrow\text{GL}(k'',\re)$ between $\phi''$ and $\tilde{\phi}''$. Then, the transition map for $E\oplus F$ between $\Phi$ and $\tilde{\Phi}$ is given by
\begin{equation}\tilde{\Phi}\circ\Phi^{-1}(p,(v',v''))=\lp p,\tau(p)(v',v'')\rp,\end{equation}
where $\tau(p)=\tau'(p)\oplus\tau''(p)\in\text{GL}(k
+k'',\re)$, which in matrix form is given by
\begin{equation}
\begin{pmatrix}
\tau'(p)&0\\0&\tau''(p)
\end{pmatrix}.\end{equation}
Since this expression depends smoothly on $p$, by the Chart Lemma it follows that $E\oplus F$ is indeed a vector bundle over $M$.
\end{example}

\begin{example}\label{dual}[Dual space]
Now, suppose $\pi:E\rightarrow M$ is a rank $k$ vector bundle over $M$. One may define its \textbf{dual bundle} given by $E^*=\bigcup_{p\in M}(\{p\}\times(E_p)^*)$ with $\pi^*:E^*\rightarrow M$ being the obvious projection. For each $p\in M$ one may choose an isomorphism $T_p:(E_p)^*\rightarrow E_p$ and let $(U,\Phi)$ be a local trivialization around the point. Define $\Phi^*:\pi^*(U)\rightarrow U\times\re^k$, by
\begin{equation}\Phi^*(\omega)=\lp\pi^*(\omega),(\pi_{\re^k}\circ\Phi\circ T_{\pi^*(\omega)})(\omega)\rp.\end{equation}
Such mapping clearly satisfies the second condition from the Chart lemma, so that one needs only to verify condition 3. Indeed, if $(U_{\alpha},\Phi_{\alpha})$ and $(U_{\beta},\Phi_{\beta})$ are local trivializations with $U_{\alpha}\cap U_{\beta}\neq\emptyset$ then respectively define the applications $\Phi_{\alpha}^*$ and $\Phi_{\beta}^*$ as done before. Now, if $\tau_{\alpha\beta}:\pi(U_{\alpha}\cap U_{\beta})\rightarrow\text{GL}(k,\re)$ is the transition map with respect to $\phi_{\alpha}$ and $\phi_{\beta}$ then one can see that
\begin{equation}
    \begin{split}
\Phi_{\alpha}^*\circ(\Phi_{\beta}^*)^{-1}(p,v)&=\Phi_{\alpha}^*\lp T_p^{-1}\circ\Phi_{\beta}^{-1}(p,v)\rp\\
&=\lp p,\pi_{\re^k}\circ\Phi_{\alpha}\circ T_p\circ T^{-1}_p\circ\Phi_{\beta}^{-1}(p,v)\rp\\
&=\lp p,\pi_{\re^k}(\Phi_{\alpha}\circ\Phi_{\beta}^{-1}(p,v))\rp\\
&=\lp p,\tau_{\alpha\beta}(p)v\rp,
\end{split}
\end{equation}
so that $\tau_{\alpha\beta}$ is the transition map between $\tilde{\phi}_{\alpha}$ and $\tilde{\phi}_{\beta}$ as well, which completes the construction.
\end{example}
\begin{definition}
Let $\pi:E\rightarrow M$ be a vector bundle. A (smooth) map $S:M\rightarrow E$ such that $\pi\circ S=\text{Id}$ is called a \textbf{(smooth) section} in $E$.
\end{definition}
\begin{preremark}\upshape The requirement that $\pi\circ S=\text{Id}$ can be more easily seen as $S(p)\in E_p$, for every $p\in M$, that is, for each evaluation on $p$, the vector $S(p)$ lies precisely on the fiber over $p$, this of course just a generalization of the notion of vector fields in Definition \ref{vectorfields}.
\end{preremark}
From now on, smooth sections shall simply be called sections. If $E$ is a vector bundle over a manifold $M$, then one may set
\begin{equation}
    \Gamma(E)=\{S:M\rightarrow E\;:\;S\;\text{is a section over $E$}\},
\end{equation}
 whose elements can be point-wisely added and multiplied by scalars making use of the vector space structure in each fiber. With that said, $\Gamma(E)$ is endowed with a vector space structure. Moreover, if $f\in C^{\infty}(M)$ then
\begin{equation}(fS)(p)=f(p)S(p)\end{equation}
defines a section $fS$ in such a way that the previous definition turns $\Gamma(E)$ into a module over the algebra of smooth function $C^{\infty}(M)$. 
\begin{preremark}\upshape
Taking $E=TM$ in Example \ref{dual} yields the \textbf{cotangent bundle} $T^*M$. Then, for each $p\in M$ the fiber is given by $T^*_pM$, the vector space dual to $T_pM$. In addition, sections of this space are given by smooth applications $\omega:M\rightarrow T^*M$ which for each $p\in M$ define a linear functional $\omega_p:T_p M\rightarrow \re$. Such elements are called the ${\bm 1}$-\textbf{forms} over the manifold $M$ and one may denote
\begin{equation}
    \Gamma(T^*M)=\Omega^1(M).
\end{equation}
\end{preremark}
\begin{definition}
A \textbf{local section} on $E$ is a smooth application $S:U\rightarrow E$ defined over some open subset $U\subset M$ such that $\pi\circ S=\text{Id}$. In order to emphasize the difference, sometimes sections (defined over all of $M$) shall be denoted \textbf{global sections}.
\end{definition}
\begin{example}
As said before, vector fields are section from the tangent bundle $TM$ over $M$. A special symbol is commonly given to the space $\Gamma(TM)$ of such sections, namely
\begin{equation}\Gamma(TM)=\mathfrak{X}(M).\end{equation}
\end{example}
\begin{example}
The \textbf{zero section} of $E$ is a global section $Z:M\rightarrow E$ with
\begin{equation}Z(p)=0\in E_p,\;\text{for every}\;p\in M.\end{equation}
\end{example}
\begin{example}
Let $E=M\times \re^k$ be the product bundle of rank $k$ over $M$. Then, there is a bijective correspondence between smooth applications $S:M\rightarrow \re^k$ and section $\widetilde{S}:M\rightarrow M\times\re^k$, given by
\begin{equation}\tilde{S}(x)=(x,S(x)).\end{equation}
Such relation produces a natural identification between $C^{\infty}(M)$ and the trivial line bundle $M\times \re$.
\end{example}

\begin{definition}\label{frame}
Let $(E,\pi,M)$ be a rank $k$ vector bundle. A \textbf{frame} for $E$ over $M$ is a collection $\{S_1,\ldots,S_k\}$ of sections $S_i:M\rightarrow E$ such that for each $p\in M$ the set $\{S_1(p),\ldots,S_k(p)\}$ is a basis for the vector space $E_p$.
\end{definition}

\begin{preremark}\upshape 
A \textbf{local frame} for $E$ over an open set $U$ is a collection $\{S_1,\ldots,S_k\}$ of sections $S_i:U\rightarrow E$ for which $\{S_1(p),\ldots,S_k(p)\}$ is a basis for $E_p$, whenever $p\in U$. Analogously, one may say that a frame as given in Definition \ref{frame} is called a \textbf{global frame} so that it is clear that its domain is the whole space $M$.
\end{preremark}

\begin{example}
Let $E=M\times\re^k$ be a product bundle. The canonical basis $\{e_1,\ldots,e_k\}$ for $\re^k$ produces a global frame $\{\tilde{e}_i,\ldots,\tilde{e}_k\}$ denoted the \textbf{canonical frame} for the product space, which is defined by
\begin{equation}\tilde{e}_i(p)=(p,e_i).\end{equation}
\end{example}

\begin{definition}
Let $\pi:E\rightarrow M$ be a rank $k$ vector bundle over $M$ and $\{S_1,\ldots,S_k\}$ a local frame for $E$ over $U$. If there is a local trivialization $\Phi:\pi^{-1}(U)\rightarrow U\times\re^k$ such that
\begin{equation}S_i(p)=\Phi^{-1}\circ\tilde{e}_i(p),\end{equation}
then one says the local frame $\{S_1,\ldots,S_k\}$ \textbf{is associated with} $\Phi$. 
\end{definition}

\begin{proposition}
Let $\pi:E\rightarrow M$ be a rank $k$ vector bundle. Then, given a local trivialization $\Phi:\pi^{-1}(U)\rightarrow U\times\re^k$ there is a local frame $\{\sigma_1,\ldots,\sigma_k\}$ for $E$ over $U$ which is associated with $\Phi$. On the other hand, for every local frame for $E$ over $U$ there exists a local trivialization $\Phi$ associated.
\end{proposition}
\begin{proof}
Let $\Phi:\pi^{-1}(U)\rightarrow U\times\re^k$ be a local trivialization over $U$. Then, consider the canonical basis $\{e_1,\ldots,e_k\}$ for $\re^k$ and define \begin{equation}\sigma_i(p)=\Phi^{-1}(p,e_i).\end{equation}
One may proceed to show that $\{\sigma_1,\ldots,\sigma_k\}$ is a frame. Indeed, since $\Phi^{-1}$ is a diffeomorphism, the maps $\sigma_i$ are smooth for every $i\in\{1,\ldots,k\}$. Besides,
\begin{equation}\pi\circ\sigma_i(p)=\pi\circ\Phi^{-1}(p,e_i)=\pi_1(p,e_i)=p,\end{equation}
and it follows that each $\sigma_i$ is a section for $E$ over $U$. Now, notice that $\Phi$ restricted to $E_p$ is an isomorphism onto $\{p\}\times\re^k\simeq\re^k$ and it maps the canonical basis of $\re^k$ to $\{\sigma_1,\ldots,\sigma_k\}$ since
\begin{equation}\Phi(\sigma_i(p))=(p,e_i).\end{equation}
It then follows that $\{\sigma_1(p),\ldots,\sigma_k(p)\}$ is indeed a basis for $E_p$. Therefore, $\sigma$ is a local frame associated with $\Phi$.

Conversely, suppose $\sigma=\{\sigma_1,\ldots,\sigma_k\}$ is a smooth local frame for $E$ over $U$ and let $\Psi:U\times\re^k\rightarrow\pi^{-1}(U)$ be defined by
\begin{equation}\Psi(p,(v^1,\ldots,v^k))=v^i\sigma_i(p).\end{equation}
Notice that since $\{\sigma_1(p),\ldots,\sigma_k(p)\}$ is a basis for each $p\in M$, it follows that $\Psi$ is bijective. Also,
\begin{equation}\Psi\circ\tilde{e}_i(p)=\Psi(p,e_i)=\sigma_i(p).\end{equation}
Therefore, if it is proven that $\Psi$ is a diffeomorphism then $\Psi^{-1}$ will precisely be the trivialization associated to the local frame $\sigma$.
It suffices to show that $\Psi$ is a local diffeomorphism, since it is already bijective. For that end, let $q\in U$ and consider a trivialization $\Phi:\pi^{-1}(V)\rightarrow V\times\re^k$. One can consider $V\subset U$, otherwise just take $V'=V\cap U$ and use such open set instead of $V$. Notice that if one shows that the map $\Phi\circ\Psi\Big{|}_{V\times\re^k}$ is a diffeomorphism, then since $\Phi$ is one itself then it must follow that $\Psi$ restricts to a diffeomorphism from $V\times\re^k$ to $\pi^{-1}(V)$.

Now, for each $i\in\{1,\ldots,k\}$ the composite map 
\begin{equation}\Phi\circ\sigma_i{\Bigg |}_V:V\rightarrow V\times\re^k\end{equation}
is smooth, so there are $k$ smooth functions $\sigma_i^1,\ldots,\sigma^k_i:V\rightarrow\re$ such that
\begin{equation}\Phi\circ\sigma_i(p)=\lp p,\lp \sigma^1_i(p),\ldots,\sigma^k_i(p)\rp\rp.\end{equation}
Therefore, on $V\times\re^k$ there holds
\begin{equation}\Phi\circ\Psi\lp p,(v^1,\ldots,v^k)\rp=\lp p,\lp v^i\sigma^1_i(p),\ldots,v^i\sigma^k_i(p)\rp\rp,\end{equation}
which is also smooth.

To show smoothness of $(\Phi\circ\Psi)^{-1}$, just notice that the matrix $(\sigma^j_i(p))$ is invertible for every $p\in V$, since $\{\sigma_1(p),\ldots,\sigma_k(p)\}$ is a basis for $E_p$. If $(\tau^j_i)$ is its inverse, then since inversion is a smooth map from $\text{GL}(k,\re)$ to itself, the functions $\tau^j_i$ are all smooth. Finally, 
\begin{equation}(\Phi\circ\Psi)^{-1}\lp p,(w^1,\ldots,w^k)\rp=\lp p,\lp w^i\tau^1_i(p),\ldots,w^i\tau^k_i)\rp\rp,\end{equation}
and therefore $\Phi\circ\Psi$ is a diffeomorphism from $V\times\re^k$ to itself, which concludes the proof.
\end{proof}

By the last proposition, it is possible to see that the local trivializability property of a vector bundle $E$ over $M$ is equivalent to the existence, for each $p\in M$, of local frames around a neighborhood $U$ for $p$. This implies the
\begin{corollary}
A vector bundle $E$ over $M$ is diffeomorphic to the trivial one if and only if there is a global frame for $E$.
\end{corollary}

One can trace a result equivalent to Lemma \ref{canetavoadora} with respect to local frames using the last proposition. Suppose $E=\{E_p\}_{p\in M}$ is a discrete vector bundle over $M$ and consider for an open set $U\subset M$ the application $S:U\rightarrow E$. Then, if $\pi\circ S=\text{Id}$ one says that $S$ is a \textbf{discrete section} over $E$. One may also define discrete global and local frames the same way, but since $E$ is not necessarily endowed with a smooth structure, one may not evoke the smoothness condition. 

If $F=\{S_1,\ldots,S_k\}$ and $\tilde{F}=\{\tilde{S}_1,\ldots,\tilde{S}_k\}$ are both local frames for a vector bundle $E$ over $U$ and $\tilde{U}$ respectively then there are functions $a^i_j:U\cap\tilde{U}\rightarrow\re$ such that
\begin{equation}\tilde{S}_j(p)=\sum_{i=1}^ka^i_jS_i(p),\end{equation}
for every $p\in U\cap\tilde{U}$. One says that the frames $F$ and $\tilde{F}$ are \textbf{smoothly compatible} if each function $a^i_j$ is smooth over $U\cap \tilde{U}$.

\begin{proposition}
Let $E=\{E_p\}_{p\in M}$ be a discrete rank $k$ vector bundle over $M$. If there is an open cover $\{U_{\alpha}\}_{\alpha\in A}$ of $M$ such that
\begin{enumerate}
    \item for every $\alpha\in A$ there is a discrete local frame $F^{\alpha}=\{S^{\alpha}_1,\ldots,S^{\alpha}_k\}$ for $E$ over $U_{\alpha}$;
    \item for each $\alpha,\beta\in A$ the local frames $F^{\alpha}$ and $F^{\beta}$ are smoothly compatible,
\end{enumerate}
then $E$ is a vector bundle over $M$ with the unique topology and smooth structure given in Lemma \ref{canetavoadora}.
\end{proposition}

\begin{corollary}
Let $E$ and $F$ be vector bundles over $M$. The following spaces are vector bundles over a manifold $M$:
\begin{enumerate}
    \item The \textbf{tensor bundle} of $E$ and $F$ denoted by $E\otimes F$, with fibers $(E\otimes F)_p=E_p\otimes F_p$.
    \item The \textbf{rank} ${\bm k}$ \textbf{symmetric bundle} over $E$ denoted by $\text{Sym}^k(E)$, with fibers $(\text{Sym}^k(E))_p=\text{Sym}^k(E_p)$.
    \item The \textbf{rank} ${\bm k}$ \textbf{anti-symmetric bundle} over $E$ denoted by $\Lambda^k(E)$, with fibers $(\Lambda^k(E))_p=\Lambda^k(E_p)$.
    \item The \textbf{Hom-bundle} of $E$ and $F$, denoted $\text{Hom}(E,F)$, with fibers $(\text{Hom}(E,F))_p=\text{Hom}(E_p,F_p)$.\footnote{If $V$ and $W$ are vector spaces, $\text{Hom}(V,W)$ is the vector space of all linear applications from $V$ to $W$. Since $Hom(V,W)\simeq V^*\otimes W$, this example follows directly from the ones before.}
\end{enumerate}
\end{corollary}
These constructions can obviously be made over the tangent bundle as well. One may consider the $\binom{k}{l}$-\textbf{tensor bundle} over $TM$ denoted by $T^k_l(M)$, for which each fiber at $p\in M$ is given by $(T^k_l(M))_p=T^k_l(T_p M)$\footnote{If $V$ is a vector space, then $T^k_l(V)=V^{\otimes^l}\otimes(V^*)^{\otimes^k}$.}. Its sections are called \textbf{tensor fields} over $M$, which are $C^{\infty}$-multilinear maps 
    \begin{equation}F:\underbrace{\Omega^1(M)\times\cdots\times\Omega^1(M)}_{k\;\text{times}}\times\underbrace{\mathfrak{X}(M)\times\cdots\times\mathfrak{X}(M)}_{l\; \text{times}}\rightarrow C^{\infty}(M),\end{equation}
    and one writes 
    \begin{equation}
        F\in\mathcal{T}^k_l(M)=\Gamma(T^k_l(M)).
   \end{equation}
    Besides, if $(U;x^1,\ldots,x^n)$ are local coordinates around $p\in M$ then there are $n^{k+l}$ functions $F^{j_1\cdots j_k}_{i_1\cdots i_l}\in C^{\infty}(U)$ with indices taking values in $\{1,\ldots,n\}$ such that
    \begin{equation}F=F^{j_1\cdots j_k}_{i_1\cdots i_l}\partial_{j_1}\otimes\cdots\otimes\partial_{j_k}\otimes dx^{i_1}\otimes\cdot\otimes dx^{i_l}\end{equation}
    is the local description of the tensor field $F$.
    
    Moreover, the rank $k$ symmetric bundle over $TM$ is denoted by $Sym^k(T^*M)$ with fibers $Sym^k(T^*M)_p=Sym^k(T^*_pM)$. Its sections are the symmetric $k$-multilinear applications
   \begin{equation}
        S:\underbrace{\mathfrak{X}(M)\times\ldots\times\mathfrak{X}(M)}_{k\;\text{times}}\rightarrow C^{\infty}(M).
   \end{equation}
    The space of section for this bundle is denoted by
    \begin{equation}
        S\in\mathcal{S}^k(M)=\Gamma(Sym^k(T^*M)).
    \end{equation}
    In addition, in order to consider more general differential forms one has the rank $k$ anti-symmetric bundle $\Lambda^k(T^*M)$. Then, its sections are called ${\bm k}$\textbf{-forms} over $M$ and their space is denoted by
    \begin{equation}
        \Gamma(\Lambda^k(T^*M))=\Omega^k(M),
    \end{equation}
such that for each $\omega\in\Omega^k(M)$ one has the $C^{\infty}(M)$-linear alternating mapping
\begin{equation}
    \omega:\underbrace{\mathfrak{X}(M)\times\ldots\times\mathfrak{X}(M)}_{k\;\text{times}}\rightarrow C^{\infty}(M),
\end{equation}
and introducing $p\in M$ yields 
\begin{equation}
\omega_p:\underbrace{T_pM\times\cdots\times T_p M}_{k\;\text{times}} \rightarrow\re,
\end{equation}
an alternating $k$-multilinear application. Moreover, one can define the total space
\begin{equation}\Omega(M)=\bigoplus_{p\in\mathbb{N}}\Omega^p(M)\end{equation}
of differential forms over $M$. Such space is also an algebra with respect to the wedge product. If $\omega\in\Omega^p(M)$ and $\eta\in\Omega^q(M)$ then their wedge product $\omega\wedge\eta\in\Omega^{p+q}(M)$ is given by
\begin{equation}(\omega\wedge\eta)(X_1,\ldots,X_{p+q})=\frac{p!q!}{(p+q)!}\sum_{\sigma\in S_{p+q}}\varepsilon(\sigma)\omega(X_{\sigma(1)},\ldots,X_{\sigma(p)})\eta(X_{\sigma(p+1)},\ldots,X_{\sigma(p+q)}),\end{equation}
where $S_{p+q}$ is the set of permutations (bijections) of the set $\{1,\ldots,p+q\}$ to itself and $\varepsilon:S_{p+q}\rightarrow\{-1,1\}$ is the sign of $\sigma\in S_{p+q}$. If  $(U;x_1,\ldots,x_n)$ are local coordinates, then taking the dual coordinate basis $\{dx^1,\ldots,dx^n\}$ one has
\begin{equation}
\omega=\sum_{i_1,\ldots,i_p}\omega_{i_1,\ldots,i_p}dx^{i_1}\wedge\cdots\wedge dx^{i_p},
\end{equation}
where $\omega_{i_1,\ldots,i_p}\in C^{\infty}(U)$. 
\begin{definition}
Let $E$ be a vector bundle over $M$. The space of the ${\bm E}$\textbf{-valued} ${\bm p}$\textbf{-differential forms} over $M$ is defined as the set of smooth sections of the vector bundle $\Lambda^p(T^*M)\otimes E$. Namely,
\begin{equation}\Omega^p(M,E)=\Gamma(\Lambda^p(TM)\otimes E).\end{equation}
\end{definition}
\begin{preremark}\label{exartc}\upshape
If $E$ and $F$ are vector bundles over $M$, and $k>0$ is an integer, then one can see that there is a bijection between $\Gamma(\Lambda^k(E)\otimes F)$ and the set $X=\{f:\Gamma(E)\times\ldots\times\Gamma(E)\rightarrow\Gamma(F)\}$, where each $f$ is $C^{\infty}(M)$-multilinear alternating and with domain equal to $k$ copies of $\Gamma(E)$. Indeed, let $S\in\Gamma(\Lambda^k(E)\otimes F)$, $p\in M$ and $S(p)=\omega\otimes f$, with $\omega\in\Lambda^k(E)$ and $f\in F$. Then, one can define a map $S_*$ with
\begin{equation}(S_*(X_1,\ldots,X_k))_p=\omega_p(X_1(p),\ldots,X_k(p))f_p,\end{equation}
and analogously for the converse.
\end{preremark}
By the last Remark, it follows that an element $\omega\in\Omega^p(M,E)$ can be written as a $C^{\infty}(M)$-multilinear alternating application defined in terms of
\begin{equation}
    \omega:\underbrace{\mathfrak{X}(M)\times\ldots\times\mathfrak{X}(M)}_{k\;\text{times}}\rightarrow\Gamma(E).
\end{equation}
Locally, with respect to the same local coordinates $(U;x_1,\ldots,x_n)$ and local frame $\{e_1,\ldots,e_k\}$, there are smooth functions $f^i_{i_1,\ldots,i_p}$ over $U$ such that
\begin{equation}
\omega=\sum_{i_1,\ldots,i_p,i}f^i_{i_1,\ldots,i_p}dx^{i_1}\wedge\cdots\wedge dx^{i_p}\otimes e_{i}.
\end{equation}

\section{Connections}

As previously seen, smooth functions can be perceived as sections of the product bundle $M\times\re$, so that a choice of vector field $X\in\mathfrak{X}(M)$ and smooth section $f\in\Gamma(M\times\re)= C^{\infty}(M)$ yields a derivation $X(f)$ with respect to the direction of $X$. Now, let $E$ be a more general vector bundle over $M$. One may wonder which properties must an operator have so that it would be possible to make sense of the usual derivation for the sections of such bundle along the direction of vector fields $X\in\mathfrak{X}(M)$. As usual, if $f,g\in C^{\infty}(M)$ then since $X_p$ is a derivation for each $p\in M$ there holds
\begin{equation}
X(fg)(p)=X_p(f)g(p)+f(p)X_p(g),
\end{equation}
which is the well known Leibniz rule for derivations. It turns out that this property is fundamental for understanding derivatives of more general sections. 

\begin{definition}
A \textbf{connection} on a vector bundle $E$ over a manifold $M$ is a map
\begin{equation}
    \begin{split}
\nabla:\mathfrak{X}(M)\times\Gamma(E)&\rightarrow\Gamma(E)\\
(X,S)&\mapsto\nabla_XS,
\end{split}
\end{equation}
such that
\begin{equation}\nabla_{fX+Y}S=f\nabla_XS+\nabla_YS,\;\;\;\;\;\;\nabla_X(fS)=f\nabla_XS+X(f)S,\end{equation}
for every $f\in C^{\infty}(M)$, vector fields $X,Y\in\mathfrak{X}(M)$ and $S\in\Gamma(E)$. The section $\nabla_X(S)$ is denoted the \textbf{covariant derivative} of $S$ in the direction $X$.
\end{definition}

\begin{example}\upshape
Let $E=M\times \re^k$ be the trivial vector bundle of rank $k$. As seen before, a section $\tilde{S}:M\rightarrow E$ is given by
\begin{equation}
    \tilde{S}(p)=(p,S(p)),
\end{equation}
where $S:M\rightarrow \re^k$ has $S(p)=(S^1(p),\ldots,S^k(p))$ with each $S^i\in C^{\infty}(M)$. Then, one may define a connection $\nabla$ for $X\in\mathfrak{X}(M)$ by setting
\begin{equation}
    \nabla_X \tilde{S}(p)=\Big{(}p,\Big{(}X(S^i)(p),\ldots,X(S^k)(p)\Big{)}\Big{)},
\end{equation}
where $X(S^i)$ is the usual derivative in the direction of $X$. Such connection called the \textbf{trivial connection} over $E$.
\end{example}
\begin{preremark}\label{conecfor}\upshape
More generally, there is an one-to-one correspondence between connections over $E$ and 1-forms $k\times k$ matrices of the form $(\omega^i_j)$. In that manner, taking a local frame $\{e_1,\ldots,e_k\}$ for $E$ over $U$ one has
\bege\label{nablaforms}\nabla_Xe_j=\sum_{i=1}^{k}\omega^i_j(X)e_i.\enge
If a section is locally given by $S=S^je_j$ for functions $S^j:M\rightarrow\re$, then using the Leibniz rule comes
\begin{equation}\label{conexoxo}
    \begin{split}
    \nabla_XS=\nabla_X(S^je_j)&=S^j\nabla_Xe_j+X(S^i)e_i\\
    &=S^j\omega^i_j(X)e_i+X(S^i)e_i\\
    &=\big{(}S^j\omega^i_j(X)+X(S^i)\big{)}e_i.
    \end{split}
\end{equation}
Notice the trivial connection locally appears whenever $(\omega^i_j)=0$. 
\end{preremark}

\begin{example}\upshape
If $E$ and $E'$ are vector bundles over $M$ respectively endowed with connections $\nabla$ and $\nabla'$, then it is possible to consider a new connection over $E\oplus E'$. Namely, one can define the connection $\nabla\oplus \nabla'$ by setting
\begin{equation}(\nabla\oplus\nabla')_X(S,S')=(\nabla_XS,\nabla'_XS').\end{equation}
\end{example}

One could ask if it is always possible to define a connection over a vector bundle $E$. It turns out that the connection is actually a local operator, which can be glued together by the partition of unity usual construction in order to globally define it. 

\begin{proposition}\label{localconec}
If $\nabla$ is a connection on a vector bundle $E$ over $M$ and $X\in\mathfrak{X}(M)$, $S\in\Gamma(E)$ and $p\in M$, then $\nabla_XS(p)$ depends only on the values of $S$ on an arbitrarily small neighborhood of $p$ and of $X_p$. In other words, if $X_p=\tilde{X}_p$ and $S=\tilde{S}$ in a neighborhood of $p$ then \begin{equation}\nabla_XS(p)=\nabla_{\tilde{X}}\tilde{S}(p).\end{equation}
\end{proposition}
\begin{proof}
First, notice that by replacing $S$ for $S-\tilde{S}$ it suffices to shows that $\nabla_XS(p)=0$ whenever $S$ vanishes in a neighbourhood $U$ of $p$. In that case take a (bump) function $f\in C^{\infty}(M)$ with support inside $U$ such that $f(p)=1$. It follows that $fS=0$ identically on $U$ and therefore
\begin{equation}
    \nabla_X (fS)=0.
\end{equation}
Now, using the Leibniz rule there holds
\begin{equation}
    \nabla_X (fS)=f\nabla_XS+X(f)S=0,
\end{equation}
and by hypothesis the second term in the this equation is zero. Then, evaluating at $p$ yields
\begin{equation}
    (f\nabla_XS)(p)=f(p)\nabla_X S(p)=\nabla_X S(p)=0,
\end{equation}
as wanted.

Now, by the same reasoning one must only show that $\nabla_X S(p)=0$ whenever $X_p=0$. Since $\nabla_X S$ depends only locally on $S$, take local coordinates $(U;x^1,\ldots,x^n)$ and write $X=X^i\partial_i$ around $U$. It follows that
\begin{equation}
    \nabla_X S(p)=\nabla_{X^i\partial_i}S(p)=X^i(p)\nabla_{\partial_i}S(p)=0
\end{equation}
since $X^i(p)=0$ for each $i\in\{1,\ldots,n\}$.
\end{proof}
The case of most interest in this work is when $E=TM$, so that it may be valuable to explicitly name it:
\begin{definition}[Affinely connected spaces] Let $M$ be a manifold and consider its tangent bundle $TM$. A connection \begin{equation}\nabla:\mathfrak{X}(M)\times\mathfrak{X}(M)\rightarrow\mathfrak{X}(M)\end{equation}
over the tangent bundle it is called an \textbf{affine connection} and the pair $(M,\nabla)$ is said to be an \textbf{affinely connected space}.
\end{definition}
One may consider its local description: take a coordinate basis $\{\partial_1,\ldots,\partial_n\}$ for some neighborhood $U$ of $p\in M$. Then, there are $n^3$ functions $\Gamma^i_{jk}:M\rightarrow \re$, with $i,j,k\in\{1,\ldots,n\}$ such that
\begin{equation}\nabla_{\partial_j}(\partial_k)=\Gamma^i_{jk}\partial_i.\end{equation}
These functions are called the \textbf{Christoffel symbols} with respect to these local coordinates. From Remark \ref{conecfor} one has that $\omega^i_k(\partial_j)=\Gamma^i_{jk}$, which gives
\begin{proposition}
Let $(M,\nabla)$ be an affinely connected space and let $X,Y\in\mathfrak{X}(M)$. If in coordinate basis one has $X=X^i\partial_i$ and $Y=Y^j\partial_j$ then
\begin{equation}
\nabla_X(Y)=(X(Y^k)+X^iY^j\Gamma^k_{ij})\partial_k.
\end{equation}
\end{proposition}
\begin{preremark}\upshape 
A connection $\nabla$ over the tangent bundle $TM$ is the trivial one if and only if its Christoffel symbols vanish. In the context of affine connections, the affinely connected space $(M,\nabla)$ is called a \textbf{flat space}.
\end{preremark}

Given an affinely connected space $(M,\nabla)$, one is able to extend the connection to the space of tensor fields over $M$ in such a way that some useful properties come in hand.
\begin{proposition}
Let $\nabla$ be an affine connection over $M$. Then, $\nabla$ can be uniquely extended to the $(k,l)$-tensor bundle $T^k_l(M)$ over $M$ in such a way that
\begin{enumerate}
    \item In $T^0(M)=C^{\infty}(M)$ one has $\nabla_Xf=X(f)$, the usual differentiation for functions.
    \item There holds 
    \begin{equation}\nabla_X(T\otimes S)=(\nabla_XT)\otimes S+T\otimes(\nabla_XS).\end{equation}
    \item If $F\in T^k_l(M)$, $Y_i\in\mathfrak{X}(M)$, $\omega^j\in\Omega^1(M)$, where $1\leq i\leq k$ and $1\leq j \leq l$, there holds
    \begin{equation}
    \begin{split}
        (\nabla_X F)(\omega^1,\ldots&,\omega^l,Y_1,\ldots,Y_k)=X(F(\omega^1,\ldots,\omega^l,Y_1,\ldots,Y_k))\\
        &-\sum_{j=1}^lF(\omega^1.\ldots,\nabla_X\omega^j,\ldots,\omega^l,Y_1,\ldots,Y_k)\\
        &-\sum_{i=1}^kF(\omega^1,\ldots,\omega^l,Y_1,\ldots,\nabla_XY_i,\ldots,Y_k).
        \end{split}
        \end{equation}
\end{enumerate}
\end{proposition}
In particular, take a $1$-form $\omega\in\Omega^1(M)$ and $X\in\mathfrak{X}(M)$. Considering local coordinate $(U;x^1,\ldots,x^n)$ one may write
\begin{equation}
    \begin{split}
        \nabla_X\omega(\partial_k)&=\nabla_{X^i\partial_i}(\omega_jdx^j)(\partial_k)\\
        &=X^i(\nabla_{\partial_i}(\omega^jdx^j)(\partial_k))\\
        &=X^i(\partial_i(\omega_jdx^j(\partial_k))-(\omega_jdx^j(\nabla_{\partial_i}\partial_k)))\\
        &=X^i(\partial_i(\omega_k)-\omega_j\Gamma^j_{ik}),
    \end{split}
\end{equation}
which shows that in local coordinates
\begin{equation}
    \nabla_X\omega=(X^i\partial_i(\omega_k)-X^i\omega_j\Gamma^j_{ik})dx^i.
\end{equation}
Since $\nabla_X F$ is $C^{\infty}(M)$-linear over $X$, one can construct another tensor field, namely
\begin{definition}
Let $(M,\nabla)$ be an affinely connected space and let $F\in T^k_l(M)$. The $\binom{k+1}{l}$-tensor $\nabla F:\Omega^1(M)\times\cdots\times\Omega^1(M)\times\mathfrak{X}(M)\times\cdots\times\mathfrak{X}(M)\rightarrow C^{\infty}(M)$, given by
\begin{equation}\nabla F(\omega^1,\ldots,\omega^l,Y_1,\ldots,Y_k,X)=\nabla_X F(\omega^1,\ldots,\omega^l,Y_1,\ldots,Y_k)\end{equation}
is called the \textbf{total covariant derivative} for $F$.
\end{definition}

Let $(U;x^1,\ldots,x^n)$ be local coordinates around $p\in M$ and take a tensor field $F\in T^k_l(M)$ which around $U$ as seen have the local description
\begin{equation}F=F^{j_1\cdots j_k}_{i_1\cdots i_l}\partial_{j_1}\otimes\cdots\otimes\partial_{j_k}\otimes dx^{i_1}\otimes\cdot\otimes dx^{i_l}.\end{equation}
Then, the $m$-direction derivative of the coordinate functions of $F$ shall be denoted by
\begin{equation}F^{j_1\cdots j_k}_{i_1\cdots i_l,m}=\partial_m (F^{j_1\cdots j_k}_{i_1\cdots i_l}).\end{equation}
Moreover, the components of the total covariant field $\nabla F$ may be written as
\begin{equation}\nabla F=\nabla_mF^{j_1\cdots j_k}_{i_1\cdots i_l}\partial_{j_1}\otimes\cdots\otimes\partial_{j_k}\otimes dx^{i_1}\otimes\cdot\otimes dx^{i_l}\otimes dx^m.\end{equation}
One can then consider a formula for the components of the total covariant derivative of arbitrary tensor fields, which is given by direct computation:
\begin{proposition}
Let $(M,\nabla)$ be an affinely connected space. Then, the components of a $\binom{k}{l}$-tensor field $F$ with respect to a coordinate system $(U;x^1,\ldots,x^n)$ is given by
\begin{equation}\nabla_mF^{j_1\cdots j_l}_{i_1\cdots i_k}= F^{j_1\cdots j_l}_{i_1\cdots i_k,m}+\sum_{s=1}^l F^{j_1\cdots p \cdots j_l}_{i_1\cdots i_k}\Gamma^{j_s}_{mp}-\sum_{s=1}^k F^{j_1\cdots j_l}_{i_1\cdots p \cdots i_k}\Gamma^p_{m i_s}.\end{equation}
\end{proposition}

\section{Parallel Transport}

Since a connection on a vector bundle $E$ over a manifold $M$ is a way of differentiating vector fields and may even be extended to more general tensor fields, as seen before, one may be tempted to see if it is possible to describe it in any similar way to what is already known from derivatives of functions

In differential calculus one analyzes real functions $f:I\subset\re\rightarrow\re$ defined in an interval $I$. A special class of functions emerges when the so-called difference quotient of $f$ is considered: if $t,t_0\in I$, then the derivative of $f$ at $t_0$ is given by
\begin{equation}\label{newton}
    f'(t_0)=\lim_{t\rightarrow t_0}\frac{f(t)-f(t_0)}{t-t_0},
\end{equation}
whenever the right-handed side exists. In that case, the real function $f$ is called differentiable. Notice that to make sense of this expression, the difference between the real numbers $f(t)$ and $f(t_0)$ must be defined. Since there is a natural way to take one (namely the difference induced by the field structure of the real numbers), one gets eqn (\ref{newton}).

However, if one was to do the same with sections $S\in\Gamma(E)$ of a vector bundle $E$, then this would be impossible at first. Indeed, take two points $p,p_0\in M$ both in an arbitrarily small open set. Even so, the expression
\begin{equation}
    S(p)-S(p_0)
\end{equation}
makes no sense, since $S(p)\in E_{p}$ and $S(p_0)\in E_{p_0}$, which are intrinsically different vector spaces. In order to \textit{connect} them, a new structure is then required. This structure is called the parallel transport and one may see that there is an one-to-one correspondence between such structures and connections as follows.

Fix a vector bundle $E$ over $M$ from now on and let $\gamma:I\rightarrow M$ be a curve over $M$. The \textbf{velocity} of $\gamma$ at a point $t_0\in I$ is given by the push-forward 
\begin{equation}\gamma'(t_0)=\gamma_*\Big{(}\frac{d}{dt}\Bigg{|}_{t_0}\Big{)},\end{equation}
which acts as a derivation on a function $f:M\rightarrow \re$ as
\begin{equation}\gamma'(t_0)(f)=\gamma_*\Big{(}\frac{d}{dt}\Bigg{|}_{t_0}\Big{)}(f)=\frac{d(f\circ \gamma)}{dt}\Bigg{|}_{t_0}.\end{equation}
Also, given local coordinates $(U;x^1,\ldots,x^n)$ one may write the coordinate representation of $\gamma$ as \begin{equation}\gamma(t)=(\gamma^1(t),\ldots,\gamma^n(t)),\end{equation}
or even sometimes $\gamma(t)=(\gamma^i(t))$. Its velocity in the coordinate basis $\{\partial_1,\ldots,\partial_n\}$ at $t_0\in I$ is then given by
\begin{equation}\gamma'(t_0)=\dot{\gamma}^i(t_0)\partial_i.\end{equation}

\begin{definition}
Let $\gamma:I\rightarrow M$ be a curve over $M$. Then, a curve $V:I\rightarrow E$ on a vector bundle $E$ over $M$ is said to be a \textbf{section along} $\bm{\gamma}$ if there holds
\begin{equation}V(t)\in E_{\gamma(t)},\;\forall t\in I.\end{equation}
Moreover, the set of all such section is denoted by $\Gamma(E,\gamma)$.
\end{definition}
\begin{preremark}\upshape
In general, given a section $S\in\Gamma(E)$ and a curve $\gamma:I\rightarrow M$, one can produce such sections by composition, namely  
\begin{equation} S\circ\gamma:I\rightarrow E.\end{equation} 
\end{preremark}
\begin{definition}
Let $\gamma:I\rightarrow M$ be a curve over $M$ and suppose $V:I\rightarrow E$ is a section along $\gamma$. Then, if there is $S\in\Gamma(E)$ such that
\begin{equation}V(t)=S(\gamma(t)),\end{equation}
then $V$ is said to be \textbf{extendible} over $E$ and $S$ is called an \textbf{extension} of $V$.
\end{definition}
\begin{preremark}\upshape
Notice that not every section $V$ along $\gamma$ needs to be extendible: if $\gamma$ has $t_0,t_1\in I$ with $\gamma(t_0)=\gamma(t_1)$ but such that $V(t_0)\neq V(t_1)$, then $V$ is not extendible.
\end{preremark}

\begin{lemma}
Let $E$ be a vector bundle over $M$ and consider a connection $\nabla$ over $E$. Then, for each curve $\gamma:I\rightarrow M$ there is an unique operator
\begin{equation}D_t:\Gamma(E,\gamma)\rightarrow\Gamma(E,\gamma),\end{equation}
with the property that if $V,W\in \Gamma(E,\gamma)$ then:
\begin{enumerate}
    \item $D_t$ is linear over $\re$, that is
    \begin{equation}D_t(\lambda V+\mu W)=\lambda D_t V+\mu D_t W,\;\;\;\forall\lambda,\mu\in\re.\end{equation}
    \item Let $f\in C^{\infty}(I)$. Then, a Leibniz rule holds:
    \begin{equation}D_t( f V)=f D_t V+  \dot{f}V\end{equation}
    \item If $V$ is extendible then, for every extension $S:I\rightarrow E$ of $V$,
    \begin{equation}D_t V=\nabla_{\gamma'}S.\end{equation}
\end{enumerate}
The operator $D_t V$ is called the \textbf{covariant derivative of} $\bm{V}$ \textbf{along} $\bm{\gamma}$.
\end{lemma}
\begin{proof}
To show uniqueness, suppose there is such operator $D_t$ and fix $t_0\in I$. Proposition \ref{localconec} shows that the value of $D_t V$ depends only on a neighbourhood near $t_0$. Therefore, one may proceed locally: let $(U;x^1,\ldots,x^n)$ be local coordinates around $\gamma(t_0)$ and let $\{\partial_1,\ldots,\partial_n\}$ be the coordinate basis around this system. Then,
\begin{equation}V(t)=V^i(t)\partial_i{\Big |}_{\gamma(t)},\end{equation}
which shows that $V$ is extendible around $U$. Then,
\begin{equation}\label{covariantline}
    \begin{split}
        D_t V(t_0)&=\dot{V}^k(t_0)\partial_k+V^j(t_0)\nabla_{\gamma'(t_0)}\partial_j\\
        &=(\dot{V}^k(t_0)+V^j(t_0)\dot{\gamma}^i(t_0)\omega^k_j(\partial_i))\partial_k,
    \end{split}
\end{equation}
where $\omega^k_j\in\Omega^1(M)$ are the connection $1$-forms. Therefore, $D_t$ is locally unique. Since $\gamma(I)$ is a compact subset of $M$, one may realize the same procedure to a finite number of neighborhoods, and since these relations must agree on overlaps, it follows that $D_t$ is, in fact, unique all over $\gamma$, if it exists. Now, just take eqn (\ref{covariantline}) as the definition of $D_t V$ and it follows that it automatically satisfies all wanted relations.
\end{proof}

\begin{definition}
Let $E$ be a vector bundle over $M$ and $\nabla$ a connection on $E$. If $V:I\rightarrow E$ is a section along a curve $\gamma:I\rightarrow M$ such that
\begin{equation}D_t V=0,\end{equation}
then $V$ is said to be \textbf{parallel with respect to} $\bm{\gamma}$.
\end{definition}

\begin{lemma}
Let $E$ be a vector bundle over $M$ and $\nabla$ a connection on $E$. Consider a curve $\gamma:I\rightarrow M$ and $t_0\in I$. Then, for any fixed $V_0\in E_{\gamma(t_0)}$ there is an unique parallel path $V:I\rightarrow E$ above $\gamma$ such that $V(t_0)=V_0$.
\end{lemma}
\begin{proof}
As discussed before, one may proceed locally: aiming to find a section $V:I\rightarrow E$ along $\gamma$ satisfying the definition of parallel path and eqn (\ref{covariantline}), there follows
\begin{equation}\dot{V}^i(t)=-\sum_{j=1}^kV^j(t)\omega^i_j(\gamma'(t)),\;\;\;\;V(t_0)=V_0.\end{equation}
Defining the matrix $A(t)=-\omega(\gamma'(t))$, where now $\omega=(\omega^i_j)$ is the connection $1$-form matrix, then the last equation can be translated to
\begin{equation}V'(t)=A(t)V(t),\;\;\;\;V(t_0)=V_0,\end{equation}
and its existence and uniqueness is a direct result of ODE\footnote{Ordinary Differential Equations.} theory.
\end{proof}

\begin{definition}
Let $E$ be a vector bundle over $M$, $\nabla$ a connection over $E$, $\gamma:I\rightarrow M$ a curve and $t_0,t_1\in I$. Then, the \textbf{parallel transport} along $\gamma$ (with respect to $\nabla$) from $t_0$ to $t_1$ is the map
\begin{equation}T_{\gamma}^{t_0,t_1}:E_{\gamma(t_0)}\rightarrow E_{\gamma(t_1)},\end{equation}
which sends $V_0\in E_{\gamma(t_0)}$ to the unique vector $V(t_1)\in E_{\gamma(t_1)}$ such that $V(t)$ is the parallel curve above $\gamma$ with $V(t_0)=V_0$.
\end{definition}

\begin{preremark}\upshape 
From the uniqueness property, and the linear dependence of the ODE on its initial condition, it is straightforward to see that $T_{\gamma}^{t_0,t_1}$ is a well-defined linear transformation. Moreover, it follows from uniqueness that 
\begin{equation}T_{\gamma}^{t_1,t_2}\circ T_{\gamma}^{t_0,t_1}=T_{\gamma}^{t_0,t_2},\end{equation}
\end{preremark}
and taking $t_0=t_2$ it follows that $T$ is an isomorphism of vector spaces. It so happens that the parallel transport completely defines the connection and vice-versa, as one can see in the following result. The parallel transport manages to \textit{connect} the different fibers around a neighborhood of a point $p\in M$, so that the usual (calculus) notion of derivatives can be perceived, as previously stated.
\begin{theorem}
Let $\gamma:I\rightarrow M$ be a smooth curve over $M$ with $\gamma(t_0)=p$ and $\gamma'(t_0)=X_0\in T_pM$. Then, it follows that for every section $S\in\Gamma(E)$ there holds
\begin{equation}\nabla_{X_0}S(p)=\lim_{t\to t_0}=\frac{\lp T_{\gamma}^{t_0,t}\rp^{-1}\lp S\lp\gamma(t)\rp-S(\gamma(t_0)\rp}{t-t_0}.\end{equation}
\end{theorem}
\begin{proof}
Let $\{e_1,\ldots,e_k\}$ be a basis for $E_p$ and define $e_i(t)=T_{\gamma}^{t_0,t}(e_i)$. Since $T_{\gamma}^{t_0,t}$ is an isomorphism, the set $\{e_1(t),\ldots,e_k(t)\}$ is a basis for $E_{\gamma(t)}$. Therefore, there are smooth functions $a^i:U\rightarrow \re$ in some neighborhood of $p$ such that $S(\gamma(t))=a^i(t)e_i(t)$. Since the parallel transport is linear, it follows that 
\begin{equation}\lp T_{\gamma}^{t_0,t}\rp^{-1}\lp S\lp\gamma(t)\rp\rp=a^i(t)e_i.\end{equation}
Then,
\begin{equation}\lim_{t\to t_0}=\frac{\lp T_{\gamma}^{t_0,t}\rp^{-1}\lp S\lp\gamma(t)\rp-S(\gamma(t_0)\rp}{t-t_0}=\lim_{t\to t_0}\frac{a^i(t)e^i-a^i(t_0)e^i}{t-t_0}=\dot{a}^i(t_0)e_i.\end{equation}
On the other hand, the definition of parallel transport yields $\nabla_{\gamma'(t)}(e_i(t))=0$. Since $\gamma'(t_0)=X_0$, one may calculate the expression $\nabla_{\gamma'(t)}S=D_t V$ and then set $t=t_0$. It follows that
\begin{equation}\nabla_{\gamma'(t)}(S)=\dot{a}^i(t)e_i+a^i(t)\nabla_{\gamma'(t)}(e_i(t))=\dot{a}^i(t)e_i,\end{equation}
as wanted.
\end{proof}

\chapter{Riemannian Manifolds}
In this chapter some elementary results on Riemannian manifolds are presented. A vector bundle $E$ over a manifold $M$ may be endowed with a smooth parameterized choice of metrics on the fibers $E_p$, yielding the concept of a metric over the manifold $M$. This more general case may be considered in detail but the case $E=TM$ is focused herein, as follows.

\section{Riemannian Metrics}
Generally, metrics over vector spaces $V$ are functions which allow one to define the notions of sizes and angles of vectors, as well as the concept of orthogonality. One may consider the basic example of $V=\re^n$ endowed with the canonical Euclidean metric $\langle\cdot,\cdot\rangle$ given by
\begin{equation}
    \langle u,v\rangle=\sum_{i}^n u^iv^i,
\end{equation}
where $u=(u^1,\ldots,u^n)$ and $v=(v^1,\ldots,v^n)$ are vector in $\re^n$. Further on, the Euclidean norm may be defined by
\begin{equation}
    \Vert u\Vert=\sqrt{\langle u,u\rangle},
\end{equation}
 which precisely measures the Euclidean size of a vector $u\in\re^n$. In addition, the angle $\theta$ between two vectors $u,v\in\re^n$ can be seen to be given by
\begin{equation}
    \theta=\arccos\frac{\langle u,v\rangle}{\Vert u\Vert\Vert v\Vert},
\end{equation}
in such a way that this notion may be generalized as follows.
\begin{definition}\label{metric} A \textbf{metric} over a vector space $V$ is a function
\begin{equation}g:V\times V\rightarrow \re,\end{equation}
satisfying the following properties:
\begin{enumerate}
    \item Symmetry: $g(u,v)=g(v,u)$ for every $u,v\in V$.
    \item Bilinearity: $g(u,\lambda v+w)=\lambda g(u,v)+g(u,w)$ for every $u,v,w\in V$ and $\lambda\in \re$.
    \item Positive-definiteness: $g(u,u)\geq 0$ for every $u\in V$; $g(u,u)=0$ if and only if $u=0$.
\end{enumerate}
One says that the pair $(V,g)$ is a vector space endowed with a metric.
\end{definition}

It is straightforward to see that the Euclidean metric satisfies the above conditions. Moreover, in the context of manifolds $M$ one may endow the tangent bundle $TM$ with a metric in such a way enabling the possibility of measuring these quantities in each tangent space $T_pM$. A choice of metric $g_p$ for each tangent space produces a collection $\{g_p\}_{p\in M}$. However, as one might expect, such indexation shall be required to be smooth.   
\begin{definition}
Let $M$ be a manifold and consider its tangent bundle $TM$. A \textbf{metric} over $M$ is a family $\{g_p\}_{p\in M}$ of metrics $g_p:T_pM\times T_pM\rightarrow \re$, which vary smoothly with $p\in M$. Namely, if $X,Y\in\mathfrak{X}(M)$, then the function $g(X,Y):M\rightarrow \re$ defined by
\begin{equation} g(X,Y)(p)=g_p(X_p,Y_p),\end{equation}
is smooth. In such case, the pair $(M,g)$ is called an \textbf{Riemannian manifold}.
\end{definition}
As in the case of affine connections over $TM$, one may see that every manifold $M$ admits a metric by the usual partition of unity argument. In what follows some of the properties and classical constructions related to Riemannian manifold $(M,g)$ are presented.
\begin{preremark}\upshape 
An important notion in geometry is that of symmetries. Namely, whenever one has two vector spaces endowed with metrics $(V,g)$ and $(\tilde{V},\tilde{g})$, one may look for an isomorphism $T:V\rightarrow \tilde{V}$ preserving this structure in such a way that for every $u,v\in V$ there holds
\begin{equation}
    \tilde{g}(T(u),T(v))=g(u,v).
\end{equation}
An analogous notion can be considered for Riemannian spaces. Namely, given $(M,g)$ and $(\tilde{M},\tilde{g})$ two Riemannian spaces, an \textbf{isometry} between them is a diffeomorphism $F:M\rightarrow\tilde{M}$ such that for every $p\in M$ and $X_p,Y_p\in T_pM$ there holds
\begin{equation}
    \tilde{g}(dF_p(X_p),dF_p(Y_p))=g(X_p,Y_p).
\end{equation}
\end{preremark}

One may now look at some elementary properties and constructions that can be considered over a Riemannian manifold $(M,g)$. Let $(U;x^1,\ldots,x^n)$ be local coordinates and let $\{\partial_1,\ldots,\partial_n\}$ be the associated coordinate basis. Then, note that $g\in \mathcal{S}^2(M)=\Gamma(Sym^2(TM))$. Therefore, one may write
\begin{equation}g=g_{ij}dx^i\otimes dx^j,\end{equation}
where in this case it follows that 
\begin{equation}
g_{ij}=g_{ji},
\end{equation} 
since $g$ is symmetric. In addition, since $g$ is positive-definite, it follows that the matrix $(g_{ij})$ defines an isomorphism. Therefore, there exists an inverse matrix which shall be denoted by $(g^{ij})$. Then, it follows that
\begin{equation}
    g_{ik}g^{kj}=g^{ik}g_{kj}=\delta^i_j=\begin{cases}
            1, &         \text{if } i=j,\\
            0, &         \text{if } i\neq j.
    \end{cases}
\end{equation}
One elementary but interesting property of metrics over manifolds is that they allow one to convert vectors to covectors and the opposite. Namely, given a vector field $X\in\mathfrak{X}(M)$, one can define the $1$-form $X^{\flat}$ by the relation
\begin{equation}
    X^{\flat}(Y)=g(X,Y),
\end{equation}
for every vector field $Y\in\mathfrak{X}(M)$. Then, $(\cdot)^{\flat}$ is called the \textbf{flat} isomorphism and in local coordinates one has
\begin{equation}
    X^{\flat}=g(X^i\partial_i,\cdot)=g_{ij}X^idx^j.
\end{equation}
Putting it in coordinates with $X^{\flat}=X_jdx^j$ it reads
\begin{equation}
    X_j=g_{ij}X^i.
\end{equation}
In the same manner, if $\omega$ is an $1$-form then the vector field $\omega^{\sharp}$ may be defined in terms of the inverse metric $g^{ij}$ by setting $\omega^{\sharp}=\omega^i\partial_i$ and taking
\begin{equation}
    \omega^i=g^{ij}\omega_j.
\end{equation}
The map $(\cdot)^{\sharp}$ is called the \textbf{sharp} isomorphism and together with the flat one these are called the \textbf{musical isomorphisms} with respect to $g$. It is also clear that one is the inverse of the other. Another important notion when working with metric is that of orthogonality.
\begin{definition}
Let $(M,g)$ be a Riemannian manifold and fix $p\in M$. Then, two vectors $X_p,Y_p\in T_pM$ are said to be \textbf{orthogonal} if
\begin{equation}
    g(X_p,Y_p)=0.
\end{equation}
Moreover, if a collection of vectors $X^1_p,\ldots,X^k_p\in T_pM$ are pair-wise orthogonal and \textbf{unitary}, that is
\begin{equation}
    g(X^i_p,X^i_p)=1,\;\;\;\;\forall i\in\{1,\ldots,k\},
\end{equation}
then the set $\{X^i_p\}_{i=1}^k$ is said to be \textbf{orthonormal}.
\end{definition}
\begin{preremark}\upshape
Let $(M,g)$ be a Riemannian space and suppose that $\{e_1,\ldots,e_n\}$ is a local frame for $p\in M$ in some open neighbourhood $U$. Then, if for some $q\in U$ the basis $\{e_1(q),\ldots,e_n(q)\}$ for $T_q M$ is orthonormal then it is called an \textbf{orthonormal basis}. In addition, if this property holds for every point in $U$ then $\{e_1,\ldots,e_n\}$ is called an \textbf{local orthonormal frame}. By the usual Gram-Schmidt orthonormalizing process it is straight-forward to prove the
\end{preremark}
\begin{proposition}
Let $(M,g)$ be a Riemannian manifold. Then, for every $p\in M$ there is a local orthonormal frame $\{E_1,\ldots,E_n\}$ over a neighbourhood $U$ of $p$.
\end{proposition}
Now, just like connections it is possible to extend the metric $g$ for all of the tensor bundle $T^k_l(M)$. Of course, a metric for a more general vector bundle $E$ over $M$ is given by a smooth parameterization of metrics $g_p:E_p\times E_p\rightarrow \re$ over the fibers satisfying the properties given in Definition \ref{metric}.
\begin{proposition}\label{tensormet}
Let $(M,g)$ be a Riemannian manifold. Then, one may uniquely extend the metric $g$ to the tensor bundle $T^k_l(M)$ with the property that if $\{E_1,\ldots,E_n\}$ is an orthonormal basis for $T_p M$ with $\{E^1,\ldots,E^n\}$ its dual basis, then the usual tensor basis for $T^k_l(T_pM)$ associated with them is also orthonormal.
\end{proposition}
\begin{proof}
Let $R_p,S_p\in T^k_l(T_pM)$. Then, taking local coordinates $(U;x^1,\ldots,x^n)$ one may see by direct computation that such metric must be given by
\begin{equation}
    g\big{(}R,S\big{)}=g^{i_1r_1}\cdots g^{i_kr_k}g_{j_1s_1}\cdots g_{j_ls_l}R^{j_1\cdots j_l}_{i_1\cdots i_k}S^{s_1\cdots s_l}_{r_1\cdots r_k}.
\end{equation}
\end{proof}
\begin{preremark}\upshape In particular, notice that for $\alpha,\omega\in\Omega^1(M)$ there holds in local coordinates
\begin{equation}
    g(\alpha,\omega)=g^{ij}\alpha_i\omega_j,
\end{equation}
and on the other hand
\begin{equation}
    g(\alpha^{\sharp},\omega^{\sharp})=g\big{(}g^{ij}\omega_j\partial_i,g^{lm}\alpha_m\partial_l\big{)}=g^{ij}g^{lm}\omega_j\alpha_ig_{il}=g^{ij}\alpha_i\omega_j,
\end{equation}
so that the musical isomorphisms preserve the metric defined by means of Proposition \ref{tensormet}. Such metric can be naturally extended to the space of $k$-forms, where for homogeneous $\alpha=\alpha_1\wedge\cdots\wedge\alpha_k$ and $\beta=\beta_1\wedge\cdots\wedge\beta_k$ the metric is given by
\begin{equation}
g(\alpha,\beta)=\det\big{(}g(\alpha_i,\beta_j)\big{)}.    
\end{equation}
\end{preremark}
\begin{preremark}\upshape
Let $F\in\mathcal{T}^k_l(M)$ be a $\binom{k}{l}$-tensor field. Locally one has
\begin{equation} F=F^{j_1\cdots j_k}_{i_1\cdots i_l}\partial_{j_1}\otimes\cdots\otimes\partial_{j_k}\otimes dx^{i_1}\otimes\cdot\otimes dx^{i_l}\end{equation}
so that one can define a $\binom{k-1}{l+1}$ tensor called the \textbf{lowering} of an upper index as follows: choose one of the contravariant entries of $F$, say $1\leq m\leq k$. Then, with respect to those local coordinates one may define the symbols
\begin{equation}
    F^{j_1\ldots j_{m-1}j_{m+1}\ldots j_k}_{j_m i_1\ldots i_l}=F^{j_1\cdots j_{m-1}nj_{m+1}\ldots j_k}_{i_1\cdots i_l}g_{nj_m}
\end{equation}
where the dummy index $n$ sums over all values $\{1,\ldots,k\}$. Analogously it is possible to define the \textbf{raising} of a lower index of the tensor $F$ by again choosing $1\leq m\leq l$ and setting
\begin{equation}
    F^{j_1\cdots j_ki_m}_{i_1\cdots i_{m-1}i_{m+1}\cdots i_l}=F^{j_1\cdots j_k}_{i_1\cdots i_{m-1}ni_{m+1}\cdots i_l}g^{ni_m}.
\end{equation}
These operations shall be used in several occasions to come.
\end{preremark}

One may consider the case when $M$ is oriented and perceive the relations with the Riemannian metric. 
\begin{lemma}
Let $(M,g)$ be a Riemannian manifold with $M$ oriented. Then, there is an unique $n$-form $\emph{vol}_g$ such that for every oriented orthonormal basis $\{E_1(p),\ldots,E_n(p)\}$ for $T_p M$ there holds
\begin{equation}\label{orienta}
    \emph{vol}_g(E_1,\ldots,E_n)=1,
\end{equation}
which is called the \textbf{Riemannian volume element} for $(M,g)$.
\end{lemma}
\begin{proof}
Fix $p\in M$ and let $(U;x^1,\ldots,x^n)$ be local coordinates around it. Then, take a (positively) oriented orthonormal basis $\{E_1,\ldots,E_n\}$ for $T_pM$ and consider its dual basis $\{E^1,\ldots,E^n\}$. Define
\begin{equation}
    \text{vol}_g=E^1\wedge\cdots\wedge E^n.
\end{equation}
Obviously this $n$-form satisfies eqn (\ref{orienta}) for the chosen orthonormal basis. In addition, if $\{\tilde{E}^1,\ldots,\tilde{E}^n\}$ is another oriented orthonormal basis for the dual space $T^*_p M$ then, denoting by $T$ the transition matrix between such basis, there follows
\begin{equation}
    E^1\wedge\cdots\wedge E^n=(\det T)\tilde{E}^1\wedge\cdots\wedge \tilde{E}^n.
\end{equation}
However, since both of them and orthonormal and positively oriented one has $\det T=1$, which gives the desired result.
\end{proof}
\begin{preremark}\upshape
Let $\text{vol}_g$ be the Riemannian volume form for $(M,g)$, written in terms of an orthonormal basis as
\begin{equation}
    \text{vol}_g=E^1\wedge\cdots\wedge E^n.
\end{equation}
Then, taking the local coordinates $(U;x^1,\ldots,x^n)$ one may denote by $A$ the matrix with $\partial_i$ in its $i$-th column with respect to this orthonormal basis. It follows that
\begin{equation}
    \Vert dx^1\wedge\cdots\wedge dx^n\Vert=|\det A|=\sqrt{\det(A^T A)}.
\end{equation}
Since by construction $A^T A=(g_{ij})$, one has the relation
\begin{equation}
    \text{vol}_g=\sqrt{\det(g_{ij})}dx^i\wedge\cdots\wedge dx^n.
\end{equation}
\end{preremark}
\begin{preremark}\upshape
Given an $n$-dimensional Riemannian manifold $(M,g)$ with $M$ oriented, one may consider the \textbf{Hodge star} operation 
\begin{equation}\star:\Omega^k(M)\rightarrow\Omega^{n-k}(M),\end{equation}
which is defined by the relation
\begin{equation}
    \langle\omega,\alpha\rangle\text{vol}_g=\omega\wedge\star\alpha=\alpha\wedge\star\omega,
\end{equation}
where $\omega\in\Omega^k(M)$ and $\alpha\in\Omega^{n-k}(M)$. There holds 
\begin{equation}
    \star^2=(-1)^{k(n-k)}.
\end{equation}
In addition, consider the \textbf{interior product} \bege\iprod:\mathfrak{X}(M)\times\Omega^k(M)\rightarrow \Omega^{k-1}(M),\enge which for $X,X_1,\ldots,X_{k-1}\in\mathfrak{X}(M)$ and $\omega\in\Omega^k(M)$ is given by
\begin{equation}
    \lp X\iprod\omega\rp(X_1,\ldots,X_{k-1})=\omega(X,X_1,\ldots,X_{k-1}).
\end{equation}
One can then prove the important relation
\begin{lemma}\label{lemmaimport}
Let $(M,g)$ be a Riemannian manifold with $M$ oriented, $\omega\in\Omega^k(M)$ and $X$ a vector field. There holds
\begin{equation}
\star(X\iprod\omega)=(-1)^{k+1}(X^{\flat}\wedge\star\omega)\label{tudoa01}.
\end{equation}
In particular, when $\omega=\emph{\text{vol}}_g$ one has
\begin{equation}
    X\iprod\emph{\text{vol}}_g=\star X^{b}.
\end{equation}
\end{lemma}
\begin{proof}
Let $\alpha\in\Omega^{k-1}(M)$. Then, there follows
\begin{equation}
    \begin{split}
        \alpha\wedge\star(X\iprod\omega)&=\langle\alpha,X\iprod\omega\rangle\text{vol}_g\\
        &=(X\iprod\omega)(\alpha^{\sharp})\text{vol}_g\\
        &=\omega(X\wedge\alpha^{\sharp})\text{vol}_g\\
        &=\langle\omega,X^{\flat}\wedge\alpha\rangle\text{vol}_g\\
        &=(X^{\flat}\wedge\alpha)\wedge\star\omega\\
        &=(-1)^{k-1}\alpha\wedge(X^{\flat}\wedge\star\omega).
    \end{split}
\end{equation}

\end{proof}
\end{preremark}

\section{Geodesics and Normal Coordinates}
The important notion of geodesic curves may now be presented, being those the curves with no acceleration over an affinely connected space $(M,\nabla)$. With the aid of the covariant derivative along curves developed in the last chapter, one may be able to define what acceleration means in a rigorous way in this case. Then using these entities and a metric over $M$ one may be able to present normal coordinates around a point.

\begin{definition}
Let $(M,\nabla)$ be an affinely connected space and $\gamma:I\rightarrow M$ a curve over $M$. Its \textbf{acceleration} is defined as the vector field $D_t\gamma'$ along $\gamma$.
\end{definition}
\begin{example}\upshape
It is possible to see that this definition is compatible with the usual notion of acceleration on for curves in $\re^n$. Let $\gamma:I\rightarrow \re^n$ be a curve and assume $\nabla$ to be the flat affine connection on $E=T(\re^n)\simeq\re^n$. Then, taking the canonical (global) frame $\{\tilde{e}_1,\ldots,\tilde{e}_n\}$ for $E$ there holds for $t_0\in I$
\begin{equation}D_t \gamma'(t_0)=(\Ddot{\gamma}^k(t_0)+\dot{\gamma}^j(t_0)\Ddot{\gamma}^i(t_0)\omega^k_j(\tilde{e}_i))\tilde{e}_k=\Ddot{\gamma}^k(t_0)\tilde{e}_k=\gamma''(t_0).\end{equation}
\end{example}
\begin{definition}
Let $(M,\nabla)$ be an affinely connected space and $\gamma:I\rightarrow M$ a curve over $M$. If its acceleration is zero for all $t\in I$, that is
\begin{equation}D_t \gamma'=0,\end{equation}
then $\gamma$ is said to be a \textbf{geodesic with respect to} $\bm{\nabla}$.
\end{definition}

\begin{theorem}\label{geogeo}[Existence and Uniqueness of Geodesics]
Let $(M,\nabla)$ be an affinely connected space, take $p\in M$, $X_0\in T_pM$ and $t_0\in \re$. Then, there is an open interval $I\subset \re$ with $t_0\in I$ and a geodesic $\ga:I\rightarrow M$ such that $\ga(t_0)=p$ and $\ga'(t_0)=X_0$. Also, two such geodesics agree on their common domain.
\end{theorem}

The last result can be used in order to define a relation between the tangent space $T_p M$ and points in a neighbourhood of $p$. Namely, each $X\in TM$ encompasses a point $\pi(X)=p$ and vector $X_p$, which by Theorem \ref{geogeo} yields a geodesic $\gamma_X$ such that $\gamma_X(t_0)=0$ and $\gamma'_X(t_0)=X_p$ for some $t_0\in \re$. Then, one may define the subset $\mathfrak{E}$ of $TM$ 
\begin{equation}\mathfrak{E}=\{X\in TM:\text{the maximal geodesic}\;\gamma_{X}:I\rightarrow M\;\text{has}\;[0,1]\subset I\}.\end{equation}
The \textbf{exponential map} $\exp:\mathfrak{E}\rightarrow M$ is then given by
\begin{equation}\exp(X)=\gamma_{X}(1).\end{equation}

\begin{lemma}[Rescaling lemma]
Let $X\in TM$ and let $c,t\in\re$. Then there holds
\begin{equation}\gamma_{cX}(t)=\gamma_X(ct).\end{equation}
\end{lemma}

\begin{proposition}
The exponential map $\exp:\mathfrak{E}\rightarrow M$ has the following properties:
\begin{enumerate}
    \item $\mathfrak{E}$ is an open subset of the tangent bundle $TM$ containing the zero section.
    \item The geodesic $\gamma_X$ is given by
    \begin{equation}\gamma_X(t)=\exp(tX),\end{equation}
    for all $t\in I$ and all $X\in TM$.
    \item The exponential map is smooth.
\end{enumerate}
\end{proposition}

It is also possible to consider for each $p\in M$ the restriction $\mathfrak{E}_p=\mathfrak{E}\cap T_pM$ and then define 
\begin{equation}
    \exp_p:\mathfrak{E}_e\rightarrow M.
\end{equation}
\begin{lemma}[Normal Neighborhood Lemma]
For every $p\in M$ there is a neighborhood $N_0$ of the origin in $T_pM$ and $N_p$ of $p$ in $M$ such that $\exp_p:N_0\rightarrow N_p$ is a diffeomorphism.
\end{lemma}
\begin{proof}
One might see that the push-forward \begin{equation} d(\exp_p)_p:T_0(T_pM)=T_pM\rightarrow T_pM\end{equation}
is, in fact, the identity and therefore an isomorphism. Hence, the result follows from the Inverse Mapping Theorem. Indeed, let $X_p\in T_pM$ and $\gamma:I\rightarrow T_pM$ given by $\gamma(t)=tX_p$. Then,
\begin{equation} d(\exp_p)_p(X_p)=\frac{d}{dt}\Big{|}_{t=0}\lp \exp_p\circ \gamma\rp(t)=\frac{d}{dt}\Big{|}_{t=0}\exp_p(tX_p)=\frac{d}{dt}\Big{|}_{t=0}\gamma_{X_p}(t)=X_p.\end{equation}
\end{proof}
Now, it is in our interest to introduce coordinates related to the exponential diffeomorphism containing desired properties to work with. Remarkably, the normal coordinates are extensively used in Riemannian geometry but cannot be presented without the presence of a metric $g$.

\begin{definition} An \textbf{affinely connected Riemannian space} is given by a triple $(M,g,\nabla)$ for which $M$ is a manifold endowed with a metric $g$ and an affine connection $\nabla$.
\end{definition}

Consider then the affinely connected Riemannian space $(M,g,\nabla)$. If $\{E_1,\ldots,E_n\}$ is an orthonormal basis for $T_pM$ then there is a natural isomorphism $E:T_pM\rightarrow \re^n$ given by
\bege\label{amor}E(x^iE_i)=(x_1,\ldots,x_n),\enge
which in turn defines a coordinate chart
\begin{equation}\phi=E\circ \exp_p^{-1}:N_p\rightarrow\re^n,\end{equation}
called the \textbf{normal coordinates} centered at $p$.

\begin{proposition}[Properties of Normal Coordinates]\label{propere}
Let $(N_p,\phi)=(N_p;x_1,\ldots,x_n)$ be normal coordinates centered at $p$. Then,
\begin{enumerate}
    \item For any vector field $X=X^i\partial_i\in T_pM$, the geodesic $\gamma_X$ at $p$ is represented in normal coordinates by the radial line segment
    \begin{equation}\phi(\gamma_X(t))=(tX^1,\ldots,tX^n)\end{equation}
    whenever $\gamma_X(t)$ is inside of $N_p$. Moreover, there is a neighbourhood $N'_p$ of $p$ called the \textbf{restricted normal neighborhood} at $p\in M$ for which every pair of points $a,b\in N'_p$ has an unique geodesic between them.
    \item $\phi(p)=(0,\ldots,0)$.
    \item The components of the metric at $p$ are $g_{ij}=\delta_{ij}$.
\end{enumerate}
\end{proposition}
\begin{preremark}\upshape
Fix $\dim M=n$. Then, since in normal coordinates there holds $g_{ij}=\delta_{ij}$ at the point $p$, this in turn makes
\begin{equation}
g_{ij}g^{ij}=\delta_{ij}\delta^{ij}=\tr(\text{Id})=n
\end{equation}
where $\tr$ denotes the trace of a square matrix.
\end{preremark}
\begin{preremark}\upshape 
In the next section, a notion of compatibility between a connection $\nabla$ and metric $g$ over the manifold $M$ will be considered, yielding the notions of torsion and metric-compatibility. In that case, it also follows that if $\nabla$ is metric-compatible then the partial derivatives $g_{ij,k}$ vanish and if $\nabla$ is torsionless then the Christoffel symbols $\Gamma^k_{ij}$ also vanish at $p$. It is important to note that these properties hold \textit{only} on the "central" point $p$. The description, even in normal coordinates, of geodesics between other points, as well as the properties of the partial derivatives and Christoffel symbols vanishing do not hold in all of $N_p$ in general.
\end{preremark}
\begin{preremark}\upshape
Notice that the second property in the Proposition \ref{propere} shows that for each $p\in M$ the metric $g_p$ over the tangent space in normal coordinates is given by the Euclidean one. There is a generalization of this concept (structures over a manifold that locally look like a fixed system in the tangent space) called $G$-structures. Their more general formal theory will not be presented, but the specific case of the so-called $G_2$-structures will be perceived further on.
\end{preremark}

\section{Metric-Compatibility and Torsion}

As said before, it may be the case when $M$ is endowed with both a connection $\nabla$ and a metric $g$. In this section, the relations between these structures are investigated and a sense of compatibility between them can be considered.
\begin{definition}
Let $M$ be a manifold. If $\nabla$ is an affine connection and $g$ a metric over $M$ then the triple $(M,g,\nabla)$ is called an \textbf{affinely connected Riemannian space}.
\end{definition}
\begin{definition}[Metric compatibility]\label{compatibl}
Let $(M,g,\nabla)$ be an affinely connected Riemannian space. Then, the connection $\nabla$ is said to be \textbf{metric-compatible} or \textbf{compatible with} $\bm{g}$ if 
\begin{equation}X(g(Y,Z))=g(\nabla_X(Y),Z)+g(Y,\nabla_XZ),\end{equation}
for every $X,Y,Z\in\mathfrak{X}(M)$. 
\end{definition}
\begin{preremark}\upshape 
The condition in definition \ref{compatibl} can be put in the following manner: let $\gamma:I\rightarrow M$ be a curve in $M$ and $X,Y:I\rightarrow TM$ sections along $\gamma$. Then, the connection $\nabla$ is compatible with $g$ if and only if
\begin{equation}\frac{d}{dt}g(X,Y)=g(D_t X,Y)+g(X,D_t Y).\end{equation}
\end{preremark}

\begin{proposition}\label{nablazero}
Let $(M,g,\nabla)$ be an affinely connected Riemannian space. The following statements are equivalent:
\begin{enumerate}
    \item The connection $\nabla$ is compatible with $g$.
    \item For any curve $\gamma:[t_0,t_1]\rightarrow M$ with $\ga(t_0)=p$ and $\ga(t_1)=q$, the parallel transport $T_{\gamma}^{t_0,t_1}$ with respect to $\nabla$ is an isometry between $(T_pM,g_p)$ and $(T_qM,g_q)$.
    \item The total covariant derivative $\nabla g\equiv 0$
\end{enumerate}
\end{proposition}

Now, one may recall the \textbf{Lie bracket} operation 
\begin{equation}
    [\cdot,\cdot]:\mathfrak{X}(M)\times\mathfrak{X}(M)\rightarrow\mathfrak{X}(M),
\end{equation}
which, for a function $f\in C^{\infty}(M)$, is given by
\begin{equation}
    [X,Y]f=X(Y(f))-Y(X(f)).
\end{equation}
Taking a coordinate basis $\{\partial_1,\ldots,\partial_n\}$ for $T_pM$, it follows that
\begin{equation}
    \begin{split}
        [X^i\partial_i,Y^j\partial_j](f)&=X^i\partial_i(Y^j\partial_j(f))-Y^j\partial_j(X^i\partial_i(f))\\
        &=X^i\partial_i(Y^j)\partial_j(f)-Y^j\partial_j(X^i)\partial_i(f)+X^iY^j(\partial_i\partial_j)(f)-Y^jX^i(\partial_j\partial_i)(f)\\
        &=\lp X^i\partial_i(Y^j)-Y^i\partial_i(X^j)\rp\partial_j,
    \end{split}
\end{equation}
where the relation $\partial_i\partial_j=\partial_j\partial_i$ (by Schwarz's Theorem) was used. The Lie bracket is intrinsically related to Lie algebras and groups (more details on refs. \cite{warner,leeriemann,michor}). It measures the failure of the commutation of the global derivations spanned by the vector fields $X,Y\in\mathfrak{X}(M)$. Note that the relation
\begin{equation}
    [\partial_i,\partial_j]=0,
\end{equation}
was used in the previous calculation. Such identity always holds for coordinate bases, but not necessarily in more general ones. Besides, for any $X,Y,Z\in\mathfrak{X}(M)$ one may see the so-called \textbf{Jacobi identity}
\begin{equation}\label{jacobi}
    [X,[Y,Z]]+[Y,[Z,X]]+[Z,[X,Y]]=0
\end{equation}
holds.
\begin{definition}
Let $(M,\nabla)$ be an affinely connected space and define the operator 
\begin{equation}T:\mathfrak{X}(M)\times\mathfrak{X}(M)\rightarrow \mathfrak{X}(M),\end{equation}
called the \textbf{torsion} with respect to $\nabla$ by the relation
\begin{equation} T(X,Y)=\nabla_X (Y)-\nabla_Y(X)-[X,Y].\end{equation}
\end{definition}
\begin{preremark}\upshape 
It is an easy task to verify that the torsion $T$ for a connection $\nabla$ is a $C^{\infty}(M)$-linear alternating operator. Therefore, there holds $T\in\Gamma(\Lambda^2(TM)\otimes TM)=\Omega^2(M,TM).$ In addition, a connection is called \textbf{torsionless} or \textbf{torsion-free} whenever $T=0$.
\end{preremark}

\begin{preremark}\upshape
Let $p\in M$ and choose local coordinates $(U;x^1,\ldots,x^n)$ with such the coordinate vectors are given by $\{\partial_1,\ldots,\partial_n\}$. Then, one can define the torsion symbols $T^i_{jk}$ by
\begin{equation} T^i_{jk}=dx^i\lp T(\partial_i,\partial_j)\rp=dx^i(\Gamma^i_{jk}\partial_i-\Gamma^i_{kj}\partial_i)=\Gamma^i_{jk}-\Gamma^i_{kj}.\end{equation}
Notice that by definition the symbols satisfy
\begin{equation}
    T^i_{jk}=-T^i_{kj},
\end{equation}
and the connection is torsionless if and only if $\Gamma^i_{[jk]}=0$. One may as well lower the upper index, yielding
\begin{equation} T_{ijk}=g_{il}T^l_{jk}.\end{equation}
\end{preremark}
The metric-compatibility and the torsion of a connection $\nabla$ are related, in the sense that given a metric $g$ these properties uniquely define a connection over $M$.
\begin{lemma}[Fundamental Lemma of Riemannian Geometry]
Let $(M,g)$ be a Riemannian manifold. Then, there is an unique torsionless connection $\nabla$ compatible with $g$. Such connection is called the \textbf{Levi-Civita connection} of $(M,g)$.
\end{lemma}
\begin{proof}
Suppose first that $\nabla$ is indeed a connection compatible with $g$ and let $X,Y,Z\in\mathfrak{X}(M)$ be arbitrary vector fields. Using the compatibility relation, it follows that
\begin{equation}\begin{split}
    Xg(Y,Z)&=g(\nabla_X Y,Z)+g(Y,\nabla_X Z),\\
    Yg(Z,X)&=g(\nabla_Y Z,X)+g(Z,\nabla_Y X),\\
    Zg(X,Y)&=g(\nabla_ZX,Y)+g(X,\nabla_ZY).
\end{split}\end{equation}
Since $\nabla$ is torsionless, there holds $\nabla_X Z=\nabla_Z X+[X,Z]$ and similar relations with respect to $\nabla_YX$ and $\nabla_ZY$, yielding
\begin{equation}\begin{split}
    Xg(Y,Z)&=g(\nabla_X Y,Z)+g(Y,\nabla_Z X)+g(Y,[X,Z]),\\
    Yg(Z,X)&=g(\nabla_Y Z,X)+g(Z,\nabla_X Y)+g(Z,[Y,X]),\\
    Zg(X,Y)&=g(\nabla_ZX,Y)+g(X,\nabla_YZ)+g(X,[Z,Y]).
\end{split}\end{equation}
One may add the two first equations and subtract the third, resulting in
\begin{equation} Xg(Y,Z)+Yg(Z,X)-Zg(X,Y)=2g(\nabla_X Y,Z)+g(Y,[X,Z])+g(Z,[Y,X])-g(X,[Z,Y]).\end{equation}
Rearranging the terms comes
\bege\label{excele}
g(\nabla_X Y,Z)=\frac{1}{2}\lp Xg(Y,Z)+Yg(Z,X)-Zg(X,Y)-g(Y,[X,Z])-g(Z,[Y,X])+g(X,[Z,Y])\rp.
\enge
Now, since the right-hand side of eqn (\ref{excele}) does not depend on the connection $\nabla$, if $\nabla^1$ and $\nabla^2$ are both torsionless connection, then it follows that
\begin{equation}g(\nabla^1_X Y-\nabla^2_X Y,Z)=0,\end{equation}
for every $X,Y,Z\in\mathfrak{X}(M)$. Since $g$ is non-degenerate, it follows that $\nabla^1_XY=\nabla^2_XY$, for every $X,Y\in\mathfrak{X}(M)$. This proves uniqueness.

In order to prove existence one may show that such connection exists locally, then by uniqueness this must be the only possible connection. A local chart $(U,x_1,\ldots,x_n)$ may be considered and applying eqn (\ref{excele}) on the coordinate vectors $\{\partial_1,\ldots,\partial_n\}$ yields
\begin{equation}
g(\nabla_{\partial_i}\partial_j,\partial_l)=\frac{1}{2}\big{(}\partial_ig(\partial_j,\partial_l)+\partial_jg(\partial_l,\partial_i)-\partial_lg(\partial_i,\partial_j)\big{)}.
\end{equation}
It then follows in index notation that
\bege\label{laste}
\Gamma^m_{ij}g_{ml}=\frac{1}{2}(g_{jl,i}+g_{il,j}-g_{ij,i}).
\enge
Using the metric inverse of $g^{lk}$, there holds
\begin{equation}
\Gamma^k_{ij}=\frac{1}{2}g^{kl}(g_{jl,i}+g_{il,j}-g_{ij,i}).
\end{equation}
Since $\Gamma^k_{ij}=\Gamma^k_{ji}$, it follows that $\nabla$ is torsionless. By Lemma \ref{nablazero}, to prove that $\nabla$ is compatible with $g$ it suffices to prove that $\nabla g=0$. In components, one has
\begin{equation}\nabla_kg_{ij}=g_{ij,k}-\Gamma^l_{ki}g_{lj}-\Gamma^l_{kj}g_{il}.\end{equation}
One may use eqn (\ref{laste}) to conclude that
\begin{equation}\begin{split}
\Gamma^l_{ki}g_{lj}+\Gamma^l_{kj}g_{il}&=\frac{1}{2}(g_{ij,k}+g_{kj,i}-g_{ki,j})+\frac{1}{2}(g_{ji,k}+g_{ki,j}-g_{kj,i})\\
&=\partial_kg_{ij}.
\end{split}\end{equation}
Therefore, $\nabla g=0$ and so $\nabla$ is indeed compatible with $g$.
\end{proof}
\begin{preremark}\upshape
Since the Levi-Civita connection for a Riemannian manifold $(M,g)$ is unique, one may sometimes denote it by $\nabla^g$ to make explicit its dependence on the metric $g$. Moreover, another tensor which shall be of great importance is the contorsion or the Cartan torsion\footnote{This tensor is sometimes simply referred as torsion in the literature, although there are some differences between them in the general case.} of a connection $\nabla$ as follows.
\end{preremark}
\begin{definition} 
Let $(M,g,\nabla)$ be an affinely connected Riemannian space and consider its Levi-Civita connection $\nabla^g$. Then, the \textbf{contorsion} is an application $S:\mathfrak{X}(M)\times\mathfrak{X}(M)\rightarrow\mathfrak{X}(M)$ measuring the difference between $\nabla$ and the Levi-Civita connection. It is given by the relation
\begin{equation}\label{contorsi}
    S(X,Y)+\nabla_X Y= \nabla^{g}_X Y.
\end{equation}
\end{definition}
It is straightforward to see that the contorsion is a $\binom{1}{2}$-tensor field, since by definition it is $C^{\infty}(M)$-linear on the first entry and for $f\in C^{\infty}(M)$
\begin{equation}
    \begin{split}
        S(X,fY)&=\nabla^g_X (fY)-\nabla_X(fY)\\
        &=f\lp \nabla^g_XY-\nabla_X Y\rp+X(f)Y-X(f)Y\\
        &=f S(X,Y).
    \end{split}
\end{equation}
Additionally, it encompasses a considerable amount of information about the original connection $\nabla$. For the next result, fix the affinely connected Riemannian space $(M,g,\nabla)$ with contorsion $S$ as in eqn (\ref{contorsi}).
\begin{proposition} 
The connection $\nabla$ is torsion-free if and only if the contorsion tensor $S$ is symmetric.
\end{proposition}
\begin{proof}
By a straight-forward computation, one has
\begin{equation}\label{conttt}
\begin{split}
    T(X,Y)&=\nabla_X Y-\nabla_Y X-[X,Y]\\
    &=\nabla^g_XY-\nabla^g_YX-[X,Y]+S(Y,X)-S(X,Y)\\
    &=S(Y,X)-S(X,Y),
    \end{split}
\end{equation}
hence the result.
\end{proof}
One may explicitly define the tensor obtained by lowering the contorsion upper index given by
\begin{equation}A:\mathfrak{X}(M)\times\mathfrak{X}(M)\times\mathfrak{X}(M)\rightarrow\re,\end{equation} 
where, as usual, there holds
\begin{equation}
    A(X,Y,Z)=g(S(X,Y),Z).
\end{equation}
\begin{proposition}
The connection $\nabla$ is metric-compatible if and only if the tensor $A(X,Y,Z)=g(S(X,Y),Z)$ is anti-symmetric in $Y$ and $Z$.
\end{proposition}
\begin{proof}
The connection $\nabla$ is metric-compatible if and only if 
\begin{equation}
    X(g(Y,Z))=g(\nabla_X Y,Z)+g(Y,\nabla_X Z),
\end{equation}
for every $X,Y,Z\in\mathfrak{X}(M)$. On the other hand,
\begin{equation}
    g(\nabla_XY,Z)+g(Y,\nabla_XZ)=g(\nabla^g_X Y,Z)+g(Y,\nabla^g_XZ)-g(S(X,Y),Z)-g(Y,S(X,Z)),
\end{equation}
and since $\nabla^g$ is metric-compatible, it follows that
\begin{equation}
    g(S(X,Y),Z)+g(S(X,Z),Y)=0.
\end{equation}
\end{proof}
\begin{definition}
Given two affine connections $\nabla$ and $\widetilde{\nabla}$ over $M$, one says that they have the \textbf{same set of geodesics} over $M$ whenever a geodesic $\gamma:I\rightarrow M$ for $\nabla$ is also a geodesic for $\widetilde{\nabla}$ and vice-versa.
\end{definition}
\begin{proposition}
The connection $\nabla$ has the same set of geodesics as $\nabla^g$ if and only if $A(X,Y,Z)=g(S(X,Y),Z)$ is anti-symmetric in $X$ and $Y$. Equivalently, if and only if $S$ is alternating.
\end{proposition}
\begin{proof}
Indeed, let $X$ be any vector field and $p\in M$. Then, let $\gamma$ be a geodesic with $\gamma(0)=p$ and $\gamma'(0)=X_p$. If $\nabla$ and $\nabla^g$ share the same geodesics, then
\begin{equation}
    S(X_p,X_p)=\nabla^g_{\gamma'}\gamma'-\nabla_{\gamma'}\gamma'=0
\end{equation}
and then it follows that $S$ is alternating. Conversely, let $\gamma$ be any curve. If $S$ is alternating then
\begin{equation}
    \nabla_{\gamma'}\gamma'=\nabla^g_{\gamma'}\gamma',
\end{equation}
and therefore a curve $\gamma$ is a solution of the geodesic equation for $\nabla$ if and only if it is a solution of the geodesic equation for $\nabla^g$.
\end{proof}
\begin{preremark}\upshape
Notice that, by means of eqn (\ref{conttt}), if $S$ is alternating (hence anti-symmetric), then,
\begin{equation}
     T(X,Y)=-S(X,Y)+S(Y,X)=-2S(X,Y),
\end{equation}
which, in coordinates, may be seen as
\begin{equation}
    S^i_{jk}=-\frac{1}{2}T^i_{jk}=-\Gamma_{[jk]}.
\end{equation}
In conclusion
\end{preremark}
\begin{theorem}\label{millmanmetric}
Let $(M,g,\nabla)$ be an affinely connected Riemannian space and denote by $\nabla^g$ its Levi-Civita connection and by $S$ the associated contorsion. Then, $\nabla$ is metric-compatible and shares the same geodesics as $\nabla^g$ if and only if the tensor $A(X,Y,Z)=g(S(X,Y),Z)$ is totally anti-symmetric. In addition, in that case there holds \bege T=-2S.\enge
\end{theorem}
\begin{preremark}\upshape 
We believe that Theorem \ref{millmanmetric} may provide a good explanation on why the contorsion of an affine connection $\nabla$ is simply denoted by "torsion" in some of the literature. Firstly, because in many applications (one of which shall be presented in the next chapter) one considers an affinely connected Riemannian space $(M,g,\nabla)$, for which the contorsion is assumed to be totally anti-symmetric, so that all the desired properties listed in Theorem \ref{millmanmetric} may hold for the original connection $\nabla$. Besides, since in that case the torsion and the contorsion are multiples, there might be some confusion.
\end{preremark}

The contorsion is an important object in the so-called Einstein-Cartan theory (of gravity), where one considers the nonvanishing of torsion in the underlying space. More details on the contorsion and its applications to physics can be found in \cite{agricola, fabri}.

\section{Curvature}
In this section, the notion of curvature for more general affinely connected spaces $(M,\nabla)$ is analyzed and its local properties are investigated. The Riemannian geometric interpretation of the curvature in the context of Riemannian manifolds $(M,g)$ may be considered (see e.g \cite{leeriemann}) but is for now left aside. Whenever a metric is needed (and its Levi-Civita connection by extent) their emergence shall be explicitly made clear, so that results here may be presented as generally as possible.
\begin{definition}
Let $(M,\nabla)$ be an affinely connected space. The \textbf{curvature} tensor is the map $R:\mathfrak{X}(M)\times\mathfrak{X}(M)\times\mathfrak{X}(M)\rightarrow\mathfrak{X}(M)$ given by
\begin{equation} R(X,Y)Z=\nabla_X\nabla_Y Z-\nabla_Y\nabla_X Z-\nabla_{[X,Y]}Z.\end{equation}
\end{definition}

By straight-forward calculations one may see that the curvature tensor is indeed a $\binom{1}{3}$-tensor field. By proceeding locally, if $(U;x^1,\ldots,x^n)$ are local coordinates around a point $p\in U$ then there are $4n$ smooth functions $R^i_{jkl}:U\rightarrow \re$ such that
\begin{equation}R=R^i_{jkl} \partial_i\otimes dx^j\otimes dx^k\otimes dx^l,\end{equation}
where the convention\footnote{One must be extremely careful about the curvature tensor convention, as it changes depending on the literature.}
\begin{equation} R(\partial_k,\partial_l)\partial_j=R^i_{jkl}\partial_i\end{equation}
is taken. One can also compute the curvature tensor in terms of the Christoffel symbols, namely
\begin{equation}
    R(\partial_k,\partial_l)\partial_j=\nabla_{\partial_k}\nabla_{\partial_l}\partial_j-\nabla_{\partial_l}\nabla_{\partial_k}\partial_j-\nabla_{[\partial_k,\partial_l]}\partial_j,
\end{equation}
and since locally $[\partial_k,\partial_l]=0$, there holds
\begin{equation}\begin{split}
    R(\partial_k,\partial_l)\partial_j&=\nabla_{\partial_k}\nabla_{\partial_l}\partial_j-\nabla_{\partial_l}\nabla_{\partial_k}\partial_j\\
    &=\nabla_{\partial_k}(\Gamma^m_{lj}\partial_m)-\nabla_{\partial_l}(\Gamma^m_{kj}\partial_m)\\
    &=\Gamma^m_{lj}\Gamma^i_{km}\partial_i-\Gamma^m_{kj}\Gamma^i_{lm}\partial_i+\Gamma^i_{lj,k}\partial_i-\Gamma^i_{kj,l}\partial_i,
\end{split}\end{equation}
so that
\begin{equation}
R^i_{jkl}=\Gamma^m_{lj}\Gamma^i_{km}-\Gamma^m_{kj}\Gamma^i_{lm}+\Gamma^i_{lj,k}-\Gamma^i_{kj,l}.
\end{equation}
Whenever the manifold $M$ is endowed with a Riemannian metric $g$ one may also take the lowering of the first index and get a covariant curvature $4$-tensor given by \begin{equation} R_{ijkl}=g_{im}R^m_{jkl}.\end{equation} 
\begin{theorem}[Generalized Bianchi Identities \cite{michor}]\label{curvaturesym}
Let $(M,\nabla)$ be an affinely connected space,  $X,Y,Z\in\mathfrak{X}(M)$ and let $T$ denote its torsion tensor. Then, there holds
\begin{enumerate}
\item $R(X,Y)Z=-R(Y,X)Z$,
    \item $\displaystyle\sum_{\text{cyclic}} R(X,Y)Z= \sum_{\text{cyclic}}\Big{(}\big{(}\nabla_X T\big{)}(Y,Z)+T(T(X,Y),Z)\Big{)}\quad\quad$(first Bianchi identity),
    \item $\displaystyle\sum_{\text{cyclic}}\Big{(}\big{(}\nabla_XR\big{)}(Y,Z)+R(T(X,Y),Z)\Big{)}=0\quad\quad$(second Bianchi identity).
\end{enumerate}
\end{theorem}
\begin{preremark}\upshape
Notice that Theorem \ref{curvaturesym} is, in fact, a generalization of the more well-known Bianchi identities seen in Riemannian geometry introductory courses. In such background, this is due to considering the Levi-Civita connection of a metric defined upon the manifold $M$ instead of a more general connection. Indeed, consider now the Riemannian manifold $(M,g)$ and endow it with the unique torsionless metric-compatible connection $\nabla^g$. Its curvature tensor shall be denoted by
\begin{equation}
    \mathring{R}(X,Y)Z=\nabla^g_X\nabla^g_YZ-\nabla^g_Y\nabla_X^gZ-\nabla^g_{[X,Y]}Z
\end{equation}
which will be called the \textbf{Riemannian curvature tensor}. By definition, the torsion for $\nabla^g$ vanishes and therefore the Bianchi identities $(2)$ and $(3)$ in Theorem \ref{curvaturesym} take the more well-known forms
\begin{equation}\begin{split}
    \displaystyle\sum_{\text{cyclic}} \mathring{R}(X,Y)Z&=0\quad\quad\text{(first Riemannian Bianchi identity)}.\\
    \displaystyle\sum_{\text{cyclic}}\big{(}\nabla_X\mathring{R}\big{)}(Y,Z)&=0\quad\quad\text{(second Riemannian Bianchi identity)}.
\end{split}\end{equation}
When a Riemannian metric is present, one may moreover consider the following also well-known symmetries.
\end{preremark}
\begin{proposition}
Let $(M,g,\nabla)$ be an affinely connected Riemannian space. Then,
\begin{enumerate}
    \item If $\nabla$ is metric-compatible then
    \bege\label{indp}g(R(X,Y)Z,W)=-g(R(X,Y)W,Z).\enge
    \item If $\nabla=\nabla^g$ then
    \bege\label{inds}g(R(X,Y)Z,W)=g(R(Z,W)X,Y).\enge
    \end{enumerate}
\end{proposition}

It is often useful to consider the index notation which in general summarizes, at least in the local setting, the information contained in the curvature tensor. Let $(M,g,\nabla)$ be an affinely connected Riemannian space and take the local coordinates $(U;x^1,\ldots,x^n)$. Then, $(i_1\cdots i_n)$ denotes the symmetrization of the indices $i_1,\ldots,i_n$ whereas $[i_1\cdots i_k]$ denote their alternation. The first equation in Theorem \ref{curvaturesym} reads
\begin{equation}
    R^i_{jkl}=-R^i_{jlk}.
\end{equation}
Following up, the generalized Bianchi identities can be perceived by the relations
\begin{equation}\begin{split}
    R^i_{(jkl)}-T^i_{(jk;l)}-T^i_{m(j}T^m_{kl)}&=0,\\
    R^{i}_{\;j(mk;l)}+R^{ij}_{\;\;n(m}T^n_{\;kl)}&=0,
\end{split}\end{equation}
with their well-known counterparts when $T=0$. Considering now a metric-compatible connection $\nabla$, one may now lower/raise indices. Equation (\ref{indp}) reads
\begin{equation}
    R_{ijkl}=-R_{jikl},
\end{equation}
and moreover assuming that $\nabla$ is torsionless (so that $\nabla=\nabla^g$) yields
\begin{equation}\label{thissa}
    R_{ijkl}=R_{klij}.
\end{equation}

A few contractions may be considered, as follows, in order to make the curvature tensor easier to manipulate. Sometimes, the following notions carry enough information so that one may consider them as constraints in geometrical problems. This introductory chapter is then concluded with a brief application of the developed theory, with the Einstein field equations in general relativity being presented along with the notion of Einstein spaces.
\begin{definition}
Let $(M,g,\nabla)$ be an affinely connected space. The \textbf{Ricci curvature tensor} is defined as the contraction of the curvature tensor using the metric $g$ in the first and third indices. Its components are denoted $R_{ij}$ and satisfy
\begin{equation}
    R_{ij}=R^k_{ikj}=g^{km}R_{mikj}.
\end{equation}
Further contracting the Ricci tensor gives the \textbf{scalar curvature} $R$, given by
\begin{equation}
    R=R^i_i=R_{ij}g^{ij}.
\end{equation}
\end{definition}
\begin{preremark}\upshape
Notice that from eqn (\ref{thissa}) it follows that whenever $\nabla=\nabla^g$, then the Ricci curvature is a symmetric $2$-tensor.
\end{preremark}
The geometric theory of Riemannian manifolds $(M,g)$ is remarkably prominent in the general relativity theory of gravity, in which one considers the Einstein field equations
\begin{equation}
    R_{\mu\nu}-\frac{1}{2}Rg_{\mu\nu}=8\pi GT_{\mu\nu}-\Lambda g_{\mu\nu},
\end{equation}
where $G$ is the gravitational constant, $\Lambda$ is called the cosmological constant and $T_{\mu\nu}$ the energy-momentum tensor, which describes the mass distribution of a given phenomenon. In this background, $M$ is considered to be a $4$-dimensional manifold, called the space-time, which is endowed with a (to be determined) metric\footnote{To the well-aware reader, this shall be, in fact, a pseudo-Riemannian metric over the manifold $M$.} $g$ over $M$ and one considers the Levi-Civita connection. Intuitively, since $g$ depends on the observed phenomenon (the data given by $T_{\mu\nu}$), this equation describes how mass, and therefore gravity, is connected to the underlying geometry of space-time, here encompassed by the unknown metric which in turn defines the unique connection $\nabla^g$. A thorough exposition on general relativity with many examples can be seen in \cite{carroll}.

Now, in order to determine the metric one may consider the elementary case, namely the vacuum-state of this theory ($T=0$). Contracting the field equations with $g^{\mu\nu}$ gives
\begin{equation}
    R-2R=-4\Lambda,
\end{equation}
that is,
\begin{equation}
    R=4\Lambda.
\end{equation}
Then, inserting this relation back to the equations yields
\begin{equation}\label{riccii}
    R_{\mu\nu}=\Lambda g_{\mu\nu}.
\end{equation}
Equation \ref{riccii} then motivates the
\begin{definition} An affinely connected Riemannian space $(M,g,\nabla)$ is called an \textbf{Einstein space} if the metric $g$ is a scalar multiple of the Ricci curvature, that is, there is a function $\lambda:M\rightarrow\re$ such that
\begin{equation}
    R_{ij}=\lambda g_{ij}
\end{equation}
all over $M$. Besides, whenever $\lambda=0$ identically then $M$ is called \textbf{Ricci-flat}.
\end{definition}

\part{Geodesic Loops and Applications}

\chapter{Geodesic Loops}
This chapter is devoted to establishing the theory of the so-called geodesic loops over an affinely connected space $(M,\nabla)$. The construction and basic notions of such structure shall be considered and one shall be able to see that they are intrinsically related to the geometry produced by the connection $\nabla$. 

It is subdivided in the following way: the first section is devoted to the aforementioned basic constructions over geodesic loops as given by \cite{kikkawa}; in the second one, the algebraic aspects of these structures are studied, such as their fundamental tensors and $W$-algebras; subsequently, a metric over the manifold $M$ is considered in order to depict some of the relations between the fundamental tensors of a geodesic loop in an affinely connected Riemannian space and its underlying geometry, which are results developed in \cite{akivis151,akivis152}; finally, in the last section an introduction to applications of this theory is provided, for instance in the context of supergravity, namely the Kaluza-Klein spontaneous compactification theory, following the work presented in \cite{loginov2,loginov3}. 
\section{Local Loops}
In this section, the more general notion of local loops is considered, as originally presented in \cite{kikkawa}. The word 'loop' here is, in fact, related to the algebraic definition of a loop, as follows.

\begin{definition}[Quasi-group]
Let $A$ be a non-empty set and $*:A\times A\rightarrow A$ a binary operation. If, for every $x,y\in A$, there are unique $z_1,z_2\in A$ such that
\begin{equation}\label{quasigroup}
x*z_1=y\;\;\text{and}\;\;z_2*x=y,\enge
then the pair $(A,*)$ is called a \textbf{quasi-group}.
\end{definition}

\begin{definition}[Loop]\label{loop}
Let $(A,*)$ be a quasi-group. If there is an element $e\in A$ such that \bege x*e=e*x=x\enge for every $x\in A$, then $(A,*)$ is called a \textbf{loop}. Such element $e$ is called an \textbf{identity element} (or unity) in the loop $(A,*)$.  
\end{definition}
\begin{preremark}\upshape
Notice that, in a loop $(A,*)$, the unity $e$ must be unique. Indeed, if there is another element $\tilde{e}\in A$ with such property, then
\begin{equation}\tilde{e}=e*\tilde{e}=e,\end{equation}
so that one may call $e$ \textit{the} unity in the loop $(A,*)$. Moreover, taking $y=e$ in the quasi-group relation (\ref{quasigroup}) produces the elements $x^{-1}_L,x_R^{-1}\in A$ such that
\begin{equation}
x*x^{-1}_R=e\;\;\text{and}\;\;x^{-1}_L*x=e,\enge
respectively called the \textbf{right} and \textbf{left inverses} in the loop $(A,*)$. It is straightforward to see that an associative loop is a group, since there the right and left inverses agree. Indeed,
\begin{equation} 
x^{-1}_L=x^{-1}_L*(x*x^{-1}_R)=(x^{-1}_L*x)*x^{-1}_R=x^{-1}_R.\end{equation}
The notion of local loops in a topological space $M$ may now be introduced. In what follows, whenever $x\in U\subset M$, one denotes $U_{x}^1=\{x\}\times U$ and $U_{x}^2=U\times\{x\}$ in $M\times M$.
\end{preremark}

\begin{definition}[Local Loop]\label{localloop}
Let $M$ be a topological space. If there is an open set $U\subset M$ and a continuous map
\begin{equation}\mu:U\times U\rightarrow U\end{equation}
such that, for every $x\in U$, there holds
\begin{enumerate}
    \item  $\mu|_{U_x^1}$ and $\mu|_{U_x^2}$ are homeomorphisms onto $U$;
    \item  There is $e\in U$ such that $\mu(e,x)=\mu(x,e)=x$,  
\end{enumerate}
then the pair $\mathcal{L}(U,\mu)$ is called a \textbf{local loop} over $M$.
\end{definition}
\begin{preremark}\upshape
For simplicity, the juxtaposition product notation
\begin{equation} \mu(x,y)=xy\end{equation}
may be used. It is clear that property $(2)$ in Definition \ref{localloop} guarantees the existence of a unity for this product. In addition, given $x\in U$ condition $(1)$ shows that $\mu$ restricted to $\{x\}\times U$ or $U\times\{x\}$ is bijective onto $U$, so that given $y\in U$ there are unique $z_1,z_2\in U$ such that
\begin{equation}
\mu(x,z_1)=xz_1=y\;\;\text{and}\;\;\mu(z_2,x)=z_2x=y,\enge
which is precisely relation (\ref{quasigroup}) in the definition of a quasi-group. Hence, $U$ endowed with this juxtaposition product is a loop as given in Definition \ref{loop}.
\end{preremark}
The main example analyzed here is due to Kikkawa \cite{kikkawa} where an affinely connected space $(M,\nabla)$ is considered. Since now one deals with a smooth structure, consider the
\begin{definition}[Differentiable Local Loop]
A local loop $\mathcal{L}(U,\mu)$ over a manifold $M$ is called a \textbf{differentiable local loop} whenever $\mu$ is smooth.\footnote{More details on smooth loops as a generalization of Lie groups and other constructions regarding loops may be seen in \cite{Grigorian:2020tlr}.}.
\end{definition}
In that setting, a differentiable local loop may be defined when considering the affinely connected space $(M,\nabla)$. Indeed, take $e\in M$ and consider its normal neighbourhood $N_e$, over which the exponential map \begin{equation}\exp_e:N_0\rightarrow N_e\end{equation} is a diffeomorphism. One may moreover assume that $N_e$ is the restricted normal neighborhood for $e$ so that every point in $N_e$ is connected by an unique geodesic.

Now, fix $x,y\in N_e$ and consider the unique geodesic $\gamma$ between $e$ and $y$ in $N_e$ with $\gamma(t_0)=e$ and $\gamma(t_1)=y$, so that the parallel transport $T^{t_0,t_1}_{\gamma}$ over this curve can be taken into consideration. One may then define the \textbf{geodesic loop product} between the point $x$ and $y$ by means of the expression \cite{kikkawa2}
\begin{equation}\label{geoprodd}
    \mu(x,y)=\exp_y\circ T^{t_0,t_1}_{\gamma}\circ\exp^{-1}_e(x).
\end{equation}
It is also straightforward to see that $\mu(x,e)=\mu(e,x)=x$ for every $x\in N_e$ so that condition $(2)$ in  Definition \ref{localloop} is already satisfied. 
\begin{figure}[h]
    \centering
    \includegraphics[width=0.55\textwidth]{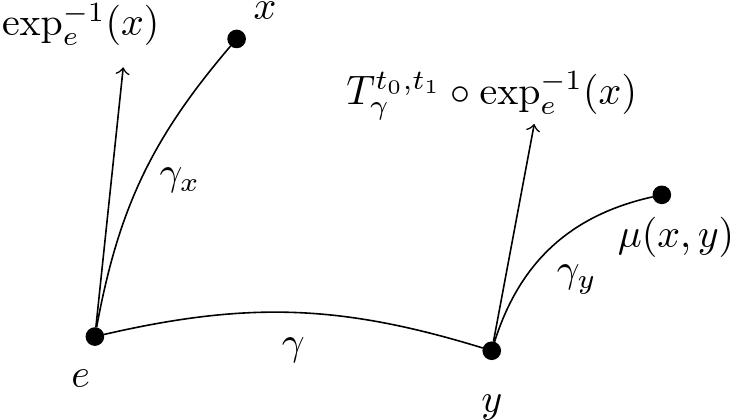}
    \caption{The geodesic loop product, as given in eqn \ref{geoprodd}.}
    \label{fig:my_label}
\end{figure}
\begin{theorem}\label{ocara}
Let $(M,\nabla)$ be an affinely connected space. Then, for every $e\in M$ eqn (\ref{geoprodd}) defines a differentiable local loop $\mathcal{L}(N_e,\mu)$ around a restricted normal neighbourhood $N_e$.
\end{theorem}
\begin{proof}
In order to see that $\mu$ is smooth, consider the local coordinates $(N_e;x^1,\ldots,x^n)$ centered at $e$ and let $\Gamma^i_{jk}$ be the Christoffel symbols for the connection $\nabla$ with respect to the coordinate basis $\{\partial_1,\ldots,\partial_n\}$. For the sake of simplicity, let $T_{e,p}$ denote the parallel transport over the unique geodesic joining $e$ and $p\in N_e$. 

Consider $x,y\in N_e$ and denote by $\gamma_x(t)$ and $\ga(s)$ the unique geodesics between $e$ and each of $x$ and $y$ respectively, with
\begin{equation}\gamma_x(0)=\gamma(0)=e,\;\;\;\;\gamma_x(1)=x,\;\;\ga(1)=y.\end{equation} Then of course $\gamma_x'(0)=\exp^{-1}_e(x)$. Moreover, letting \begin{equation}X(s)=X_{\ga(s)}=T_{e,\ga(s)}(X_e)=T_{e,\ga(s)}(\gamma_x'(0))\end{equation}
it follows that $X(s)$ is the unique parallel vector field over $\gamma(s)$ with $X(0)=\gamma_x'(0)$ so that it is completely determined in $N_e$ by the differential equation
\bege\label{adoxi}
\dot{X}^i(s)+\dot{\gamma}^j(s)X^k(s)\Gamma^i_{jk}(\ga(s))=0,\;\;\;\quad i,j,k=1,\ldots,n,
\enge
and the initial condition $X^i(0)=
\dot{\gamma}_x^i(0)$. Notice that by construction \begin{equation}X(1)=T_{\gamma}^{0,1}\circ\exp_e^{-1}(x).\end{equation}

Consider now the geodesic $\gamma_y(t)$ such that $\gamma_y(0)=y$ and $\gamma_y'(0)=X(1)$. It follows that $\gamma_y(t)$ is uniquely determined inside of $N_e$ as the solution of
\bege\label{geod}
\Ddot{\gamma}_y^i(t)+\Gamma^i_{jk}(\gamma(t))\dot{\gamma}_y^j(t)\dot{\gamma}_y^k(t)=0,\;\;\;\quad i,j,k=1,\ldots,n.
\enge
 Since $\gamma_y$ is a geodesic and $N_e$ is the restricted normal neighbourhood of all of its points, then it is defined for $t=1$ so that
 \begin{equation}
\gamma_y(1)=\exp_y(\gamma_y'(0))=\exp_y(X(1))=\exp_y\circ T^{0,1}_{\gamma}\circ\exp_e(x),
 \end{equation}
 which is precisely eqn (\ref{geoprodd}). The application $\mu$ may then be defined by
\begin{equation}\begin{split}
\begin{split}
\mu: N_e\times N_e&\longrightarrow N_e\\
(x,y)&\longmapsto \mu(x,y)=\gamma_y(1).
\end{split}
\end{split}\end{equation}
Now, it is known that the Christoffel symbols $\Gamma^i_{jk}$ are smooth functions over $N_e$ in such a way that for two points $x,y\in N_e$, the solutions $\ga^i_y(t)$ of eqn (\ref{geod}) are differentiable with respect to the parameter $t$ and initial values $y^1,\ldots,y^n;X^1(1),\ldots,X^n(1)$. Using the same argument for eqn (\ref{adoxi}) it follows that there are uniquely determined differentiable functions $\mu^i(x^1,\ldots,x^n,y^1,\ldots,y^n)$ such that
\begin{equation}
\gamma_y^i=\mu^i(x^1,\ldots,x^n,y^1,\ldots,y^n),
\end{equation}
so that $\mu$ is, in fact, a differentiable map in $N_e\times N_e$.

Finally, in order to show property $(1)$ in Definition \ref{localloop} one can take $x\in N_e$ and consider the restriction of $\mu$ over $N_e\times \{x\}$, which shall be denoted by $\mu_x$. It is given by the expression
\begin{equation}\mu_x(z)=\mu(z,x)=\exp_x\circ T_{e,x}\circ\exp^{-1}_e(z),\end{equation}
which is clearly continuous with inverse given by
\begin{equation}
\mu_x^{-1}(z)=\exp_e\circ T_{x,e}\circ \exp^{-1}_x(z).\end{equation}
It follows that $\mu_x$ is a homeomorphism onto $N_e$. Conversely, one may study the restriction of $\mu$ over $\{x\}\times N_e$. Setting $x=e$ yields, in local coordinates, the relation
\begin{equation}
\mu^i(0,\ldots,0,z^1,\ldots,z^n)=z^i,
\end{equation}
since $\mu(e,z)=z$ and $e=(0,\ldots,0)$ in such coordinates. It then follows that
\begin{equation}
\frac{\partial \mu^i}{\partial z^j}{\Bigg|_{z=e}}=\frac{\partial \mu^i(0,\ldots,0,z^1,\ldots,z^n)}{\partial z^j}{\Bigg|_{z=e}}=\delta^i_j.
\end{equation}
Hence, by the Inverse Mapping Theorem $\mu$ can be restricted to a (possibly) smaller open set $N'_e\subset N_e$ which restricts to a diffeomorphism over $\{x\}\times N_e$ for $x\in N'_e$. Since such neighbourhood is still restricted normal, one may just denote it by the initial $N_e$. Therefore, $\mathcal{L}(N_e,\mu)$ is a differentiable local loop.
\end{proof}

\begin{definition}
Let $(M,\nabla)$ be affinely connected space. The differential local loop $\mathcal{L}(N_e,\mu)$ from Proposition \ref{ocara} is called a \textbf{geodesic loop} around the point $e\in M$.
\end{definition}
\begin{preremark}\upshape
From now on, the application $\mu$ is dropped and the juxtaposition product is employed. For simplicity, the notation $x\in\mathcal{L}_e$ is taken to mean that $x\in N_e$, whenever it is clear which geodesic loop is being considered.
\end{preremark}
\section{Algebraic Realizations}

 As previously seen, the geodesic loop is an algebraic structure which is defined via an affine connection over a manifold $M$. One may therefore analyze if information on the loop level may be related to the underlying geometry defined by the connection. The simplest case may be considered, namely an affinely connected Riemannian space $(M,g,\nabla)$, where $\nabla$ is a flat Levi-Civita connection \cite{kuuskpaal}. In that scenario, it is known that the geodesic curves are just straight lines in local coordinates.
 
 \begin{lemma}
 Let $(M,\nabla)$ be a flat affinely connected space and consider the geodesic loop $\mathcal{L}_e$ around $e\in M$. Supposing that $e=(e^1,\ldots,e^n)$ in local coordinates then for every $x,y\in \mathcal{L}_e$ there holds
\bege\label{affineloop}
(xy)^i=y^i+x^i-e^i.
\enge
\end{lemma}
\begin{proof}
Indeed, since the connection is flat, geodesics are given by straight lines in coordinates. Then, the geodesic between $e$ and $x$ is locally given by 
\begin{equation}\ga^i_x(t)=e^i+t(x^i-e^i),\end{equation}
and notice that indeed $\ga^i_x(t_0)=x^i$ if and only if $t_0=1$. The parallel transport of the vector $\ga'_x(0)=x-e$ to $y$ does not change coordinates, which in turn yields
\begin{equation}(xy)^i=\ga_y^i(1)=y^i+x^i-e^i,\end{equation}
as desired. 
\end{proof}
\begin{corollary}
Let $(M,\nabla)$ be a flat affinely connected space. Then, the geodesic loop $\mathcal{L}_e$ is an abelian group for all $e\in M$.
\end{corollary}
\begin{proof}
Let $x,y,z\in \mathcal{L}_e$. It suffices to prove that the geodesic loop is commutative and associative. Indeed, using the local coordinates one has
\begin{equation}
(xy)^i-(yx)^i=(y^i+x^i-e^i) - (x^i+y^i-e^i)=0.
\end{equation}
In addition,
\begin{equation}\begin{split}
     \begin{split}
        ((xy)z)^i-(x(yz))^i&=z^i+(xy)^i-e^i-(yz)^i-x^i+e^i\\
        &=z^i+x^i+y^i-e^i-z^i-y^i+e^i-x^i=0.
     \end{split}
\end{split}\end{equation}
Therefore, $\mathcal{L}_e$ is an abelian group.
\end{proof}
\begin{preremark}\upshape
All geodesic loops of flat vector spaces are isomorphic, since the product $\mu$ consists on the vector sum inherited from the space. In general, geodesic loops may be neither commutative nor associative. To investigate such properties, fix for now on a general affinely connected space $(M,\nabla)$ and consider the geodesic loop $\mathcal{L}_e$ around $e\in M$.
\end{preremark}
The so-called fundamental tensors may then be discussed, which shall be proven to be one of the most exceptional tools when analyzing geodesic loops. The following results can be found in \cite{3web,akivis151}. As seen before, if $x,y\in \mathcal{L}_e$, then one may perceive the equation $xy=\mu(x,y)$ by means of local coordinates centered at $e$. Namely
\begin{equation}
(xy)^i=\mu^i(x^1,\ldots,x^n,y^1,\ldots,y^n)=\mu^i(x,y).
\end{equation}
By construction, one has $\mu^i(x,0)=x^i$ and $\mu^i(0,y)=y^i$. Because of this, one can see that the Taylor expansion centered at $e=(0,0)$ has the special form 
\bege\label{aequacao}
(xy)^i=x^i+y^i+\lambda^i_{jk}x^jy^k+\frac{1}{2}(\mu^i_{jkl}x^jx^ky^l+\nu^i_{jkl}x^jy^ky^l)+o(\rho^3),
\enge
where 
\begin{equation}
\lambda^i_{jk}=\frac{\partial^2(x,y)^i}{\partial x^j\partial y^k}{\Big |}_{x=y=0}
\end{equation}
and
\begin{align}
    \mu^i_{jkl}=&\frac{\partial^3 (x,y)^i}{\partial x^j\partial x^k\partial y^l }{\Big |}_{x=y=0},\label{coisa}\\
    \nu^i_{jkl}=&\frac{\partial^3 (x,y)^i}{\partial x^j\partial y^k\partial y^l }{\Big |}_{x=y=0}.\label{coisa2}
\end{align}
These coefficients do not define tensors on $M$ (considering the local description in $\re^n$), since a change of coordinates does not factor through the derivatives in eqns (\ref{coisa}, \ref{coisa2}). In order to define one, let 
\begin{equation}\Lambda:N_p\times N_p\rightarrow N_p\end{equation}
be the map given in coordinates by
\begin{equation}
\Lambda^i(x,y)=\lambda^i_{jk}x^jy^k.
\end{equation}
Then, one may define the application
\begin{equation}A:N_p\times N_p\rightarrow N_p\end{equation}
given by the relation
\begin{equation}
A(x,y)=\frac{1}{2}\lp\Lambda(x,y)-\Lambda(y,x)\rp,
\end{equation}
which in turn has the local description
\begin{equation}
A^i(x,y)=\alpha^i_{jk}x^jy^k=\frac{1}{2}(\lambda^i_{jk}-\lambda^i_{kj})x^jy^k=\lambda^i_{[jk]}x^jy^k.
\end{equation}
Therefore, one may consider the symbols
\begin{equation}\label{ptf}
    \alpha^i_{jk}=\lambda_{[jk]}
\end{equation}
which by construction are anti-symmetric. Namely,
\begin{equation} \alpha^i_{jk}=-\alpha^i_{kj}.\end{equation}

In a similar way, consider the following maps locally given by
\begin{equation}\begin{split}
M^i(x,y,z)&=\mu^i_{jkl}x^jy^kz^l,\\
N^i(x,y,z)&=\nu^i_{jkl}x^jy^kz^l.
\end{split}\end{equation}
These allow one to construct the following application:
\bege\label{celu}
B(x,y,z)=2\lp N(x,y,z)-M(x,y,z)+\Lambda(x,\Lambda(y,z))-\Lambda(\Lambda(x,y),z)\rp,
\enge
which in coordinates is given by
\begin{equation}\label{sft}
B^i(x,y,z)=\beta^i_{jkl}x^jy^kz^l,
\end{equation}
where by eqn (\ref{celu}) yields
\bege\label{secondf}
\beta^i_{jkl}=2(\nu^i_{jkl}-\mu^i_{jkl}+\lambda^m_{kl}\lambda^i_{jm}-\lambda^m_{jk}\lambda^i_{ml}).
\enge
\begin{definition} Let $(M,\nabla)$ be an affinely connected space. The tensors $\alpha^i_{jk}$ and $\beta^i_{jkl}$ defined by eqns (\ref{ptf}, \ref{sft}) are  respectively called \textbf{first} and \textbf{second fundamental tensors} of the geodesic loop $\mathcal{L}_e$.
\end{definition}
\begin{preremark}\upshape
Notice that alternating eqn (\ref{celu}) gives
\begin{align}\label{ladan}
    \begin{split}
        B(u,v,w)+B(v,w,u)&+B(w,u,v)-B(v,u,w)-B(w,v,u)-B(u,w,v)\\
        =& A(u,A(v,w))+A(v,A(w,u))+(w,A(u,v)),
    \end{split}
\end{align}
which can be seen using the symmetries
\begin{equation}\begin{split}
\mu^i_{jkl}&=\mu^i_{jlk},\\
\nu^i_{jkl}&=\nu^i_{jlk}.
\end{split}\end{equation}
Then, in index notation, eqn (\ref{ladan}) reads
\bege\label{cachorrao}
\beta^i_{[jkl]}=\alpha^m_{[jk}\alpha^i_{l]m]},
\enge
which is called the \textbf{generalized Jacobi Identity}. The relation with the usual Jacobi identity (\ref{jacobi}) shall be shortly unveiled.
\end{preremark}
These tensors' symbols are intimately related to the commutativity and associativity of the geodesic loop $\mathcal{L}_e$, as follows. As mentioned before, if $x\in \mathcal{L}_e$ then $x^{-1}_L$ and $x^{-1}_R$ are the left and right inverses of $x$, respectively. Also, one may define the left and right commutators of the geodesic loop product, respectively given by
\begin{align}
    \begin{split}
        \alpha_L(x,y)=(xy)^{-1}_L(xy),\\
        \alpha_R(x,y)=(xy)(xy)^{-1}_R.
    \end{split}
\end{align}
\begin{proposition}\label{comute}
Up to second order terms, the left and right commutators of the geodesic loop $\mathcal{L}_e$ are equal and determined by the first fundamental tensor $\alpha^i_{jk}$.
\end{proposition}
\begin{proof}
From the equalities $z^{-1}_Lz=e=zz^{-1}_R$, one can see that in local coordinates centered at $e$ there holds
\begin{align}\label{gatofraj}
    \begin{split}
    (z^{-1}_L)^i=-z^i+\lambda^i_{jk}z^jz^k+\sigma^i_{jkl}z^jz^kz^l+o(\rho^3),\\
    (z^{-1}_R)^i=-z^i+\lambda^i_{jk}z^jz^k+\tau^i_{jkl}z^jz^kz^l+o(\rho^3),
    \end{split}
\end{align}
where
\begin{align}
    \sigma^i_{jkl}=-\frac{1}{2}(\mu^i_{(jkl)}-\nu^i_{(jkl)})-\lambda^i_{p(j}\lambda^p_{kl)},\\
    \tau^i_{jkl}=\frac{1}{2}(\mu^i_{(jkl)}-\nu^i_{(jkl)})-\lambda^i_{(j|p|}\lambda^p_{kl)}.
\end{align}
As usual, parentheses denote symmetrization in the indexes. Notice that the difference between the left and right inverses begin to appear only in the third-order term of the Taylor expansion. Using eqn (\ref{gatofraj}), one can then calculate
\bege
((xy)^{-1}_L)^i=-x^i-y^i-\lambda^i_{jk}x^jy^k+\lambda^i_{jk}(x^j+y^k)(x^k+y^j)+o(\rho^2),
\enge
which in turn with eqn (\ref{aequacao}) yields
\bege
\alpha^i_L(x,y)=2\alpha^i_{jk}x^jy^k+o(\rho^2).
\enge
Following the same reasoning it is possible to obtain
\bege
\alpha^i_R(x,y)=2\alpha^i_{jk}x^jy^k+o(\rho^2),
\enge
which concludes the proof.
\end{proof}

In order to investigate associativity one may define the left and right associators, namely
\begin{align}
    \begin{split}
        \beta_L(x,y,z)=(x(yz))^{-1}_L((xy)z)),\\
        \beta_R(x,y,z)=((xy)z)(x(yz))^{-1}_R.
    \end{split}
\end{align}
Using a similar procedure as in Proposition \ref{comute}, one may prove the
\begin{proposition}\label{associate}
Up to second order terms, the left and right associators of the geodesic loop $\mathcal{L}_e$ are equal and determined by the second fundamental tensor $\beta^i_{jkl}$.
\end{proposition}
\begin{preremark}\upshape
In conclusion, Propositions \ref{comute} and \ref{associate} assert that $\alpha^i_{jk}$ and $\beta^i_{jkl}$ are rough approximations for the failure of commutativity and associativity in the geodesic loop $\mathcal{L}_e$. It directly follows that
\end{preremark}
\begin{corollary}
Let $(M,\nabla)$ be an affinely connected space and consider the geodesic loop $\mathcal{L}_e$ around $e\in M$. The following statements hold:
\begin{enumerate}
    \item If $\mathcal{L}_e$ is commutative then the first fundamental tensor $\alpha^i_{jk}$ vanishes.
    \item If $\mathcal{L}_e$ is associative then the second fundamental tensor $\beta^i_{jkl}$ vanishes.
\end{enumerate}
\end{corollary}

\begin{preremark}\upshape
The fundamental tensors of a geodesic loop may be further used in order to endow vector spaces with multilinear operations. The notion of $W$-algebras, which are vector spaces equipped with certain kinds of multilinear operations are now briefly presented.
\begin{definition}
Let $V$ be a vector space and let the two multilinear operations 
\begin{equation}\begin{split}
[\cdot,\cdot]:&V\times V\rightarrow V,\\ 
[\cdot,\cdot,\cdot]:& V\times V\times V\rightarrow V,
\end{split}\end{equation}
with $[\cdot,\cdot]$ anti-symmetric, be related by
\begin{equation}\label{jajacobi}
\begin{split}
[X,[Y,Z]]+[Y,[Z,X]]+[Z,[X,Y]]=[X,Y,Z]&+[Y,Z,X]+[Z,X,Y]+[Z,X,Y]\\
&-[X,Z,Y]-[Z,Y,X]-[Y,X,Z].
\end{split}
\end{equation}
Then, the triple $(V,[\cdot,\cdot],[\cdot,\cdot,\cdot])$ is called a \textbf{W-algebra} over $V$. 
\end{definition}

Given any affinely connected space $(M,\nabla)$ and a geodesic loop $\mathcal{L}_e$ around $e\in M$ one is able to define a $W$-algebra over an $n$-dimensional vector space by means of the fundamental tensors of $\mathcal{L}_e$. Namely, the \textbf{geodesic commutator} 
\begin{equation}[\cdot,\cdot]:V\times V\rightarrow V\end{equation} may be defined by the relation
\begin{equation}
[u,v]=2\alpha^i_{jk}u^jv^ke_i,
\end{equation}
where $\{e_1,\ldots,e_n\}$ is a basis for $V$ for which $u=u^je_j$ and $v=v^ke_k$. It follows directly from the anti-symmetry of the first fundamental tensor that this operation is also anti-symmetric. 

Similarly, one may use the second fundamental tensor in order to define the \textbf{geodesic associator}, given by
\begin{equation}
    [u,v,w]=\beta^i_{jkl}u^jv^kw^l,
\end{equation}
where $w=w^ke_k$. It follows directly from eqn (\ref{cachorrao}) that the relation (\ref{jajacobi}) is satisfied. Notice that when the right-hand side of the same equation vanishes, one is left with the usual Jacobi identity. In that case, the geodesic commutator $[\cdot,\cdot]$ satisfies all the properties of the Lie bracket and therefore $(V,[\cdot,\cdot])$ is isomorphic to a Lie algebra. The more natural vector space over which one may consider this construction is the tangent space $T_e M$ and then clearly $\lp T_eM,[\cdot,\cdot],[\cdot,\cdot,\cdot]\rp$ is a $W$-algebra.
\end{preremark}
\begin{preremark}\upshape
If one is already given a more general algebra $(A,*)$ over a vector space $A$, it is possible to define the commutator and associator operations, namely
\begin{equation}\begin{split}
[u,v]&=u*v-v*u\\
[u,v,w]&=(u*v)*w-u*(v*w).
\end{split}\end{equation}
Whenever these operations satisfy the generalized Jacobi identity, the triple $(A,[\cdot,\cdot],[\cdot,\cdot,\cdot])$ is called the W-algebra associated to $(A,*)$. 
\end{preremark}

\section{Geometric Realizations}
One may now proceed to investigate connections between the developed algebraic properties of the geodesic loop with the geometry of its underlying space. More explicitly, one may ask at what extent does commutativity and associativity of the product heretofore discussed intervene in geometric structures.

Theorem (\ref{kika}), proved by Kikkawa \cite{kikkawa}, is a direct application of the geodesic loop construction. It shows that one can indeed expect that the algebraic information about the geodesic loop would give relevant information about the connection $\nabla$.

\begin{definition}
A loop $(A,*)$ is called \textbf{left diassociative} if, for every $x,y\in A$,  
\begin{equation}x*(x*y)=(x*x)*y.\end{equation}
\end{definition}

\begin{theorem}\label{kika}
Let $(M,\nabla)$ be an affinely connected space and suppose $\nabla$ is torsionless. If the geodesic loop $\mathcal{L}_e$ around $e\in M$ is left diassociative, then the curvature tensor $R$ vanishes at the point $e$.
\end{theorem}
\begin{proof}
Let $X_e\in T_eM$ and consider the vector field 
\begin{equation}X(p)=T_{e,p}(X_e),\end{equation}
for each $p\in \mathcal{L}_e$ called the vector field adapted to $X_e$. One may then consider $x(t)$ to be the geodesic through $e$ with $x'(t)=X_{x(t)}$. By construction, the geodesic curve $y(t)$ through a point $y\in \mathcal{L}_e$ with tangent vector equal to $X_y$ at $y$ is given by the geodesic loop product 
\begin{equation}y(t)=x(t)y.\end{equation}
Notice that each geodesic arc through $e$ gives rise to a $1$-parameter local subgroup, that is, there holds $x(t^1)x(t^2)=x(t^1+t^2)$. Then, by applying the diassociative hypothesis,
\begin{equation} y(s+t)=x(s+t)y=(x(s)x(t))y=x(s)(x(t)y).\end{equation}
Therefore, there holds $y'(t)=X_{y(t)}$ and since $y$ is arbitrary in $\mathcal{L}_e$, it follows that all trajectories of the vector field $X$ are geodesic arcs. It follows that
\bege\label{prim}
\nabla_X X=0
\enge
all over $\mathcal{L}_e$ for every adapted field $X$.

One may then take two vectors $X_e$ and $Y_e$ tangent to $e$ and consider their adapted fields $X$ and $Y$. It follows that $X+Y$ is adapted to $X_e+Y_e$, and therefore it satisfies eqn (\ref{prim}). Then, there holds
\bege\label{segund}
\nabla_X Y+\nabla_Y X=0.
\enge
It follows that
\begin{align}\label{terc}
    \begin{split}
        R(X,Y)X=&\nabla_X\nabla_YX-\nabla_Y\nabla_XX-\nabla_{[X,Y]}X\\
        =&-\nabla_X\nabla_XY-\nabla_Y\nabla_XX-\nabla_{[X,Y]}X.
    \end{split}
\end{align}
By assumption one has a torsionless affine connection, which implies that 
\bege\label{quart}
T(X,Y)=\nabla_XY-\nabla_YX-[X,Y]=0,
\enge
where $T$ is the torsion tensor. Using eqn (\ref{segund}) one gets
\begin{equation}\nabla_XY=\frac{1}{2}[X,Y].\end{equation}
From the above equation and using that $Y$ is adapted to $Y_e$ there holds $\nabla_XY=0$ all over the trajectory $x(t)$ of $X$. Then,
\bege\label{quint}
[X,Y]_e=0
\enge
and
\bege\label{sext}
(\nabla_X\nabla_X Y)_e=0.
\enge
Using eqns (\ref{prim}, \ref{quint}, \ref{sext}) it follows that
\bege\label{setim}
R_e(X_e,Y_e)X_e=-(\nabla_X\nabla_XY)_e-(\nabla_Y\nabla_XX)_e-(\nabla_{[X,Y]}X)_e=0.
\enge
Therefore,
\bege 
R_e(X_e,Y_e)Z_e=-R_e(Z_e,Y_e)X_e,\;\text{for every}\;X_e,Y_e,Z_e\in T_e M,
\enge 
which comes from the linear expansion of $R_e(X_e+Z_e,Y_e)(X_e+Z_e)=0$. Since $T=0$, one can then use the first Riemannian Bianchi identity, which can be further modified by the anticommutation relation in the first two entries, yielding
\begin{align}\label{octa}
\begin{split}
0&=R_e(X_e,Y_e)Z_e+R_e(Z_e,X_e)Y_e-R_e(Z_e,Y_e)X_e\\
&=R_e(X_e,Y_e)Z_e-R_e(Y_e,X_e)Z_e+R_e(X_e,Y_e)Z_e\\
&=3R_e(X_e,Y_e)Z_e,
\end{split}
\end{align}
which finally gives
\bege
R_e(X_e,Y_e)Z_e=0,\;\;\text{for all}\;\;X_e,Y_e,Z_e\in T_e M.
\enge
Therefore, $R_e=0$.
\end{proof}

One of the most powerful results on geodesic loops so far may now be presented. To that end, however, normal coordinates will be required, so that in this proof one is obliged to consider a metric $g$ over $M$. Hence, an affinely connected Riemannian spaces $(M,g,\nabla)$ must be at hand. Additionally consider respectively $T,S$ and $R$ the torsion, contorsion and curvature tensors for the connection $\nabla$. 

In the light of all that was previously discussed, the following result gives a local relation between the fundamental tensors of the geodesic loops (algebraic information) and the geometry of the underlying space. Some complementary results are considered in what follows, before presenting the main theorem due to Akivis \cite{akivis152}.

Consider an affinely connected Riemannian space $(M,g,\nabla)$ and let $\mathcal{L}_e$ be a geodesic loop around the point $e\in M$. The local equation of a geodesic $\gamma(t)$ through the restricted normal neighbourhood $N_e$ is given as the solution of the differential equation
\bege\label{diffeq}
\Ddot{\gamma}^i(t)+\Gamma^i_{jk}(\gamma(t))\dot{\gamma}^j(t)\dot{\gamma}^k(t)=0.
\enge
In normal coordinates $(N_e;x^1,\ldots,x^n)$, the geodesics through $e$ are straight lines, so that a point in $N_e$ is connected to $e$ by a geodesic, which is locally given by
\begin{equation}
\ga(t)=(\lambda^it),
\end{equation}
for some fixed $\lambda_i\in\re$. In particular, eqn (\ref{diffeq}) reads
\bege\label{eq7}
\Gamma^i_{jk}(\gamma(t))\lambda^j\lambda^k=0.
\enge
Now, considering $e$ itself, eqn (\ref{eq7}) holds for every choice of $\lambda^i\in\re$, so that taking $\lambda^i=\delta^i_j$, it yields
\bege\label{cann} 
\Gamma^i_{jj}(e)=0.
\enge
Then, using (\ref{cann}) and taking now $\lambda^i=\delta^i_j+\delta^i_k$ one has
\bege\label{va}
\Gamma^i_{(jk)}(e)=0.
\enge
In order to simplify the notation, let
\begin{equation}\mathring{\Gamma}^i_{jk}=\Gamma^i_{jk}(e).\end{equation}
Since, by definition,
\begin{equation}
    \Gamma^i_{[jk]}=\frac{1}{2}T^i_{jk},
    \end{equation}
or using the contorsion tensor $S$,
\begin{equation}\label{contee}
    \Gamma^i_{[jk]}=-S^i_{jk},
\end{equation}
adding together eqns (\ref{va}, \ref{contee}) and evaluating at $e$, it comes
\bege\label{eq9}
\mathring{\Gamma}^i_{jk}=-S^i_{jk}(e)=-\mathring{S}^i_{jk}.
\enge
\begin{preremark}\upshape
As a quick note, see how eqn (\ref{eq9}) shows that considering normal coordinates around a point $e\in M$ in an affinely connected space $(M,g,\nabla)$, then the Christoffel symbols at the point $e$ are precisely given by the contorsion tensor. In the case $\nabla=\nabla^g$, the contorsion vanishes and, therefore, so do the Christoffel symbols.
\end{preremark}

Now, let us once more analyze the geodesic loop construction. Take $x,y\in N_e$ and respectively denote by $\ga_x(t)$ and $\ga(s)$ the geodesics joining $e$ to $x$ and to $y$. Also, let $\xi=(\xi^i)$ be the parallel transport of the vector $\gamma_x'(0)$ to the point $y$ through $\gamma$. As before, denote by $\gamma_y(t)$ the geodesic with initial conditions
\begin{equation}\gamma_y^i(0)=y^i,\;\;\ \dot{\gamma}_y^i(t_0)=\xi^i.\end{equation}
Such geodesic at $t=1$ gives the expression of the geodesic loop product $xy=\gamma_y(1)$ by the construction depicted before. Since $\ga_y(t)$ is also a geodesic, then one may differentiate eqn (\ref{diffeq}) yielding
\begin{equation}\begin{split}
    \Ddot{\gamma}^i_y(t)&=-\Gamma^i_{jk}\dot{\gamma}_y^j(t)\dot{\gamma}^k_y(t),\\
    \dddot{\gamma}^i_y(t)&=-(\Gamma^i_{jk,l}-\Gamma^i_{mj}\Gamma^m_{kl}-\Gamma^i_{jm}\Gamma^m_{kl})\dot{\gamma}^j_y(t)\dot{\gamma}^k_y(t)\dot{\gamma}^k_y(t).
\end{split}\end{equation}
Evaluating at the point $y$ comes
\begin{equation}\begin{split}
    \Ddot{\gamma}^i_y(0)&=-\Gamma^i_{jk}(y)\xi^j\xi^k,\\
    \dddot{\gamma}^i_y(0)&=-(\Gamma^i_{jk,l}-\Gamma^i_{mj}\Gamma^m_{kl}-\Gamma^i_{jm}\Gamma^m_{kl})(y)\xi^j\xi^k\xi^l.
\end{split}\end{equation}
Now one can Taylor expand at $y$ in order to see that
\bege\label{taylorz}
\gamma_y^i(t)=y^i+\xi^it-\frac{1}{2}(\Gamma^i_{jk})(y)\xi^j\xi^kt^2-\frac{1}{6}(\Gamma^i_{jk,l}-\Gamma^i_{mj}\Gamma^m_{kl}-\Gamma^i_{jm}\Gamma^m_{kl})(y)\xi^j\xi^k\xi^lt^3+o(t^3).
\enge
Then, using eqn (\ref{eq9}) and supposing that the geodesic $\gamma$ from $e$ to $y$ is given in normal coordinates by $\ga^i(s)=y^is$ (since $\gamma(1)=y$), one can Taylor expand the coefficients in eqn (\ref{taylorz}) around $e$ to get
\begin{equation}
    \Gamma^i_{jk}(\ga(s))=-\mathring{S}^i_{jk}+\mathring{\Gamma}^i_{jk,l}y^ls+o(s),
\end{equation}
which produces in eqn (\ref{taylorz}), with respect to the variable $\rho=\sqrt{t^2+s^2}$, the expansion
\bege\label{taylor2}
\gamma_y^i(t)=y^i+\xi^it-\frac{1}{2}(\mathring{\Gamma}^i_{jk,l})\xi^j\xi^ky^lt^2s+o(\rho^3).
\enge
One may now rewrite eqn (\ref{taylor2}) taking into account the Taylor expansion of the coordinates of the vector $\xi$. If one takes $X(s)=\tau_{e,\ga(s)}(\gamma_x'(0))$ as before, then the parallel transport equation reads
\bege
\dot{X}^i(s)=-\Gamma^i_{jk}X^j \dot{\gamma}^k(s)
\enge
It then follows that
\bege
\Ddot{X}^i(s)=(-\Gamma^i_{jk,l}-\Gamma^i_{mk}\Gamma^m_{jl})(\gamma(s))X^j \dot{\gamma}^k(s)\dot{\gamma}^l(s).
\enge

This yields, similarly as before, the expansion around $e$ ($s=0$) given by
\bege\label{carainha}
    X^i(s)=x^i+\mathring{S}^i_{jk}x^jy^ks-\frac{1}{2}(-\mathring{\Gamma}^i_{jk,l}-\mathring{S}^i_{mk}\mathring{S}^m_{jl})x^jy^ky^ls^2+o(s^2).
\enge
One can then prove the
\begin{theorem}\label{funda}
Let $(M,g,\nabla)$ be an affinely connected Riemannian space and let $T^i_{jk}$ and $R^i_{jkl}$ denote the torsion and curvature tensors of this space in normal coordinates. If $\alpha^i_{jk}$ and $\beta^i_{jkl}$ are the fundamental tensors of a geodesic loop $\mathcal{L}_e$, then
 \[
    \left\{
                \begin{array}{ll}
                 2\alpha^i_{jk}=-T^i_{jk},\\
    4\beta^i_{jkl}=-\nabla_lT^i_{jkl}-R^i_{jkl}.
                \end{array}
              \right.
  \]
\end{theorem}
\begin{proof}
Set $s=1$ in eqn (\ref{carainha}) and insert it in eqn (\ref{taylor2}). Then, setting $t=1$ and observing that
\begin{equation}\begin{split}
     X^i(1)=\xi^i,\\
     \gamma_y^i(1)=(xy)^i,
             \end{split}\end{equation}    
it follows that
\bege
(xy)^i=\gamma_y^i(1)=x^i+y^i+\mathring{S}^i_{jk}x^jy^k-\frac{1}{2}(\mathring{\Gamma}^i_{jk,l}-\mathring{S}^i_{mk}\mathring{S}^m_{jl})x^jy^ky^l-\frac{1}{2}\mathring{\Gamma}^i_{jk,l}x^jx^ky^l+o(\rho^3),
\enge
where $\rho=\text{max}(|x^i|,|y^i|)$. But since eqn (\ref{aequacao}) reads
\bege
z^i=x^i+y^i+\lambda^i_{jk}x^jy^k+\frac{1}{2}(\mu^i_{jkl}x^jx^ky^l+\nu^i_{jkl}x^jy^ky^l)+o(\rho^3),
\enge
one gets the relations
\begin{equation}\label{hagg}
\begin{split}
    \lambda^i_{jk}&=\mathring{S}^i_{jk},\;\;\;\;\mu^i_{jkl}=-\mathring{\Gamma}^i_{jk,l},\\
    \nu^i_{jkl}&=-(\mathring{\Gamma}^i_{jk,l}-\mathring{S}^i_{mk}\mathring{S}^m_{jl}).
    \end{split}
\end{equation}
Moreover, since $\mu^i_{jkl}=\mu^i_{kjl}$ and $\nu^i_{jkl}=\nu^i_{kjl}$, there holds
\bege
\mu^i_{jkl}=-\mathring{\Gamma}^i_{(jk),l},\;\;\nu^i_{jkl}=-(\mathring{\Gamma}^i_{(j(k,l)}-\mathring{S}^i_{m(k}\mathring{S}^m_{|j|l)}).
\enge
One can then calculate the fundamental tensors $\alpha^i_{jk}$ and $\beta^i_{jkl}$ in terms of eqns (\ref{hagg}). 

For the first fundamental tensor, using the anti-symmetry of the torsion tensor one can see that
\bege
\alpha^i_{jk}=\mathring{S}^i_{[jk]}=\mathring{S}^i_{jk}=-\frac{1}{2}T^i_{jk}.
\enge
On the other hand, for the second fundamental tensor there holds by definition
\bege
\begin{split}
-2\beta^i_{jkl}&=-\mathring{\Gamma}^i_{(jk),l}+\mathring{\Gamma}^i_{j(k,l)}-\mathring{S}^i_{m(k}\mathring{S}^m_{|j|l)}+\mathring{S}^m_{jk}\mathring{S}^i_{ml}-\mathring{S}^m_{kl}\mathring{S}^i_{jm}\\
&=\frac{1}{2}(\mathring{\Gamma}^i_{jl,k}-\mathring{\Gamma}^i_{kj,l})-\frac{1}{2}\mathring{S}^i_{mk}\mathring{S}^m_{jl}-\frac{1}{2}\mathring{S}^m_{ml}\mathring{S}^m_{jk}+\mathring{S}^m_{jk}\mathring{S}^i_{ml}-\mathring{S}^m_{kl}\mathring{S}^i_{jm}\\
&=\frac{1}{2}(\mathring{\Gamma}^i_{jl,k}-\mathring{\Gamma}^i_{kj,l})+\mathring{S}^m_{j[k}\mathring{S}^i_{|m|l]}-\mathring{S}^m_{kl}\mathring{S}^i_{jm}.
\end{split}
\enge
Now, one may use the local relations 
\begin{equation}\begin{split}
    S^i_{jk}&=-\frac{1}{2}T^i_{jk}=-\Gamma^i_{[jk]},\\
    \frac{1}{2}R^i_{jkl}&=-\Gamma^i_{j[k,l]}-\Gamma^m_{j[k}\Gamma^i_{|m|l]}.
\end{split}\end{equation}
Differentiating the first equation at $e$ yields
\begin{equation}
\mathring{\Gamma}^i_{[jk],l}=-\mathring{S}^i_{jk,l},
\end{equation}
whereas the second one reads
\bege
\mathring{\Gamma}^i_{j[k,l]}=-\frac{1}{2}\mathring{R}^i_{jkl}-\mathring{S}^m_{j[k}\mathring{S}^i_{|m|l]},
\enge
where $\mathring{R}^i_{jkl}=R^i_{jkl}(e)$. Subtracting, it comes
\bege
\frac{1}{2}(\mathring{\Gamma}^i_{jl,k}-\mathring{\Gamma}^i_{kj,l})=\frac{1}{2}\mathring{R}^i_{jkl}+\mathring{S}^m_{j[k}\mathring{S}^i_{|m|l]}-\mathring{S}^i_{jk,l}
\enge
and it follows that 
\bege
-2\beta^i_{jkl}=\frac{1}{2}R^i_{jkl}-\mathring{S}^i_{jk,l}-3\mathring{S}^m_{[jk}\mathring{S}^i_{l]m]}.
\enge
Further on, since
\bege
\nabla_lS^i_{jk}=S^i_{jk,l}-S^i_{mk}\Gamma^m_{lj}-S^i_{jm}\Gamma^m_{lk}+S^m_{jk}\Gamma^i_{lm},
\enge
then it is possible to see that at $e$ the expression above has the form
\bege
\nabla_lS^i_{jk}(e)=\mathring{S}^i_{jk,l}+3\mathring{S}^m_{[jk}\mathring{S}^i_{l]m},
\enge
which finally implies that
\bege
    \beta^i_{jkl}=\frac{1}{2}(\nabla_lS^i_{jk})_e-\frac{1}{4}\mathring{R}^i_{jkl}.
\enge
Since this construction can be made upon any point $e$ in the manifold, it finally reads
\begin{align}
\alpha^i_{jk}&=S^i_{jk},\label{aequacao1}\\
\beta^i_{jkl}&=\frac{1}{2}\nabla_lS^i_{jk}-\frac{1}{4}R^i_{jkl},
\end{align}
which gives the desired result.
\end{proof}

\chapter{Spontaneous Compactification}

The search for an unifying theory of everything has been central in theoretical physics for years now. The so-called supergravity, which encompasses general relativity and supersymmetry, may offer a possibility to perceive such objective in a theoretical view. In such configuration, the maximal dimension for which one can balance bosonic and fermionic degrees of freedom with highest spin is eleven \cite{loginov3}. Therefore, one can consider spontaneous compactifications in such theory, that is, solutions of the 11-dimensional equation of motion over a space which is a product of a 4-dimensional spacetime and a compact 7-dimensional space.

A mechanism to achieve such goal is to consider the so-called Kaluza-Klein spontaneous compactification, as follows \cite{loginov2,loginov3}. The ground state is a product $M_4\times K$, where $M_4$ is a maximally symmetric 4-dimensional space (de Sitter space, anti-de Sitter space or Minkowski spacetime) and $K$ is a compact manifold called the \textbf{internal space}, which is assumed to be an Einstein space. Maximally symmetric space here means that no point in such space can be distinguished one from another, apart from the information of it being either time-like, space-like or light-like. A metric $g_{MN}$ on $M_4\times K$ is considered in the form
\bege g_{MN}=
\begin{pmatrix}
g_{\mu\nu}&0\\0&g_{mn}
\end{pmatrix},
\enge
where $g_{\mu\nu}$ is the metric on $M_4$ and $g_{mn}$ on $K$ (greek letter denote spacetime indices whereas roman letters denote the internal space indices). Such representation of $g_{MN}$ is compatible with the Einstein equation
\bege
R_{MN}-\frac{1}{2}Rg_{MN}=T_{MN}-\Lambda g_{MN},
\enge
where the energy-momentum tensor of matter fields is given by
\begin{align}
    \begin{split}
        T_{\mu\nu}&=k_1g_{\mu\nu},\\
        T_{mn}&=k_2g_{mn}.
    \end{split}
\end{align}
One may consider $(M_4,g_{\mu\nu})$ to be the 4-dimensional Riemann spacetime of signature $(+,+,+,-)$. The Christoffel symbols shall be taken as
\begin{equation}\label{choriso}
\Gamma_{ijk}=\mathring{\Gamma}_{ijk}+S_{ijk},
\end{equation}
where $\mathring{\Gamma}^i_{jk}$ are the Christoffel symbols for Levi-Civita connection taken with respect to $g_{MN}$. It is also assumed that $S_{ijk}$ is a fully anti-symmetric tensor, and therefore $\nabla$ is metric-compatible and its geodesics are the same as the ones for the Levi-Civita connection. Notice that the sign in eqn (\ref{choriso}) is the opposite as ours when defining the contorsion $S$.

One can then analyze the Freund-Rubin-Englert mechanism \cite{loginov2ref7,loginov2ref8} of spontaneous compactification for $d=11$ supergravity. The equations of motion of this theory, which encompass the Einstein field equations and equation for the anti-symmetric gauge field strength $F$, have the form \cite{loginov2ref9}
\begin{align}
    R_{MN}-\frac{1}{2}g_{MN}R=12\lp 8F_{MPQR}F_N^{\;\;PQR}-g_{MN}F_{SPQR}F^{SPQR}\rp,\\
    F^{MNPQ}_{\quad\quad\;\;,M}=-\frac{\sqrt{2}}{24}\varepsilon^{NPQM_1\ldots M_8}F_{M_1M_2M_3M_4}F_{M_5M_6M_7M_8},
\end{align}
where $\varepsilon^{M_1\ldots M_s}$ is a fully anti-symmetric covariant constant tensor and $\varepsilon_{1\ldots s}=\Vert g\Vert^{\frac{1}{2}}$. The Freund-Rubin solution \cite{loginov2ref7} for this mechanism is given by
\begin{equation}\label{olokinho}
F_{\mu\nu\sigma\lambda}=\rho\varepsilon_{\mu\nu\sigma\lambda},
\end{equation}
where $\rho$ is a real number and all other $F_{MNPQ}$ vanish (namely, the ones in the internal space). One can nonetheless obtain other solutions with non-vanishing components in the internal space. Such solutions, such as the Englert solution \cite{loginov2ref8}, were first constructed in the 7-sphere $S^7$ with torsion. They read
\begin{align}
F_{\mu\nu\sigma\lambda}&=\rho\varepsilon_{\mu\nu\sigma\lambda},\\
F_{mnpq}&=\lambda\partial_{[q}S_{mnp]}.
\end{align}
 One may analyze, using the tools developed so far, the restrictions such solutions force over the space $M_4\times K$ and its geometry. 
  
   As seen in \cite{loginov2}, this may be done by considering geodesic loops around points $e\in M_4\times K$ and analyzing its algebraic information in the light of the last section. As said before, in \cite{loginov2,loginov3} the contorsion is taken with the opposite sign, so that these changes turn Theorem \ref{funda} equations into
  \begin{align}
      \alpha^i_{jk}&=-S^i_{jk},\label{fundlog1}\\
      4\beta^i_{jkl}&=-2\nabla_l S^i_{jk}-R^i_{jkl}\label{fundlog2}.
  \end{align}
Besides, the generalized Bianchi identities read
  \begin{align}
      R^i_{[jkl]}+2\nabla_{[j}S^i_{kl]}+4S^m_{[jk}S^i_{l]m}=0\\
      \nabla_{[k}R^{ij}_{lm]}-2R^{ij}_{n[k}S^n_{lm]}=0.
  \end{align}

  The first constraint to a geodesic loop in the previously set background is presented in this section, as seen in \cite{loginov2}. Namely, one may see that geodesic loops around points $e\in M_4\times K$ must be nonassociative in order to guarantee the Einstein space property for $M_4\times K$.
  
  Indeed, if for every $e\in M_4$ or $K$ its geodesic loop $\mathcal{L}_e$ is associative, then since $\beta=0$ the generalized Jacobi identity reads
   \begin{equation}\alpha^m_{[jk}\alpha^i_{l]m}=0\end{equation}
   so that the $W$-algebra defined over the tangent space of $e$ is isomorphic to a compact Lie algebra. Such algebras may be classified depending on the underlying space ($M_4$ or $K$) and then the first fundamental tensor $\alpha^i_{jk}$ and therefore the contorsion $S^i_{jk}$ are determined by the structure constants of these algebras. Moreover, in that case, it follows from eqn (\ref{fundlog2}) and from the total anti-symmetry of the contorsion that
   \begin{equation}
       R_{ijkl}=-2S_{ijk,l}=-2S_{[ijk],l},
   \end{equation}
   and since the curvature is anti-symmetric in the last two indices, it follows that $R_{ijkl}$ is totally anti-symmetric. Then, writing
   
   \begin{equation}\label{log39}
       R^i_{jkl}=\mathring{R}^i_{jkl}+(S^i_{jl,k}-S^i_{jk,l}+S^m_{jl}S^i_{mk}-S^m_{jk}S^i_{ml}),
   \end{equation}
   where $\mathring{R}^i_{jkl}$ is the curvature with respect to the Levi-Civita connection, it follows that
   \begin{equation}\label{rhs}
       \mathring{R}_{ij}=S^m_{ik}S^k_{mj},
   \end{equation}
   so that it is possible to show that the space $M_4\times K$ cannot be an Einstein space unless the metric $g$ is degenerate \cite{loginov2}. This is because the right-hand side of eqn (\ref{rhs}) depends only on the Lie algebras' structure constants, so that a case by case inspection shows that it indeed vanishes. Therefore, the geodesic loop for this space cannot be associative. Furthermore, it can be seen that the Freund-Rubin solution
   \begin{equation}
       F_{\mu\nu\sigma\lambda}=\rho\varepsilon_{\mu\nu\sigma\lambda}
   \end{equation}
   does not impose restrictions on the spacetime $M_4$ and neither do the geodesic loop relations in that case. One may therefore analyze the Englert solution.

 \section{Englert solutions}

Since the geodesic loop $\mathcal{L}_e$ around $e\in M$ must be nonassociative, one might look for examples of such structure. One may consider the octonion algebra $\oct$ \cite{baez} which is a real division algebra over $\re^8$ with canonical basis $\{1,e_1,\ldots,e_7\}$, such that
\bege\label{octprod}
e_j\circ e_k=-\delta_{jk}+c^i_{\;jk}e_i,
\enge
where $\circ$ is the \textbf{octonion product} and the structure constants $c_{ijk}$ are totally anti-symmetric and equal to the unity for the cycles
\begin{equation}\label{cycles}
(ijk)=(123),\;(145),\;(167),\;(246),\;(275),\;(374),\;(365).
\end{equation}
The basis elements $\{e_1,\ldots,e_7\}$ are called the imaginary units and are easily seen to satisfy $e_i^2=-1$. Each choice of cycles in (\ref{cycles}) yields a different
\begin{wrapfigure}{r}{0.45\textwidth}
\centering
\includegraphics[width=0.42\textwidth]{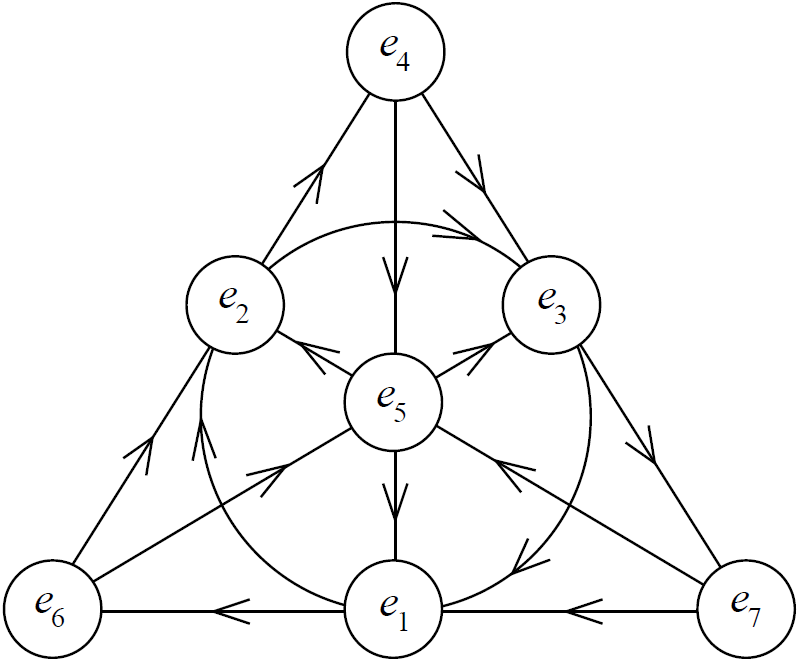}
\caption{The projective Fano plane. \cite{baez}}
\label{figg}
\end{wrapfigure}
(but isomorphic) octonion product. The octonion algebra is not associative but alternative, meaning that every 2-dimensional subspace is associative and is therefore endowed with a group structure. Equivalently, the associator
\bege
[x,y,z]=(x\circ y)\circ z-x\circ (y\circ z)
\enge
is totally anti-symmetric in $x,y,z\in\oct$. The octonion product 
may also be condensed in the so called Fano plane in Fig. \ref{figg}. Therein, each triple of basis elements determines a quaternionic-like product given by the direction of the arrows. As an example, the line segment containing the triple $\{e_6,e_2,e_4\}$ encodes the octonion product by the relation
 \begin{equation}
     e_6\circ e_2=e_4,\;\;\;\;e_2\circ e_4=e_6,\;\;\;\;e_4\circ e_6=e_2,
 \end{equation}
and $e_2\circ e_6=-e_4$ and so on for the other commutations. Additionally, one can define an involution $x\mapsto \overline{x}$, satisfying $x+\overline{x}\in\mathbb{R}$ (the \textbf{conjugation}) and a function $N(x)=x\overline{x}\in\mathbb{R}$. Then, $N$ may be seen to be a norm in $\oct$ which is precisely given by the Euclidean norm. In the next Chapter the rigorous definitions and results on more general normed division algebras, such as the octonions, are presented and some of them are deliberately used here.

The unit octonion set can be defined by 
\bege\label{essesete}
\mathbb{S}\oct=\{x\in\oct:N(x)=1\},
\enge
which is closed relative to multiplication and is, therefore, a loop. The tangent algebra at each $x\in\mathbb{S}\oct$ is given by the commutator algebra 
of pure imaginary octonions,
\bege
\imm\;\oct=\{x\in\oct:x+\overline{x}=0\}.
\enge
Of course, it has the canonical basis $\{e_1,\ldots,e_7\}$ and the commutator and associator with respect to the octonion product can be perceived by means of
\begin{equation}\begin{split}
[e_j,e_k]=2c^i_{jk}e_i,\\
[e_j,e_k,e_l]=2c^i_{jkl}e_i,
\end{split}\end{equation}
where $c_{ijkl}$ is a completely anti-symmetric nonzero tensor, equal to the unity for the cycles
\bege
(ijkl)=(4567),\;(2345),\;(2367),\;(1357),\;(1364),\;(1265),\;(1274).
\enge
One may then suppose that the geodesic loop $\mathcal{L}_e$ around a point $e\in M$ is locally isomorphic to the $7$-dimensional unit octonion space $\mathbb{SO}$. Since its tangent algebra is given by the commutator algebra $\imm\;\oct$, the fundamental tensors satisfy
\begin{equation}\begin{split}
    \alpha_{ijk}&=c_{ijk},\\
    \beta_{ijkl}&=c_{ijkl}.
\end{split}\end{equation}
Besides, $c_{ijkl}$ is fully anti-symmetric and then it follows that \begin{equation}\beta^i_{[jkl]}=\beta^i_{jkl},\end{equation} so that the generalized Jacobi identity reads
\begin{equation}\label{tatan}
    \beta^i_{jkl}=\alpha^m_{[jk}\alpha^i_{l]m}.
\end{equation}
Now, it follows from eqn (\ref{log39}) that
\begin{equation}
    \frac{1}{2}R_{[ijk]l}=S^m_{[ij}S_{k]m}-S_{ijk,l}.
\end{equation}
Hence, $\beta_{ijkl}$ and $S_{ijk,l}$ are fully anti-symmetric. It follows from (\ref{fundlog2}) that
$R_{ijkl}$ is also fully anti-symmetric. Conversely, if $R_{ijkl}$ is fully anti-symmetric, then
\begin{equation}\label{log61}
    \frac{1}{2}R_{ijkl}=S^m_{[ij}S_{kl]m}-S_{ijk,l}
\end{equation}
and therefore $S_{ijk,l}$ is fully anti-symmetric. Then again, from eqn (\ref{fundlog2}) one has that $\beta_{ijkl}$ is totally anti-symmetric, giving eqn (\ref{tatan}). In terms of the $W$-algebra, it reads
\begin{equation}
    [x,[y,z]]+[y,[z,x]]+[z,[x,y]]=6(x,y,z)
\end{equation}
which is called the Malcev identity. Now, there is only one compact simple non-Lie Malcev algebra satisfying this relation, which is precisely the commutator algebra $\imm\;\oct$ \cite{loginov2}. 

In order to generalize such construction, take a basis $\{e_1,\ldots,e_7\}$ for $\imm\;\oct$ such that
\begin{align}
[e_i,e_j]&=2kc_{ijk}e_k,\\
[e_i,e_j,e_k]&=2k^2c_{ijkl}e_l,
\end{align}
where $k$ is a real constant. It follows that
\begin{align}
\alpha_{ijk}&=kc_{ijk},\\
\beta_{ijkl}&=k^2c_{ijkl}.
\end{align}
The tensors $c_{ijk}$ and $c_{ijkl}$ are connected by self-duality relations, which can be therefore extended to the fundamental tensors. Such identities read 
\begin{align}
\varepsilon^{npqlijk}k\alpha_{ijk}&=6\beta^{npql},\\
\varepsilon^{npqlijk}\beta_{ijkl}&=24k\alpha^{npq}.
\end{align}
Additionally, there holds
\begin{align}
\alpha_{ijm}\alpha^{ijn}&=6k^2\delta^n_m,\label{68}\\
\beta_{mijk}\beta^{nijk}&=24k^2\delta^n_m,\label{69}\\
\alpha_{im}^j\alpha_{jn}^k\alpha_{kp}^i&=3k^2\alpha_{mnp}.\label{70}
\end{align}
These identities are proven in the next Chapter, namely by means of Theorem \ref{theoremcontrac}. One can then see that using such identities and the ones depicted in the last session, and assuming the Englert solution
\begin{align}
\begin{split}
F_{\mu\nu\sigma\lambda}&=\rho\varepsilon_{\mu\nu\sigma\lambda},\\
F_{mnpq}&=\lambda\partial_{[q]}S_{mnp]},
\end{split}
\end{align}
it follows that the equation of motion reads
\bege
F^{mnpq}_{\;\;\;\;\;\;\;\;\;,m}=\sqrt{2}\rho\varepsilon^{npqijkl}F_{ijkl}.
\enge
Now, from the second Bianchi identity and from eqn (\ref{log61}) there holds
\begin{equation}
    \begin{split}
        \frac{1}{2}R_{mnpq}^{\;\;\;\;\quad,m}&=S_{tm[n}S_{\;pq]}^{t\;\;\;\;,m}-S_{mnp,q}^{\;\;\;\;\;\quad,m}\\
        &=S^l_{\;m[n}S^{mt}_{\;\;\;p}S_{q]tl}-S^l_{\;m[n}S^m_{\;pq],l},
    \end{split}
\end{equation}
which together with relations (\ref{68}) and (\ref{70}) results in
\begin{equation}\label{73}
    S_{npq,m}^{\;\;\;\;\quad,m}+4k^2S_{npq}=0.
\end{equation}
In addition, eqn (\ref{log61}) yields
\begin{equation}
    S_{npq,m}=\partial_{[m}S_{npq]}.
\end{equation}
Since $F_{mnpq}=\partial_{[m}S_{npq]}$ over $K$, there holds 
\begin{equation}\label{taquase}
4k^2S^{npq}+\sqrt{2}\rho\varepsilon^{npqijkl}S_{ijk,l}=0,
\end{equation}
and a solution in the form
\begin{equation}\label{asolucao}
S_{mnp,q}=hS_{t[mn}S^t_{pq]}
\end{equation}
can be perceived as an ansatz. One can work out the parameter $h$ from
\begin{equation}
    S_{mnp,q}S^{rnp,q}=h^2\beta_{mnpq}\beta^{rmnpq}=24h^2k^4\delta^r_m,
\end{equation}
whereas from (\ref{73}) comes
\begin{equation}
    S_{mnp,q}S^{rnp,q}=-S_{mnp}S^{rnp,q}_{\;\;\;\;,q}=24k^4\delta^r_m.
    \end{equation}
It follows that $h=\pm 1$ and then inserting the ansatz in (\ref{taquase}) yields
\begin{equation}
    k=\pm 6\sqrt{2}\rho.
\end{equation}
From eqn (\ref{log61}) it comes that
\begin{align}
R_{ijkl}&=0,\;\text{if}\;h=1,\label{1}\\
R_{ijkl}&\neq 0,\;\text{if}\;h=-1\label{2}.
\end{align}
Condition (\ref{1}) is not obligatory and therefore, one can choose
\bege\label{eqfinal}
F_{mnpq}=\pm\lambda S_{t[mn}S^t_{pq]}.
\enge
It is then possible to understand the constraints imposed in the space $M_4\times K$ by this solution. Namely, inserting (\ref{eqfinal}) in the Einstein equations one gets
\begin{align}
    \mathring{R}_{\mu\nu}&=-10k^2g_{\mu\nu},\\
    \mathring{R}_{mn}&=6k^2g_{mn},
\end{align}
so that $2\lambda^2=(12k)^{-2}$. Therefore, $K$ is indeed an Einstein space and $M_4$ is the anti-de Sitter space.

\section{Cartan-Schouten geometries}
As seen in \cite{loginov3}, the above construction can be generalized, as follows. Cartan and Schouten are credited to construct 3 connections over the $7$-sphere $S^7$ \cite{loginov3ref6,loginov3ref7} and a generalization to an one parameter family was presented by Akivis \cite{akivis152}, which may be investigated in order to derive geometric information using the results heretofore developed. The 7-sphere $S^7$ may be perceived as the set $\mathbb{SO}$ of unit octonions, where the octonion product defines a parallel displacement, resulting in an affine connection, as follows.

Let $e\in S^7$ and $N_e$ be its normal neighbourhood. Such a parallel displacement can be defined by means of the octonionic right multiplication $R_x(y)=y\circ x$ for $x,y\in\oct$, in such a way that if $u,v\in N_e$, then the geodesic $\ga(e,u)$ from $e$ to $u$ can be translated into the geodesic $\ga(v,w)$, where $v=R_x(e)$ and $w=R_x(u)$ for some $x\in S^7$. This defines a product
\begin{equation}\label{vaivai}
u\bullet v=u\circ (e^{-1} \circ v).
\end{equation}
Such equation defines a local loop $\mathcal{L}_e$ with unity $e$. Such loop is nonassociative (since it makes use of the octonion multiplication in $S^7$) and is locally isomorphic to $S^7$ \cite{loginov3ref9,loginov3ref10}. One can further define an one-parameter family of loops $\mathcal{L}_e^{\alpha}$ with multiplication given by
\begin{equation}
u*v=v^{\alpha}\bullet u\bullet v^{1-\alpha},
\end{equation}
where $\alpha$ is a real constant. Note that when $\alpha=0$, one has (\ref{vaivai}). By the above observation, such loops are isomorphic to the loop $\mathcal{L}^{\alpha}$, which is defined by the following multiplication rule, denoted by juxtaposition:
\begin{equation}\label{alphageo}
uv=v^{\alpha}\circ u \circ v^{1-\alpha}=R_v(v^{\alpha}\circ u\circ v^{-\alpha}).
\end{equation}
Now, let $v$ be a fixed point in $S^7$ and let $u$ denote a point in an one-parameter subgroup from $e$. It follows that the point $u'=v^{\alpha}\circ u\circ v^{-\alpha}$ runs an one-parameter subgroup and the point $w=uv$ runs through a line obtained by translating $u'$ by $R_v$. Therefore, if $e$ and $v$ are fixed and $u$ describes the geodesic $\ga(e,u)$, then the point $w$ describes the geodesic $\ga(v,w)$. Therefore, all loops $\mathcal{L}_e^{\alpha}$ are geodesic loops of affine connections on $S^7$ which share the same set of geodesics generated by one-parameter subgroups of the loop on the unit octonions.

Therefore, by means of the fundamental tensors one can find the curvature and torsion tensors of affine connections generated on $S^7$ by the geodesic loops $\mathcal{L}^{\alpha}_e$. Since two elements in a loop generate a group, it follows that one can make use of the Campbell-Hausdorff series
\begin{equation}
(xy)^i=x^i+y^i+\frac{1}{2}c^i_{jk}x^jy^k+\frac{1}{12}c^i_{jm}c^m_{kl}(x^jx^ky^l+y^jy^kx^l)+\ldots,
\end{equation}
where $c^i_{jk}$ are the structure constants of the octonion algebra. Since $(v^{\alpha})^i=\alpha v^i$, from eqn (\ref{alphageo}) one gets 
\begin{equation}
(xy)^i=u^i+v^i+\frac{1}{2}(1-2\alpha)c^i_{jk}u^jv^k+\frac{1}{12}c^i_{jm}c^m_{kl}[u^ju^kv^l+(1-6\alpha+6\alpha^2)v^jv^ku^l]+\ldots
\end{equation}
On the other hand, in the geodesic loop $\mathcal{L}^{\alpha}$ there holds
\begin{equation}
(xy)^i=u^i+v^i+\lambda^i_{jk}u^jv^k+\frac{1}{2}[\mu^i_{jkl}u^ju^kv^l+\nu^i_{jkl}v^jv^ku^l]+\ldots
\end{equation}
Which yields the relations for the fundamental tensors
\begin{align}
2\alpha^i_{jk}&=(1-2\alpha)c^i_{jk},\label{aprim}\\
-4\beta^i_{jkl}&=\alpha(1-\alpha)c^i_{jm}c^m_{kl}+(1+3\alpha+3\alpha^2)c^i_{m[j}c^m_{kl]}\label{asegun}.
\end{align}
It follows that $\alpha^i_{jk}=kc^i_{jk}$, where $k$ is a real constant depending on the real parameter $\alpha$. Since $\alpha^i_{jk}=-S^i_{jk}$ and $c^i_{jk}$ is fully anti-symmeteric, it follows that for each fixed $\alpha$ the geodesic loop $\mathcal{L}^{\alpha}_e$ generates a metric-compatible affine connection on $S^7$ given by
\bege
\Gamma_{ijk}=\mathring{\Gamma}_{ijk}+S_{ijk},
\enge
where $\mathring{\Gamma}_{ijk}$ is the Riemannian symmetric connection and $S_{ijk}$ is  a fully anti-symmetric torsion. Using the full anti-symmetry of $S_{ijk}$ one can rewrite the tensor $\beta^i_{jkl}$ by means of Theorem \ref{funda} and equations
\begin{equation}\begin{split}
      \alpha^i_{jk}&=-S^i_{jk},\\
      4\beta^i_{jkl}&=-2\nabla_l S^i_{jk}-R^i_{jkl}.
  \end{split}\end{equation}
By eqn (\ref{aprim}) one can see that if $\alpha=1/2$ then the connection yields a torsionless Riemannian geometry. If $\alpha\neq 1/2$ then by eqn (\ref{aprim}) it follows that the new solution depends on $\alpha$ with
\begin{equation}\label{final2}
S_{ijk,l}=hS^m_{[ij}S_{kl]m},\;\;\;h=\frac{1}{1-2\alpha}.
\end{equation}
Finally, using eqn (\ref{final2}) one can find the Riemann curvature tensor of the affine connection given by the parallel translation considered in $\mathcal{L}^{\alpha}_e$. Namely,
\begin{equation}
R_{ijkl}=4\alpha(1-\alpha)S^m_{ij}S_{klm}-4\alpha(2-3\alpha)S^m_{[ij}S_{kl]m}.
\end{equation}
If $\alpha=0$ then the curvature vanishes and if $\alpha=1$ it is fully anti-symmetric, and these correspond to the solutions in (\ref{asolucao}). The geodesic loops $\mathcal{L}^0_e$ and $\mathcal{L}^1_e$ are therefore also locally isomorphic to the loop $S^7$. Using the same ideas as before, but now considering the parameter $h$ one gets
\begin{equation}
    S_{npq;m}^{\quad\quad;m}+(2hk)^2S_{npq}=0.
\end{equation}
Inserting in the equation of motion yields
\begin{equation}
    (2hk)^2S^{npq}+\sqrt{2}\rho\varepsilon^{npqijkl}S_{ijk;l}=0.
\end{equation}
Then, using again the self-duality relations gives
\begin{equation}
    h=6\sqrt{2}\rho k^{-1}.
\end{equation}
Likewise, using the Einstein equations one gets
\begin{align}
    R_{\mu\sigma\eta\lambda}F_{\nu}^{\;\sigma\eta\lambda}&=-6\rho^2 g_{\mu\nu},\\
    F_{mrpq}F_n^{\;rpq}&=24k^4h^2\lambda^2g_{mn},
\end{align}
which yield
\begin{align}
    \mathring{R}_{\mu\nu}&=-10(hk)^2g_{\mu\nu},\\
    \mathring{R}_{mn}&=6(hk)^2g_{mn},
\end{align}
with 
\begin{equation}
    2\lambda^2=(12k)^{-2}.
\end{equation}
Therefore, the constants $h$ and $k$ are determined by $\rho$ and $\lambda$ which are arbitrary in this construction. Therefore, one obtains a family of geometries in the internal space $K$ which come from solutions of the spontaneous compactification mechanisms for $d=11$.


\part{$G_2$-Structures and Deformations}
\chapter{The Octonions}
In this Chapter the main results on $G_2$-structures over 7-dimensional manifolds, which may be subsequently used to define the so-called octonion bundle $\oct M$, are established. With such structures, one can rigorously define the octonion product over $M$ which can be seen to connect with the torsion of the underlying $G_2$-structure. The main results are extracted from \cite{grigorian,grigoriantorsion} and the computational results come from \cite{spirocontas,spirotudo}.

\section{Division Algebras}
In this section division algebras are analyzed from the ground up, having as final goal the definition of the octonion algebra $\oct$ and further understanding its elementary properties. The ref. \cite{intro} is followed and omitted proofs can be found therein. For more information on octonions and division algebras, one can see for instance \cite{baez}.

Throughout this composition the vector spaces $\A=\re^n$ are considered and endowed with the usual Euclidean inner product $\langle\cdot,\cdot\rangle$. Define then the binary product
\begin{equation}\begin{split}
*\;:\;\A\times\A&\longrightarrow\;\;\A\\
(A\;,B)&\;\;\mapsto\;\;A*B.
\end{split}\end{equation}
\begin{definition}
The pair $(\A,*)$ is called a \textbf{ normed division algebra} if it is a real algebra over $\re$ with unity $1\in \A$ such that
\begin{equation}\label{divisionn}
    \Vert A*B\Vert=\Vert A\Vert\Vert B\Vert,\;\;\;\forall A,B\in\A,
\end{equation}
where as usual $\Vert A\Vert^2=\langle A,A\rangle$.
\end{definition}

The first obvious examples are the real $\re$ and complex $\com$ division algebras, which satisfy eqn (\ref{divisionn}). As it may be seen there are only two other examples, namely the quaternions $\quat\simeq \re^4$ and octonion $\oct\simeq\re^8$ algebras, the latter being of greater importance for the work to be discussed afterwards.

\begin{preremark}\upshape
For now on, whenever a normed division algebra $\A$ is introduced its product shall be denoted by juxtaposition, namely
\begin{equation}A*B=AB.\end{equation}
Therefore, whenever $\A$ is normed division algebra then this product is automatically considered. In addition, the notation for the unity may be fixed as $1\in \A$.
\end{preremark}

\begin{definition}
Let $\A$ be a normed division algebra. The \textbf{real} and \textbf{imaginary parts of} $\A$, respectively denoted by $\emph{\ree}(\A)$ and $\emph{\imm}(\A)$, are defined by the relations
\begin{equation}\begin{split}
    \emph{\ree}(\A) &= \text{span}(\{1\})=\{\lambda 1\;:\;\lambda\in\re\},\\
    \emph{\imm}(\A) &= \left(\emph{\ree}(\A)\right)^{\perp}\simeq\re^{n-1},
\end{split}\end{equation}
where the orthogonal complement is taken with respect to the Euclidean inner product over $\A=\re^n$. 
\end{definition}
\begin{preremark}\upshape
It follows that for each $A\in\A$ there are unique $\ree(A)\in\ree(\A)$ and $\imm(A)\in\imm(\A)$ such that
\begin{equation}
A=\ree(A)+\imm(A).
\end{equation}
These are  respectively called the \textbf{real} and \textbf{imaginary} parts of $A$. One can then follow to define the linear map $A\mapsto\overline{A}$ called the \textbf{conjugation} given by
\begin{equation}
\overline{A}=\ree(A)-\imm(A).\end{equation}
The conjugation is an involution over $\A$, which means that $\overline{\overline{A}}=A$. By construction, it is also an isometric reflection across the hyperplane $\imm(\A)$. Moreover, notice that
\begin{equation}
    \ree(A)=\frac{1}{2}(A+\overline{A})\;\;\;\;\;\;\;\imm(A)=\frac{1}{2}(A-\overline{A}).
\end{equation}
This description gives a characterization for the real $\ree(\A)$ and imaginary $\imm(\A)$ parts as
\begin{equation}\begin{split}
    A&\in\imm(\A)\;\;\;\iff\;\;\;\overline{A}=-A,\\
    A&\in\ree(\A)\;\;\;\iff\;\;\;\overline{A}=A,
\end{split}\end{equation}
\end{preremark}
\begin{lemma}
For every $A,B,C\in\A$ there holds
\begin{align}
    \overline{AB}&=\overline{B}\overline{A},\label{intro40}\\
    \langle A,BC\rangle&=\langle A\overline{C},B\rangle,\label{intro39b}\\
    \langle A,CB\rangle&=\langle\overline{C}A,B\rangle,\label{intro39a}\\
    \langle AC,BC\rangle&=\langle CA,CB\rangle=\langle A,B\rangle \Vert C\Vert^2.\label{intro38}
\end{align}
\end{lemma}
\begin{proof}
Each identity from eqn (\ref{intro38}) upwards may be considered. By straightforward calculations there holds
\begin{equation}
    \Vert\lp A+B\rp C\Vert^2=\Vert AC+BC\Vert^2=\Vert AC\Vert^2+2\langle AC,BC\rangle+\Vert BC\Vert^2.
\end{equation}
On the other hand,
\begin{equation}
    \Vert A+B\Vert^2\Vert C\Vert^2=\lp\Vert A\Vert^2+2\langle A,B\rangle+\Vert B\Vert^2\rp\Vert C\Vert^2.
\end{equation}
Now, since \begin{equation}\Vert\lp A+B\rp C\Vert^2=\Vert A+B\Vert^2 \Vert C\Vert^2,\end{equation} one arrives at both equalities in (\ref{intro38}).
Besides, eqns (\ref{intro39a}, \ref{intro39b}) are obviously satisfied when $C\in\ree(\A)$, since the inner product is bilinear and $\overline{C}=C$ in this case. One may hence assume $C\in\imm(\A)$ and so $\overline{C}=-C$. Notice then that by definition $C$ is orthogonal to $1$, which makes
\begin{equation}\Vert 1+C\Vert^2=1+\Vert C\Vert^2.\end{equation}
It follows that
\begin{equation}\begin{split}
    \langle A,B\rangle\lp1+\Vert C\Vert^2\rp&=\langle A,B\rangle\Vert 1+C\Vert^2=\langle A\lp 1+C\rp,B\lp 1+C\rp\rangle\\
    &=\langle A+AC,B+BC\rangle=\langle A,B\rangle+\langle AC,BC\rangle+\langle A,BC\rangle+\langle AC,B\rangle\\
    &=\langle A,B\rangle +\langle A,B\rangle\Vert C\Vert^2+\langle A,BC\rangle+\langle AC,B\rangle\\
    &=\lp1+\Vert C\Vert^2\rp\langle A,B\rangle+\langle A,BC\rangle+\langle AC,B\rangle,
\end{split}\end{equation}
which in turn implies 
\begin{equation}
\langle A,BC\rangle=-\langle AC,B\rangle=\langle A\overline{C},B\rangle,\end{equation}
as wanted. For the remaining equation, using the previous ones and the fact that the conjugation is an isometry (and therefore self-adjoint) one gets
\begin{equation}
    \langle\overline{AB},C\rangle=\langle AB,\overline{C}\rangle=\langle B,\overline{A}\overline{C}\rangle=\langle BC,\overline{A}\rangle=\langle C,\overline{B}\overline{A}\rangle=\langle \overline{B}\overline{A},C\rangle,
\end{equation}
which proves eqn (\ref{intro40}) since the previous relation holds for all $C\in\A$.
\end{proof}

\begin{corollary}
For every $A,B,C\in\A$ there holds
\begin{align}
    A\lp\overline{B}C\rp+B\lp\overline{A}C\rp&=2\langle A,B\rangle C,\label{intro314}\\
    \lp A\overline{B}\rp C+\lp A\overline{C}\rp B&=2\langle B,C\rangle A,\label{intro315}\\
    A\overline{B}+B\overline{A}&=2\langle A,B\rangle 1.\label{intro316}
\end{align}
\end{corollary}
\begin{proof}
Indeed, taking $D\in\A$ it follows from eqn (\ref{intro38}) that 
\begin{equation}\begin{split}
    \langle A,B\rangle \Vert C+D\Vert^2&=\langle A\lp C+D\rp,B\lp C+D\rp\rangle\\
    &=\langle AC,BC\rangle+\langle AD,BC\rangle+\langle AC,BD\rangle+\langle AD,BD\rangle\\
    &=\langle A,B\rangle\lp\Vert C\Vert^2+\Vert D\Vert^2\rp+\langle AD,BC\rangle+\langle AC,BD\rangle,
\end{split}\end{equation}
on the other hand
\begin{equation}\begin{split}
    \langle A,B\rangle \Vert C+D\Vert^2&=\langle A,B\rangle\lp\Vert C\Vert^2+2\langle C,D\rangle+\Vert D\Vert^2\rp\\
    &=\langle A,B\rangle\lp\Vert C\Vert^2+\Vert D\Vert^2\rp+2\langle A,B\rangle\langle C,D\rangle.
\end{split}\end{equation}
Therefore, one can see that
\begin{equation}
    \langle AD,BC\rangle+\langle AC,BD\rangle=2\langle A,B\rangle\langle C,D\rangle.
\end{equation}
With the aid of eqns (\ref{intro39a}, \ref{intro39b}) it becomes
\begin{equation}
    \langle D,\overline{A}\lp BC\rp\rangle+\langle\overline{B}\lp AC\rp,D\rangle=2\langle A,B\rangle\langle C,D\rangle,
\end{equation}
which holds for every $D\in\A$. Finally, one gets the relation
\begin{equation}
\overline{A}\lp BC\rp+\overline{B}\lp AC\rp=2\langle A,B\rangle C.\end{equation}
Making $A\mapsto \overline{A}$ and $B\mapsto\overline{B}$ gives eqn (\ref{intro314}). In addition, eqn (\ref{intro315}) is obtained from (\ref{intro314}) by making $A\mapsto \overline{C}$, $B\mapsto \overline{B}$ and taking the conjugate on both sides. Ultimately, eqn (\ref{intro316}) is just (\ref{intro314}) with $C=1$.
\end{proof}

\begin{corollary}\label{intro318}
If $A,B,C\in\imm(\A)$ then there holds
\begin{align}
    A\lp BC\rp+B\lp AC\rp&=-2\langle A,B\rangle C,\label{intro319}\\
    \lp AB\rp C+\lp AC\rp B&=-2\langle B,C\rangle A,\label{intro320}\\
    AB+BA&=-2\langle A,B\rangle 1.\label{intro321}
\end{align}
\end{corollary}
\begin{corollary}
If $A,B\in\A$ then it follows that
\begin{equation}\label{intro323}
\langle A,B\rangle=\emph{\ree}\lp A\overline{B}\rp=\emph{\ree}\lp B\overline{A}\rp=\emph{\ree}\lp\overline{B}A\rp=\emph{\ree}\lp\overline{A}B\rp
\end{equation}
and
\begin{equation}\label{intro324}
    \Vert A\Vert^2=A\overline{A}=\overline{A}A.
\end{equation}
In addition, if either $A$ or $B$ is imaginary then there holds
\begin{equation}\label{intro342}
    \langle A,B\rangle 1=-\emph{\ree}(A\overline{B}).
\end{equation}

\end{corollary}
\begin{proof}
From eqn (\ref{intro39b}) it follows that
\begin{equation}
\langle A,B\rangle=\langle A\overline{B},1\rangle=\ree\lp A\overline{B}\rp,\end{equation}
and the other equalities come from the symmetry of the inner product and from the conjugation being isometric. Finally, notice that 
\begin{equation}\overline{A\overline{A}}=\overline{\overline{A}}\overline{A}=A\overline{A},\end{equation}
so that $A\overline{A}$ is real (and similarly for $\overline{A}A$). Therefore, eqn (\ref{intro324}) follows. Now, using the previous identities and supposing $\overline{B}=-B$ yields
\begin{equation} \langle A,B\rangle = \ree\lp A\overline{B}\rp=-\ree\lp AB\rp.\end{equation}
\end{proof}

\begin{corollary}\label{intro325}
The element $A\in\A$ has the property $A^2\in\emph{\ree}(\A)$ if and only if $A\in\emph{\ree}(\A)$ or $A\in\emph{\imm}(\A)$.
\end{corollary}
\begin{proof}
One can set $A=\ree(A)+\imm(A)$ and note that since $-\overline{\imm(A)}=\imm(A)$ there holds
\begin{equation}\imm(A)^2=\imm(A)\lp-\overline{\imm(A)}\rp=-\Vert\imm(A)\Vert^2,\end{equation}
where eqn (\ref{intro324}) was used. Then,
\begin{equation}
A^2=\lp\ree(A)+\imm(A)\rp\lp\ree(A)+\imm(A)\rp=\lp\ree(A)^2-\Vert\imm(A)\Vert^2\rp1+2\ree(A)\imm(A).\end{equation}
Since $A^2$ is real, it follows that its imaginary part vanishes, namely
\begin{equation}2\ree(A)\imm(A)=0,\end{equation}
which implies that $\ree(A)=0$ or $\imm(A)=0$.
\end{proof}
\begin{corollary}
For every $A,B\in\A$ there holds
\begin{equation}\label{intro327}
    \begin{split}
        \lp AB\rp\overline{B}&=A\lp B\overline{B}\rp=\Vert B\Vert^2A=A\lp\overline{B}B\rp=\lp A\overline{B}\rp B,\\
        A\lp \overline{A}B\rp&=\lp A\overline{A}\rp B=\Vert A\Vert^2 B=\lp\overline{A}A\rp B=\overline{A}\lp AB\rp.
    \end{split}
\end{equation}
\end{corollary}
\begin{proof}
Using the identities proven so far, one has
\begin{equation}
\langle\lp AB\rp\overline{B},C\rangle=\langle AB,CB\rangle=\langle A,C\rangle \Vert B\Vert^2=\langle A\Vert B\Vert^2,C\rangle=\langle A\lp B\overline{B}\rp,C\rangle,\end{equation}
which holds for every $C\in\A$. The other identities follows similarly.
\end{proof}

As it is well-known, the real and complex normed division algebras $\re$ and $\com$ are both commutative and associative. However, this may not hold for every normed division algebra $\A$. With that in mind, one may define the usual operators measuring the failure for the product over $\A$ to be commutative or associative.
\begin{definition}
Let $\A$ be a normed division algebra. Then, one can define a bilinear map $[\cdot,\cdot]:\A\times\A\rightarrow\A$ by the relation
\begin{equation}
    [A,B]=AB-BA,\;\;\;\;\;\;\forall A,B\in\A.
\end{equation}
Such map is called the \textbf{commutator} of $\A$. Furthermore, the trilinear map $(\cdot,\cdot,\cdot]:\A\times\A\times\A\rightarrow\A$ given by
\begin{equation}
    [A,B,C]=(AB)C-A(BC),\;\;\;\;\;\;\forall A,B,C\in\A
\end{equation}
can be defined and is called the \textbf{associator} of $\A$.
\end{definition}

\begin{proposition}
The commutator and associator operators over a normed division algebra $\A$ are both alternating multilinear applications.
\end{proposition}
\begin{proof}
Notice that whenever one of the arguments is purely real, the associator vanishes. Therefore, one may exclusively consider the imaginary case. If $A,B\in\imm(\A)$ then as usual $\overline{A}=-A$ and $\overline{B}=-B$. It then follows from eqn (\ref{intro327}) that
\begin{equation}-[A,A,B]=[A,\overline{A},B]=\lp A\overline{A}\rp B-A\lp \overline{A} B\rp=0.\end{equation}
In the same way one may see that $-[A,B,B]=[A,\overline{B},B]=0$. Therefore, the associator is alternating in the first two arguments. Hence, $[A,B,A]=-[A,A,B]=0$ and it follows that $[\cdot,\cdot,\cdot]$ is totally anti-symmetric, as wanted. The commutator is alternating by definition.
\end{proof}

\begin{proposition}\label{intro332}
For every $A,B,C\in\emph{\imm}(\A)$ there holds $[A,B]\in\emph{\imm}(\A)$ and $[A,B,C]\in\emph{\imm}(\A)$.
\end{proposition} 
\begin{proof}
It suffices to show that $[A,B]$ and $[A,B,C]$ are both orthogonal to $1$. As usual, one has $\overline{A}=-A$ for every $A\in\imm(\A)$, yielding
\begin{equation}\begin{split}
    \langle [A,B],1\rangle&=\langle AB-BA,1\rangle=\langle B,\overline{A}\rangle-\langle A,\overline{B}\rangle\\
    &=-\langle B,A\rangle +\langle A,B\rangle =0.
\end{split}\end{equation}
For the associator, one has
\begin{equation}\begin{split}
    \langle[A,B,C],1\rangle&=\langle \lp AB\rp C-A\lp BC\rp,1\rangle=\langle AB,\overline{C}\rangle -\langle BC,\overline{A}\rangle\\
    &=-\langle AB,C\rangle+\langle BC,A\rangle=-\langle A,C\overline{B}\rangle+\langle BC,A\rangle\\
    &=\langle A,CB+BC\rangle=\langle A,\overline{BC}+BC\rangle=2\langle A,\ree(BC)\rangle=0,
\end{split}\end{equation}
where in the last line one uses the relation $\overline{BC}=\overline{C}\overline{B}=(-C)(-B)=CB$.
\end{proof}

\begin{proposition}\label{intro333}
The mappings $(A,B,C)\mapsto \langle A,[B,C]\rangle$ and $(A,B,C,D)\mapsto\langle A,[B,C,D]\rangle$ are multilinear and alternating.
\end{proposition}
\begin{proof}
Since these are compositions of multilinear maps, it is straightforward that they also are multilinear. Moreover, since the commutator and associator are already anti-symmetric one must only show that $\langle A,[A,B]\rangle$ and $\langle A,[A,B,C]\rangle$ both vanish. It follows that
\begin{equation}\begin{split}
    \langle A,[A,B]\rangle&=\langle A,AB-BA\rangle=\langle \overline{A}A,B\rangle-\langle A\overline{A},B\rangle\\
    &=\Vert A\Vert^2\langle 1,B\rangle - \Vert A\Vert^2\langle1,B\rangle=0
\end{split}\end{equation}
and
\begin{equation}\begin{split}
    \langle A,[A,B,C]\rangle&=\langle A,\lp AB\rp C-A\lp BC\rp\rangle=\langle A\overline{C},AB\rangle -\Vert A\Vert^2\langle 1,BC\rangle\\
    &=\Vert A\Vert^2\langle\overline{C},B\rangle-\Vert A\Vert^2\langle\overline{C},B\rangle=0.
\end{split}\end{equation}

\end{proof}

\section{$\varphi$ and $\psi$}

Maintaining the notation, consider $\A\simeq\re^n$ a normed division algebra and $\imm(\A)\simeq\re^{n-1}$ its imaginary part. The restriction of its product over the imaginary part may be then analyzed, defining the following well-known operation:
\begin{definition}
Let $\A$ be a normed division algebra. Then, the \textbf{vector cross product} $\times$ over $\emph{\imm}(\A)$ is the operation $\times:\emph{\imm}(\A)\times\emph{\imm}(\A)\rightarrow\emph{\imm}(\A)$ defined by
\begin{equation}\label{introcross}
    A \times B=\emph{\imm}(AB).
\end{equation}
\end{definition}

\begin{lemma}
Suppose that $A,B\in\emph{\imm}(\A)$. Then,
\begin{align}
    A \times B&=-B \times A,\label{intro340}\\
    \langle A \times B,A\rangle &=\langle A\times B,B\rangle= 0.\label{intro341}
\end{align}
\end{lemma}
\begin{proof}
As usual one has $\overline{A}=-A$ and $\overline{B}=-B$ and then, since $2\imm(A)=A-\overline{A}$, it follows that
\begin{equation}\label{intro343}
    2 A\times B=2\imm(AB)=AB-\overline{AB}=AB-BA=[A,B],
\end{equation}
which proves the first identity by the anti-symmetry of the commutator. Now, since $A\in\imm(\A)$ there holds $\langle \ree(AB),A\rangle=0$ and therefore
\begin{equation}
\langle A\times B,A\rangle=\langle \imm(AB),A\rangle=\langle \ree(AB)+\imm(AB),A\rangle=\langle AB,A\rangle.\end{equation}
But then, since $B\in\imm(\A)$ one has
\begin{equation}
\langle A\times B,A\rangle = \langle AB,A\rangle=\Vert A\Vert^2\langle B,1\rangle=0,\end{equation}
as wanted. By the same reasoning one gets $\langle A\times B,B\rangle = 0$, which proves the statement.
\end{proof}
\begin{preremark}\upshape
Notice that by eqn (\ref{intro342}) there holds $\ree(AB)=-\langle A,B\rangle 1$ and by definition $A\times B=\imm(AB)$ so that one gets
\begin{equation}\label{intro344}
    AB=-\langle A,B\rangle 1+A\times B.
\end{equation}
Besides, eqns (\ref{intro340}, \ref{intro341}) respectively show that $\times$ is anti-symmetric and that the vector $A\times B$ is orthogonal to both $A$ and $B$. Remarkably, these are all properties of the usual vector product $\times$ over $\re^3$ and one may see that it is indeed connected to the notion of normed division algebra here presented. In order to do that, by Propositions \ref{intro332} and \ref{intro333} one may naturally define the following applications
\end{preremark}
\begin{definition}\label{intro334}
Let $\A$ be a normed division algebra and define the $3$- and $4$- forms $\varphi$ and $\psi$ over $\emph{\imm}(\A)$ by the relations
\begin{align}
    \varphi(A,B,C)&=\frac{1}{2}\langle[A,B],C\rangle=\frac{1}{2}\langle A,[B,C]\rangle,\label{intro335}\\
    \psi(A,B,C,D)&=\frac{1}{2}\langle [A,B,C],D\rangle=-\frac{1}{2}\langle A,[B,C,D]\rangle\label{intro336},
\end{align}
for every $A,B,C,D\in\emph{\imm}(\A)$.
\end{definition}
\begin{preremark}\upshape 
Notice that eqn (\ref{intro343}) shows that $[A,B]=2 A\times B$. Therefore, by definition there follows
\begin{equation}\label{intro345}
    \varphi(A,B,C)=\langle A\times B,C\rangle,
\end{equation}
for every $A,B,C\in\imm(\A)$. In addition, since from eqn (\ref{intro344}) $AB$ and $A\times B$ differ only by a real part, one has
\begin{equation}
    \varphi(A,B,C)=\langle AB,C\rangle.
\end{equation}
\end{preremark}

\begin{lemma}
If $A,B,C\in\imm(\A)$ then
\begin{equation}\label{intro348}
A(BC)=-\frac{1}{2}[A,B,C]-\varphi(A,B,C)1-\langle B,C\rangle A+\langle A,C\rangle B-\langle A,B\rangle C.
\end{equation}
\end{lemma}
\begin{proof}
Using each of the identities in Corollary \ref{intro318} one has
\begin{equation}\begin{split}
    A\lp BC\rp&=-B \lp AC\rp-2\langle A,B\rangle C\\
    &=-B\lp-CA-2\langle A,C\rangle 1\rp-2\langle A,B\rangle C\\
    &=B\lp CA\rp+2\langle A,C\rangle B-2\langle A,B\rangle C\\
    &=-C\lp BA\rp-2\langle B,C\rangle A+2\langle A,C\rangle B-2\langle A,B\rangle C.
\end{split}\end{equation}
Now with eqn (\ref{intro316}) it is possible to see that
\begin{equation}
C\lp BA\rp-\lp AB\rp C=C\lp\overline{AB}\rp+\lp AB\rp\overline{C}=2\langle AB,C\rangle 1=2\varphi(A,B,C)1.\end{equation}
Hence, it follows that
\begin{equation}
A\lp BC\rp=-\lp AB\rp C-2\varphi(A,B,C)1-2\langle B,C\rangle A+2\langle A,C\rangle B-2\langle A,B\rangle C.\end{equation}
Finally, using $[A,B,C]=\lp AB\rp C-A\lp BC\rp$ in the last equation gives the desired result.
\end{proof}
\begin{proposition}
Let $A,B,C\in\emph{\imm}(\A)$ and $\times$ the vector cross product defined over $\emph{\imm}(\A)$. There holds
\begin{align}
    \Vert A\times B\Vert^2&=\Vert A\Vert^2\Vert B\Vert^2-\langle A,B\rangle^2=\Vert A\wedge B\Vert^2,\label{intro350}\\
     A\times(B\times C)&=-\langle A,B\rangle C+\langle A,C\rangle B-\frac{1}{2}[A,B,C] \label{intro351}\\ 
    &=-\langle A,B\rangle C+\langle A,C\rangle B-\psi(A,B,C,\cdot)^{\sharp}.\label{intro363}
\end{align}
\end{proposition}
\begin{proof}
There holds
\begin{equation}\begin{split}
    4\Vert A\times B\Vert^2&=\langle AB-BA,AB-BA\rangle=\Vert AB\Vert^2+\Vert BA\Vert^2-2\langle AB,BA\rangle\\
    &=2\Vert A\Vert^2\Vert B\Vert^2-2\langle AB,BA\rangle.
\end{split}\end{equation}
Now, since $AB=-\langle A,B\rangle 1+A \times B$ and $BA=-\langle A,B\rangle1 + B\times A$, one has
\begin{equation}
\langle AB,BA\rangle\langle-\langle A,B\rangle 1+A\times B,-\langle A,B\rangle - A\times B\rangle=\langle A,B\rangle ^2-\Vert A\times B\Vert^2.\end{equation}
Combining the two previous expression yields the first identity. Similarly, there follows
\begin{equation}\begin{split}
    A \times\lp B\times C\rp&=\langle A,B \times C\rangle1 +A\lp B\times C\rp\\
    &=\varphi(A,B,C)1+A\lp\langle B,C\rangle +BC\rp\\
    &= A\lp BC\rp+\varphi(A,B,C)1+\langle B,C\rangle A.
\end{split}\end{equation}
Then, inserting eqn (\ref{intro348}) into $A\lp BC\rp$ yields the second identity. Finally, it follows directly from (\ref{intro336}) that
\begin{equation}
    A \times\lp B\times C\rp=-\langle A,B\rangle C+\langle A,C\rangle B-\psi(A,B,C,\cdot)^{\sharp},
\end{equation}
where the sharp operation is defined via the Euclidean metric $\langle\cdot,\cdot\rangle$.
\end{proof}

\begin{corollary}
Let $A,B,C\in\emph{\imm}(\A)$. Then, it follows that
\begin{equation}\label{intro366}
    A \times \lp A\times C\rp=-\Vert A\Vert^2 C+\langle A,C\rangle A.
\end{equation}
In addition, suppose $\{A,B,C\}$ is an orthonormal set with respect to the Euclidean metric. It follows that if $A\times B=C$ then $B\times C =A$ and $C\times A=B$.
\end{corollary}

\begin{proposition}
Let $A,B,C,D\in\emph{\imm}(\A)$. Then, there holds
\begin{align}
    \langle A\times B,C\times D\rangle&=\langle A\wedge B,C\wedge D\rangle-\frac{1}{2}\langle A,[B,C,D]\rangle\label{intro370}\\
    &=\langle A\wedge B,C\wedge D\rangle +\psi(A,B,C,D).\label{intro373}
\end{align}
\end{proposition}
\begin{proof}
As previously seen, one has
\begin{equation}
\langle A\wedge B,C\wedge D\rangle=\det\begin{pmatrix}
\langle A,C\rangle & \langle A,D\rangle\\
\langle B,C\rangle & \langle B,D\rangle
\end{pmatrix}
=\langle A,C\rangle\langle B,D\rangle-\langle A,D\rangle\langle B,C\rangle.\end{equation}
Using eqn (\ref{intro345}) it follows that
\begin{equation}
\langle A\times B,C\times D\rangle=\varphi(A,B,C\times D)=-\varphi(A,C\times D,B)=-\langle A\times\lp C\times D\rp,B\rangle.\end{equation}
Then, using (\ref{intro351}) it comes
\begin{equation}\begin{split}
\langle A\times B,C\times D\rangle&=\langle-\langle A,C\rangle D+\langle A,D\rangle C-\frac{1}{2}[A,C,D],B\rangle\\
&=\langle A,C\rangle\langle B,D\rangle-\langle A,D\rangle\langle B,C\rangle+\frac{1}{2}\langle B,[A,C,D]\rangle\\
&=\langle A,C\rangle\langle B,D\rangle-\langle A,D\rangle\langle B,C\rangle-\frac{1}{2}\langle A,[B,C,D]\rangle.
\end{split}\end{equation}
Then since $\psi(A,B,C,D)=-\frac{1}{2}\langle A,[B,C,D]\rangle$ it follows that
\begin{equation}
    \langle A\times B,C\times D\rangle=\langle A\wedge B,C\wedge D\rangle +\psi(A,B,C,D).
\end{equation}
\end{proof}
\begin{preremark}\upshape
These results show that given a normed division algebra $\A$ one can define a vector cross product over the imaginary part $\imm(\A)$ satisfying some identities. Conversely, one can make the following definition:
\end{preremark}
\begin{definition}
Consider $\I\simeq\re^{n-1}$ endowed with the usual Euclidean inner product. One says that $\I$ has a vector cross product if there exists an alternating bilinear map \begin{equation}\times: \I\times\I\rightarrow \I\end{equation} such that for every $\alpha,\beta,\gamma\in\I$ there holds
\begin{align}
    \langle \alpha\times \beta,\alpha\rangle&=\langle\alpha\times\beta,\beta\rangle=0,\label{intro356}\\
    \Vert \alpha\times \beta\Vert^2&=\Vert \alpha\Vert^2\Vert \beta\Vert^2-\langle \alpha,\beta\rangle^2=\Vert \alpha\wedge \beta\Vert^2.\label{intro357}
\end{align}
\end{definition}
\begin{preremark}\upshape What has so far been developed is that if one has a normed division algebra $\A$ then setting $\I=\imm(\A)$ it follows that the vector cross product $\times$ as defined in eqn (\ref{introcross}) satisfy properties (\ref{intro356}, \ref{intro357}). In fact, one may prove the
\end{preremark}
\begin{theorem}
There is an one-to-one correspondence between normed division algebras $\A=\re^n$ and spaces $\I=\re^{n-1}$ admitting vector cross products.
\end{theorem}
\begin{proof}
As stated before, given a normed division algebra $\A$ one can construct a vector cross product over $\imm(\A)$. Let now $\I=\re^{n-1}$ be endowed with a vector cross product $\times$. Set $\A=\re\oplus\I$ and equip it with the usual Euclidean metric, namely
\begin{equation}
\langle \lp a,\alpha\rp,\lp b,\beta\rp\rangle=ab+\langle \alpha,\beta\rangle,\end{equation}
where $a,b\in \re$ and $\alpha,\beta\in\I$. Then, one can define the following product over $\A$:
\begin{equation}\label{intro360}
    \lp a,\alpha\rp\lp b,\beta\rp=\lp ab-\langle\alpha,\beta\rangle,a\beta+b\alpha+\alpha\times \beta\rangle\rp.
\end{equation}
Such product is clearly bilinear with $(1,0)$ its identity. Therefore, for $\A$ to be a normed division algebra, one must only check if eqn (\ref{divisionn}) holds. Calculating comes
\begin{equation}\begin{split}
    \Vert (a,\alpha)(b,\beta)\Vert^2&=\lp ab-\langle\alpha,\beta\rangle+\Vert a\beta+b\alpha+\alpha\times\beta\Vert^2\rp\\
    &=a^2b^2-2ab\langle\alpha,\beta\rangle+(\langle\alpha,\beta\rangle)^2+a^2\Vert\beta\Vert^2+b^2\Vert\alpha\Vert^2+\Vert\alpha\times\beta\Vert^2\\
    &\;\;\;+2ab\langle\alpha,\beta\rangle+2a\langle\beta,\alpha\times\beta\rangle+2b\langle\alpha,\alpha\times\beta\rangle.
\end{split}\end{equation}
Using the defining identities (\ref{intro356}, \ref{intro357}) of the vector cross product it follows that
\begin{equation}\begin{split}
    \Vert\lp a,\alpha\rp\lp b,\beta\rp\Vert^2&=a^2b^2+a^2\Vert\beta\Vert^2+b^2\Vert\alpha\Vert^2+\Vert\alpha\Vert^2\Vert\beta\Vert^2\\
    &=\lp a^2+\Vert\alpha\Vert^2\rp\lp b^2+\vert\beta\Vert^2\rp=\Vert\lp a,\alpha\rp\Vert\Vert\lp b,\beta\rp\Vert,
\end{split}\end{equation}
as wanted.
\end{proof}

Since a vector product can only the defined over a space of dimension $0, 1,3$ or $7$ (see e.g \cite{vectorprod} for a concise self-contained algebraic proof), then it follows that normed division algebras can only exist in dimensions $1$, $2$, $4$ and $8$. This result is commonly called Hurwitz Theorem and these four normed division algebras are precisely the real $\re$, complex $\com$, quaternion $\quat$ and octonion $\oct$ algebras. Each of these algebras is a subalgebra of the next one and their description can also be visualized from the standard Cayley-Dickson doubling construction point of view \cite{baez}. Namely, such construction depicts the algebraic properties lost from each step to the other. For instance, $\com$ loses the real property ($\overline{a}=a$) and the field ordering property. In addition, from $\com$ to $\quat$ the commutative property is lost and finally when one arrives at the octonions $\oct$ associativity is dropped giving place to alternativity. The automorphism group of the octonion algebra will be of great importance to what follows. For now on, consider the octonion algebra $\oct$ as given in the last Chapter. Namely, take an orthonormal basis $\{1,e_1,\ldots,e_7\}$ and define the octonion product by the relation
\bege
e_je_k=-\delta_{jk}+c^i_{\;jk}e_i,
\enge
where the structure constants $c_{ijk}$ are totally anti-symmetric and equal to the unity for the cycles
\begin{equation}
(ijk)=(123),\;(145),\;(167),\;(246),\;(275),\;(374),\;(365).
\end{equation}
Furthermore, the octonion algebra is deeply considered in the literature and several applications may come forth, for instance in \cite{Kuznetsova:2006ws, Toppan:2003ry, Toppan:2003dt} one may perceive the relation between (split-)division algebras and super-symmetry and the emergence of exceptional structures in for physical theories.
\section{The Exceptional Group $G_2$}

In order to establish the basic results on $G_2$-structures over $7$-dimensional manifolds $M$ one may first analyze the group $G_2$ itself. Consider the vector space $\re^7$ endowed with the usual Euclidean metric $g_o$, the orthonormal basis $\{e_1,\ldots,e_7\}$ and the volume form $\text{vol}_o=e^1\wedge\cdots\wedge e^7$ associated with $g_o$. In the light of eqns (\ref{intro335}, \ref{intro336}), a $3$-form $\varphi_o$, $4$-form $\psi_o$ and vector cross product $\times_o$ related to these structures can be considered. By the identification $\re^7\simeq\imm(\oct)$, one can set
\begin{equation}\begin{split}
    \varphi_o(\alpha,\beta,\gamma)&=\frac{1}{2}\langle[\alpha,\beta],\gamma\rangle=\langle\alpha\times_o\beta,\gamma\rangle,\\
    \psi_o(\alpha,\beta,\gamma,\delta)&=\frac{1}{2}\langle [\alpha,\beta,\gamma],\delta\rangle,
\end{split}\end{equation}
for every $\alpha,\beta,\gamma,\delta\in\re^7$. These are the $3$- and $4$- forms given in Definition (\ref{intro334}). Explicitly, letting $e^i\wedge e^j\wedge e^k=e^{ijk}$ where $\{e^1,\ldots,e^7\}$ is the associated dual basis one has
\begin{align}
    \varphi_o&=e^{123}+e^{145}+e^{167}+e^{246}-e^{257}-e^{347}-e^{356},\label{varphi}\\
    \psi_o&=e^{4567}+e^{2367}+e^{2345}+e^{1357}-e^{1346}-e^{1256}-e^{1247}\label{psi}.
\end{align}
\begin{preremark}\upshape
Note that depending on the choice of octonion product, the descriptions of $\varphi_o$ and $\psi_o$ may vary from text to text. In \cite{intro} a different definition is taken, whereas in \cite{grigorian, bryant} the ones used here can be found. 
\end{preremark}
Let $\star_o$ be the Hodge-star operation induced by $g_o$ and $\text{vol}_o$. Then, it is straightforward to see that
\begin{equation}
    \psi_o=\star_{o}\varphi_o.
\end{equation}
Also, notice that
\begin{equation}
\psi_o\wedge\varphi_o=7\text{vol}_{o},
\end{equation}
so that 
\begin{equation}
    \Vert \varphi_o\Vert^2=\Vert\psi_o\Vert^2=7.
\end{equation}
\begin{definition}\label{g2g2}
The group $G_2\subset \text{GL}(7,\re)$ is defined by
\begin{equation}
    G_2=\{ T\in \text{GL}(7,\re)\;:\;T^*(g_o)=g_o,\; T^*(\emph{\text{vol}}_o)=\emph{\text{vol}}_o,\;T^*(\varphi_o)=\varphi_o\},
\end{equation}
where $T^*$ denotes the pull-back by $T\in\text{GL}(7,\re)$.
\end{definition}
\begin{preremark}\upshape
Definition \ref{g2g2} is one of the (many) possible for the exceptional group $G_2$. Note that $G_2\subset \text{SO}(7)$ since it preserves the metric $g_o$ and orientation $\text{vol}_0$. Hence, it also preserves the Hodge star $\star_o$ and $\psi_o$, and since $\times_o$ is completely defined by the metric and $\varphi_o$ then it preserves the vector cross product as well. Nevertheless, the following result shows that an application $T\in\text{GL}(7,\re)$ only needs to preserve the $3$-form $\varphi_o$ in order to be in $G_2$ \cite{bryant}.
\end{preremark}
\begin{theorem}
$G_2=\{ T\in \text{GL}(7,\re)\;:\;T^*\varphi_o=\varphi_o\}.$
\end{theorem}
\begin{proof}
In terms of the standard dual basis $\{e^1,\ldots,e^7\}$ it is easy to show that
\begin{equation}\label{intro45}
    (u\iprod\varphi_o)\wedge(v\iprod\varphi_o)\wedge\varphi_o=6g_o(u,v)\text{vol}_o.
\end{equation}
Now, one may take $T\in\text{GL}(7,\re)$ such that $T^*\varphi_o=\varphi_o$. It follows from this relation that
\begin{equation}\label{intro46}
    (T^*g_o)(u,v)g^*\text{vol}_o=g_o(T(u),T(v))\det(T)\text{vol}_o=g_o(u,v)\text{vol}_o.
\end{equation}
This implies that
\begin{equation}
\det(T)g_o(T(u),T(v))=g_o(u,v).\end{equation}
Taking the determinant of the previous relation gives
\begin{equation}
\det(T)^9\det(g_o)=\det(g_o)\end{equation}
so that $\det(T)=1$ and then $T^*(\text{vol}_o)=\text{vol}_o$. But then, by eqn (\ref{intro46}) there also holds $T^*(g_o)=g_o$, which proves the claim.
\end{proof}
\begin{corollary}
The group $G_2$ is equal to the automorphism group $\emph{\text{Aut}}(\oct)$ of the octonion algebra $\oct$.
\end{corollary}
\begin{proof}
Let $T\in\text{Aut}(\oct)$ and $\alpha\in\imm(\oct)$. As before, there holds $\alpha^2=-\alpha\overline{\alpha}=-\Vert  \alpha\Vert^2_o$. Then, it follows that
\begin{equation}
T\lp\alpha\rp^2=T\lp\alpha^2\rp=T\lp-\Vert \alpha\Vert^2_o\rp=-\Vert \alpha \Vert^2_o,\end{equation}
since $T(a)=a$, for every $a\in\re$. It follows from Corollary \ref{intro325} that $T(\alpha)$ is either real or imaginary. Nonetheless, if it is real, say $T(\alpha)=a\in\re$, then $T(\alpha)=T(a)$ and since $T$ is an automorphism there holds $\alpha=a$, which contradicts $\alpha$ being imaginary. Then, it follows that $\overline{T(\alpha)}=-T(\alpha)$ whenever $\alpha$ is imaginary.

Now, take $A=\ree(A)1+\imm(A)\in\oct$. Notice that
\begin{equation} T(A)=\ree(A)1+T(\imm(A)),\end{equation}
since $T$ is linear and is the identity over $\ree(\oct)$. Therefore, $\overline{T(A)}=T(\overline{A})$ and it follows that
\begin{equation}
\Vert T(A)\Vert^2_o=T(A)\overline{T(A)}=T(A)T(\overline{A})=T(A\overline{A})=T(\Vert A\Vert^2_o)=\Vert A\Vert^2_o.\end{equation}
Hence, $\Vert T(A)\Vert_o=\Vert A\Vert_o$. Now, since $T(\imm(\oct))\subset \imm(\oct)$ and $T(1)=1$ one has $T\in \text{O}(7)$. This implies that
\begin{equation}\begin{split}
    (T^*\varphi_o)(\alpha,\beta,\gamma)&=\varphi_o(T(\alpha),T(\beta),T(\gamma))=\langle T(\alpha)T(\beta),T(\gamma)\rangle\\
    &=\langle T(\alpha\beta),T(\gamma)\rangle=\langle \alpha\beta,\gamma\rangle=\varphi_o(\alpha,\beta,\gamma),
\end{split}\end{equation}
for every $\alpha, \beta,\gamma\in\imm(\oct)=\re^7$. Therefore, $T\in G_2$.

Conversely, if $T\in G_2$ then it preserves the cross product and inner product over $\re^7\simeq\imm(\oct)$. Then, extending $T$ to $\re\oplus\imm(\oct)\simeq\oct$ with $T(1)=1$ it follows from eqn (\ref{intro360}) that
$T(AB)=T(A)T(B)$, for every $A,B\in\oct$, so that $T\in\text{Aut}(\oct)$.
\end{proof}
\begin{preremark}\upshape
Notice here that if $\alpha,\beta,\gamma\in\re^7$ are such that $\varphi_o(\alpha,\beta,\gamma)=0$ then in the light of eqn (\ref{intro345}) it follows that the cross product of any two elements in $\{\alpha,\beta,\gamma\}$ is orthogonal with respect to the other. In fact, a choice of a triple of vectors in $\re^7$ with such property can be seen to completely define the group $G_2$.
\end{preremark}
\begin{lemma}
Let $\{h_1,h_2,h_4\}$ be a triple of orthonormal vectors in $\re^7$ such that $\varphi_o(h_1,h_2,h_4)=0$. One may define
\begin{equation}
    h_1\times_o h_2=h_3,\;\;\;\; h_1\times_o h_4=h_5,\;\;\;\;h_2\times_o h_4=h_6\;\;\;\;h_4\times_o h_3=h_4\times_o (h_1\times_o h_2)=h_7.
\end{equation}
It follows that $\{h_1,\ldots,h_7\}$ is an oriented orthonormal basis for $\re^7$.
\end{lemma}
\begin{proof}
It is tedious but straightforward to see that $\langle h_i,h_j\rangle=\delta^i_j$ for every $1\leq i,j\leq 7$. Now, notice that if $h_j=e_j$ for $j\in\{1,2,4\}$ then the same relation is true for all $j\in\{1,\ldots,7\}$ by the octonion multiplication table. Since $e_1, e_2$ and $e_4$ are orthonormal to each other, there is $T\in\text{SO}(7)$ such that $T(e_j)=h_j$ for each $j\in\{1,2,4\}$. But then the pull-back $T^*$ takes the identity matrix in $\text{SO}(7)$ to
\begin{equation} A=(h_1|h_2|h_3|h_4|h_5|h_6|h_7),\end{equation}
which is the matrix with columns given by the elements of $\{h_1,\ldots,h_7\}$. It follows then that $A\in\text{SO}(7)$ and, hence, $\{h_1,\ldots,h_7\}$ is oriented, which concludes the proof.
\end{proof}
\begin{corollary}
The exceptional group $G_2$ can be perceived as the subgroup of $\emph{\text{SO}}(7)$ consisting of elements $T\in\emph{\text{SO}}(7)$ of form
\begin{equation}\label{intro415}
T=(h_1|h_2|h_1\times_o h_2|h_4|h_1\times_o h_4|h_2\times_o h_4|(h_1\times_o h_2)\times_o h_4),
\end{equation}
where $\{h_1,h_2,h_4\}$ is an orthonormal triple such that $\varphi_o(h_1,h_2,h_4)=0$. Furthermore, 
\begin{equation}\dim G_2=14.\end{equation}
\end{corollary}
\begin{proof}
As previously seen, a matrix $T\in\text{SO}(7)$ is in $G_2$ is and only if it preserves the cross product $\times_o$. Now, since $T(e_i)=h_i$, the result follows from the definition of the octonion product and its cross products. Furthermore, since an element in $G_2$ is fully characterized by a triple $(h_1,h_2,h_4)$ of orthonormal vectors in $\re^7$ with $\varphi_o(h_1,h_2,h_4)=0$ then in order to choose the first vector for such triple one must take $h_1\in S^6$, since it must have unit norm. Then, choosing an unitary $h_2$ orthogonal to $h_1$ means choosing $h_2$ in the $5$-sphere orthogonal to $h_1$. Finally, the unitary $h_4$ must be orthogonal to $h_1$, $h_2$ and $h_1\times_o h_2$, therefore lying in a $3$-sphere. It follows
\begin{equation}
    \dim G_2=\dim S^6+\dim S^5+\dim S^3=14.
\end{equation}
\end{proof}

\chapter{$G_2$-Structures}

As seen in the first Chapter, given a Riemannian manifold $(M,g)$ one may consider the normal coordinates at $p\in M$, which have the property that the metric $g_p$ over the tangent space $T_p M$ is precisely the Euclidean one in the $n$-dimensional vector space $\re^n$ in these coordinates. This property may be condensed as follows: consider the \textbf{Frame bundle} $\text{Fr}(M)$ defined as a (principal $\text{GL}(7,\re)$) bundle with the projection $\pi:\text{Fr}(M)\rightarrow M$ for which a fiber at $p\in M$ consists of all frames (bases) for $T_pM$, namely
\begin{equation}
\pi^{-1}(\{p\})=\{T:\re^n\rightarrow T_pM\;:\;T\;\text{is a linear isomorphism}\}.
\end{equation}
Then, for a Riemannian manifold $(M,g)$, the property that in normal coordinates there holds $g_{ij}(p)=\delta_{ij}$ can be stated as follows: for every $p\in M$, there is $T\in\pi^{-1}(\{p\})$ such that
\bege\label{gstru}
T^*g_p=\langle\cdot,\cdot\rangle,
\enge
where $\langle\cdot,\cdot\rangle$ is the Euclidean metric in $\re^n$. Notice that the choice of isomorphism $T\in\pi^{-1}(\{p\})$ is exactly the choice of basis by means of eqn (\ref{amor}). Moreover, note how if $g\in\mathcal{S}^2(T^*M)$ is just a symmetric $2$-tensor over $M$ with the property in eqn (\ref{gstru}), then $g$ would necessarily be a Riemannian metric over $M$. Moreover, in that case the subgroup $G\subset \text{GL}(n,\re)$ preserving that relation is precisely the orthogonal group $G=\text{O}(n)$, so that one may say that $g$ is an $\text{O}(n)$-structure.

Let now $M$ be a 7-dimensional manifold. It is then possible to endow such manifold with the octonionic structure discussed heretofore. The natural way to do this is to consider the $G_2$-structure over $M$, given in terms of
\begin{definition}\label{g2estruturas}
A {\bm$G_2$}\textbf{-structure} over a 7-dimensional manifold $M$ is a 3-form $\varphi\in\Omega^3(M)$ such that for each $p\in M$ there is a linear isomorphism $T:\re^7\rightarrow T_p M$ with the property that \begin{equation}T^*\varphi=\varphi_o.\end{equation}
In that case, for simplicity the pair $(M,\varphi)$ is also called a $G_2$-structure.
\end{definition}

To make more sense of the last definition, as seen in \cite{bryant} one may let $\pi:\text{Fr}(M)\rightarrow M$ be the Frame Bundle projection as previously defined. Then, define the map \begin{equation}\tilde{\pi}:\text{Fr}(M)\rightarrow \Lambda^3(T^*M)\end{equation} by the relation \begin{equation}\tilde{\pi}(T)=\lp T^{-1}\rp^*(\varphi_o),\end{equation} where $T:\re^7\rightarrow TM$ is a linear isomorphism. Then, if $\tilde{\pi}(T_1)=\tilde{\pi}(T_2)$ let $G\in \text{GL}(7,\re)$ be such that $T_1=G\circ T_2$. It follows that \begin{equation}\lp G^{-1}\rp^*(\varphi_o)=\varphi_o\end{equation} and hence $G\in G_2$. Therefore, the fibers over the bundle map $\tilde{\pi}$ are exactly the $G_2$-orbits in $\text{Fr}(M)$, that is,
\begin{equation}
\tilde{\pi}(\text{Fr}(M))\simeq \text{Fr}(M)/G_2.\end{equation}
One may denote $\tilde{\pi}(\text{Fr}(M))=\Lambda^3_+(T^*M)$. Then, for $p\in M$ the elements in the fiber $\Lambda^3_+(T_p^* M)$ are $3$-forms $\varphi_p\in\Lambda^3(T_p^* M)$ such that there is $T:T_p M\rightarrow \re^7$ with $\varphi_p=T^*(\varphi_o)$. Notice now that since $\dim \text{GL}(7,\re)=49$ and $\dim G_2=14$, it follows that 
\begin{equation}\dim \text{Fr}(M)/ G_2=49-14=35=\dim \Lambda^3(T^*M),\end{equation}
and hence $\Lambda_+^3(T_p^*M)$ is open in $\Lambda^3(T^*_p M)$. The local sections of such bundle over an open set $U\subset M$ may be denoted $\Omega^3_+(U)$ and are called \textbf{positive (or definite) 3-forms} and by Definition \ref{g2estruturas} this is precisely the space of $G_2$-structures over $U$. Now, the existence of a (global) $G_2$-structure over a 7-dimensional manifold $M$ is purely topological. More specifically, $M$ admits a $G_2$-structure if an only if it admits a spin structure and if $M$ is orientable \cite{grigorian, lawson}. Overall, a $G_2$-structure is in one-to-one correspondence with the open subset $\Omega^3_+(M)\subset\Omega^3(M)$ of positive 3-forms over $M$.

\begin{preremark}\upshape
Following \cite{spirotudo,spirocontas}, given a $G_2$-structure $\varphi$ over a $7$-dimensional manifold $M$, then by definition, for every point $p\in M$, there is an isomorphism $T:T_pM\rightarrow \re^n$ for which
\begin{equation}
    \varphi_p(u,v,w)=\langle T(u)\times_o T(v),T(w)\rangle,
\end{equation}
for every $u,v,w\in T_p M$, where $\langle\cdot,\cdot\rangle$ is the usual Euclidean inner product and $\times_o$ the cross vector product as given in the previous section. Then, one may look for a metric $g$ that globally describes this local behaviour for a fixed $G_2$-structure $\varphi$ over $M$. In what follows, $\langle \cdot,\cdot\rangle$ shall denote the desired metric, with $\Vert\cdot\Vert^2=\langle\cdot,\cdot\rangle$. In addition, since $M$ is oriented, one may take the volume form $\text{vol}$ with respect to this metric, which in turn defines a Hodge star denoted by $\star$. 
\end{preremark}

\begin{lemma}\cite{spirotudo,spirocontas}\label{tudoaqui}
For every $1$-form $\alpha$ and vector field $X$ over $M$ the following identities hold:

    \begin{align}
        \Vert \varphi\Vert^2&=7, & \Vert \psi\Vert^2&=7,\\
        \Vert \varphi\wedge\alpha\Vert^2&=4\Vert \alpha\Vert^2, &  \Vert\psi\wedge\alpha\Vert^2&=3\Vert\alpha\Vert^2\label{tudo28}\\
        \star(\varphi\wedge\star(\varphi\wedge\alpha))&=-4\alpha, & \star(\psi\wedge\star(\psi\wedge\alpha))&=3\alpha\\
        \psi\wedge\star(\varphi\wedge\alpha)&=0, & \varphi\wedge\star(\psi\wedge\alpha)&=2\psi\wedge\alpha \label{tudo26}\\
        \star(\varphi\wedge X^{\flat})&=X\iprod\psi, &  \star(\psi\wedge X^{\flat})&=X\iprod\varphi\\
        \varphi\wedge( X\iprod\varphi)&=2\star(X\iprod\varphi), &  \psi\wedge(X\iprod\varphi)&=3\star X^{\flat}\\
        \varphi\wedge(X\iprod \psi)&=-4\star X^{\flat}, &  \psi\wedge(X\iprod\psi)&=0.
    \end{align}

\end{lemma}
\begin{proof}
Every identity can be straightforwardly calculated using the pointwise description of $\varphi$ and $\psi$ given by eqns (\ref{varphi}, \ref{psi}) and Lemma \ref{lemmaimport}.
\end{proof}

\begin{proposition}
Let $X$ be a vector field over $M$. Then,
\begin{equation}\label{tudo225}
    (X\iprod\varphi)\wedge(X\iprod\varphi)\wedge\varphi=6\Vert X\Vert^2\emph{vol}.
\end{equation}
\end{proposition}
\begin{proof}
By eqn (\ref{tudoa01}), there holds
\begin{equation}
*(X^{\flat}\wedge\psi)=X\iprod\varphi,\end{equation}
which with eqns (\ref{tudo26}, \ref{tudo28}) yield
\begin{equation}
(X\iprod\varphi)\wedge\varphi=2(X^{\flat}\wedge\psi)\end{equation}
and then
\begin{equation}(X\iprod\varphi)\wedge(X\iprod\varphi)\wedge\varphi=2\Vert X^{\flat}\wedge\psi\Vert^2\text{vol}=6\Vert X\Vert^2\text{vol}.\end{equation}
\end{proof}
Polarizing eqn (\ref{tudo225}) in $X$ yields, for another vector field $Y$ over $M$
\begin{equation}
    (X\iprod\varphi)\wedge(Y\iprod\varphi)\wedge\varphi=\langle X,Y\rangle\text{vol}.
\end{equation}

\begin{lemma}
Let $\{e_1,\ldots,e_7\}$ be a choice of local frame for the tangent bundle $TM$ over some neighbourhood of $M$. Then, if locally $X=X^ke_k$ then the expression
\begin{equation}\label{tudo226}
    \frac{\lp\lp X\iprod\varphi\rp\wedge\lp X\iprod\varphi\rp\wedge\varphi\rp(e_1,\ldots,e_7)}{\lp\det\lp\lp\lp e_i\iprod\varphi\rp\wedge\lp e_j\iprod\varphi\rp\wedge\varphi\rp\lp e_1,\ldots,e_7\rp\rp\rp^{\frac{1}{9}}}
\end{equation}
does not depend on the choice of frame $\{e_1,\ldots,e_7\}$.
\end{lemma}
\begin{proof}
Let $\{ e'_1,\ldots,e'_7\}$ be another frame and let
\begin{equation}
e'_i=A_i^je_j.\end{equation}
There holds
\begin{equation}\lp e'_i\iprod\varphi\rp\wedge\big{(} e'_j\iprod\varphi\big{)}\wedge\varphi=A_i^kA_j^l\lp e_k\iprod\varphi\rp\wedge\lp e_l\iprod\varphi\rp\wedge\varphi,\end{equation}
and then the denominator of eqn (\ref{tudo226}) changes by a factor of
\begin{equation}\lp\det(A)^2\det(A)^7\rp^{\frac{1}{9}}=\det(A).\end{equation}
Since the numerator also changes by a factor of $\det(A)$ by eqn (\ref{tudo225}), the proof is finished.
\end{proof}

\begin{theorem}
Let $X_p\in T_pM$ and $\{e_1,\ldots,e_7\}$ be a basis for the tangent space. Then,
\begin{equation}\label{tudo227}
\Vert X_p\Vert^2=6^{-\frac{2}{9}}\frac{\lp\lp X\iprod\varphi\rp\wedge\lp X\iprod\varphi\rp\wedge\varphi\rp(e_1,\ldots,e_7)}{\lp\det\lp\lp\lp e_i\iprod\varphi\rp\wedge\lp e_j\iprod\varphi\rp\wedge\varphi\rp\lp e_1,\ldots,e_7\rp\rp\rp^{\frac{1}{9}}}
\end{equation}
\end{theorem}
\begin{proof}
Fix $\det(g)=\det(g_{ij})$, where $g_{ij}=\langle e_i,e_j\rangle$. Then, from eqn (\ref{tudo225}) there follows
\begin{equation}\begin{split}
    \lp\lp e_i\iprod\varphi\rp\wedge\lp e_j\iprod\varphi\rp\wedge\varphi\rp &=6g_{ij}\text{vol}\\
    &=6g_{ij}\sqrt{\det(g)}e^{1234567},
\end{split}\end{equation}
hence
\begin{equation}\begin{split}
    \det\lp\lp\lp e_i\iprod\varphi\rp\wedge\lp e_j\iprod\varphi\rp\wedge\varphi\rp\lp e_1,\ldots,e_7\rp\rp&=6^7\det(g)\det(g)^{\frac{7}{2}}\\
    &=6^7\det(g)^{\frac{9}{2}}.
\end{split}\end{equation}
Finally, calculating the numerator of eqn (\ref{tudo227}) comes
\begin{equation}\begin{split}
    \lp X_p\iprod\varphi\rp\wedge\lp X_p\iprod\varphi\rp\wedge\varphi&=6\Vert X_p\Vert^2\text{vol}\\
    &=6\Vert X_p\Vert^2\sqrt{\det(g)}e^{1234567},
\end{split}\end{equation}
and then
\begin{equation}
    \lp\lp X_p\iprod\varphi\rp\wedge\lp X_p\iprod\varphi\rp\wedge\rp\lp e_1,\ldots,e_7\rp=6\Vert X_p\Vert^2\det(g)^{\frac{1}{2}},
\end{equation}
which combined with the denominator calculation yields the desired result.
\end{proof}
In the light of the last result, polarizing the expression (\ref{tudo227}) comes
\begin{equation}
    \langle X_p,Y_p\rangle = 6^{-\frac{2}{9}}\frac{\lp\lp X_p\iprod\varphi\rp\wedge\lp Y_p\iprod\varphi\rp\wedge\varphi\rp(e_1,\ldots,e_7)}{\lp\det\lp\lp\lp e_i\iprod\varphi\rp\wedge\lp e_j\iprod\varphi\rp\wedge\varphi\rp\lp e_1,\ldots,e_7\rp\rp\rp^{\frac{1}{9}}}. 
\end{equation}
\begin{preremark}\upshape
In conclusion, for every $G_2$-structure $\varphi$ over $M$ there is a 7-form valued bilinear form $\mathcal{B}_{\varphi}$ given by
\bege
\mathcal{B}_{\varphi}(X,Y)=\frac{1}{6}\lp X\iprod\varphi\rp\wedge\lp Y\iprod\varphi\rp\wedge\varphi
\enge
for which there are an unique Riemannian metric $g_{\varphi}$ and volume form $\text{vol}_{\varphi}$ such that
\bege\label{riemann}
g_{\varphi}(X,Y)\text{vol}_{\varphi}=\mathcal{B}_{\varphi}( X,Y),
\enge
for every vector field $X$ and $Y$. In local coordinates there holds
\begin{equation}
    (g_{\varphi})_{ij}=\frac{1}{6^{\frac{2}{9}}}\frac{( \mathcal{B}_{\varphi})_{ij}}{\det(\mathcal{B}_{\varphi})^{\frac{1}{9}}}.
\end{equation}
One then says that $g_{\varphi}$ and $\text{vol}_{\varphi}$ are respectively the metric and volume form \textbf{associated with} $\varphi$. The 3-form subscript is usually lost whenever it is clear which $G_2$-structure is being considered. Then, eqn (\ref{intro345}) can be generalized over the manifold $M$. Namely, for vector fields $X,Y,Z\in\mathfrak{X}(M)$ the vector cross product \begin{equation}\times:\mathfrak{X}(M)\times\mathfrak{X}(M)\rightarrow\mathfrak{X}(M)\end{equation}
is defined be the relation
\begin{equation}\label{vectorcrossprod}
    \varphi(X,Y,Z)=\langle X\times Y,Z\rangle.
\end{equation}
Then, it follows from eqn (\ref{intro363}) that
\begin{equation}\label{fundamentalrel}
    X\times(Y\times Z)=-\langle X,Y\rangle Z+\langle X,Z\rangle Y-\psi(X,Y,Z,\cdot)^{\sharp},
\end{equation}
where the sharp isomorphism is taken with respect to the associated metric. Additionally, from (\ref{intro373}) it comes
\begin{equation}\label{fundamentalrel2}
    \langle X\times Y,Z\times W\rangle=\langle X\wedge Y,Z\wedge W\rangle +\psi(X,Y,Z,W).
\end{equation}
The following result, which depicts several important relations between the $\varphi$ and $\psi$ tensors handled so far, may then be perceived as being direct consequences of these previous equations.
\end{preremark}
\begin{theorem}\label{theoremcontrac}
Let $\varphi$ be a $G_2$-structure over a 7-dimensional manifold $M$ with associated metric $g$. Then, there holds
\begin{align}
    \varphi_{ijk}\varphi_{abc}g^{ck}&=g_{ia}g_{jb}-g_{ib}g_{ja}+\psi_{ijab},\label{intro424}\\
    \varphi_{ijk}\varphi_{abc}g^{bj}g^{ck}&=6g_{ia},\label{intro425}\\
    \varphi_{ijk}\psi_{abcd}g^{dk}&=-g_{ia}\varphi_{jbc}-g_{ib}\varphi_{ajc}-g_{ic}\varphi_{abj}+g_{aj}\varphi_{ibc}+g_{bj}\varphi_{aic}+g_{cj}\varphi_{abi},\label{intro426}\\
    \varphi_{ijk}\psi_{abcd}g^{cj}g^{dk}&=4\varphi_{iab},\label{intro427}\\
    \psi_{ijkl}\psi_{abcd}g^{ck}g^{dl}&=4g_{ia}g_{jb}-4g_{ib}g_{ja}+2\psi_{ijab},\label{intro428}\\
    \psi_{ijkl}\psi_{abcd}g^{bj}g^{ck}g^{dl}&=24g_{ia.}\label{intro429}
\end{align}
\end{theorem}
\begin{proof}
One may consider the volume form $\text{vol}$, Hodge star $\star$, vector cross product $\times$ and as usual $\psi=\star\varphi$. Fixing the notation, one may take local coordinates $(U;x^1,\ldots,x^7)$ for which
\begin{align}
    \varphi&=\frac{1}{6}\varphi_{ijk}dx^i\wedge dx^j\wedge dx^k,\\
    \psi&=\frac{1}{24}\psi_{ijkl}dx^i\wedge dx^j\wedge dx^k\wedge dx^l.
\end{align}
As usual, set $g_{ij}=g(\partial_i,\partial_j)$ and let
\begin{equation}
    \partial_i \times \partial_j=A^k_{\;ij}\partial_k.
\end{equation}
Then, by eqn (\ref{vectorcrossprod}) it follows that
\begin{align}
    \varphi_{ijk}&=A^l_{\;ij}g_{lk}\label{introA9i}\\
    \varphi_{ijk}g^{kl}&=A^l_{\;ij}\label{introA9ii}.
\end{align}
Now, from eqn (\ref{fundamentalrel}) one can see that
\begin{equation}\label{aquelala}
\begin{split}
    \partial_i\times(\partial_j\times\partial_k)&=-g_{ij}\partial_k+g_{ij}\partial_j-\psi_{ijkl}(dx^l)^{\sharp}\\
    A^m_{\;il}A^l_{\;jk}\partial_m&=-g_{ij}\partial_k+g_{ik}\partial_j-\psi_{ijkl}g^{lm}\partial_m.
    \end{split}
\end{equation}
Then, inserting this expression in the metric with $\partial_n$ yields
\begin{equation}\begin{split}
    A^m_{il}A^l_{jk}g_{mn}&=-g_{ij}g_{kn}+g_{ik}g_{jn}-\psi_{ijkl}g_{mn}g^{ml}\\
    \varphi_{ilp}g^{pm}g_{mn}\varphi_{jkb}g^{bl}&=-g_{ij}g_{kn}+g_{ik}g_{jn}-\psi_{ijkn}\\
    \varphi_{inl}\varphi_{jkb}g^{bl}&=-g_{ik}g_{in}+g_{ij}g_{kn}+\psi_{injk},
\end{split}\end{equation}
where $\varphi_{iln}=-\varphi_{inl}$ was used. This proves the first equation. The second equation is then straightforwardly derived by contraction with $g^{bj}$.

For the third and fourth equations (where again the latter is obtained by contracting the former with $g^{cj}$), one may calculate $g(\partial_a\times\partial_b,\partial_i\times(\partial_j\times\partial_k))$. First, by eqn (\ref{fundamentalrel2}) there holds
\begin{equation}\begin{split}
    g(\partial_a\times\partial_b,\partial_i\times(\partial_j\times\partial_k))&=g(\partial_a\wedge\partial_b,\partial_i\wedge(\partial_j\times\partial_k))+\psi(\partial_a,\partial_b,\partial_i,\partial_j\times\partial_k)\\
    &=g(\partial_a,\partial_i)g(\partial_b,\partial_j\times\partial_k)-g(\partial_a,\partial_j\times\partial_k)g(\partial_b,\partial_i)+\psi_{abil}A^l_{\;jk}\\
    &=g_{ai}g_{bl}A^l_{\;jk}-g_{al}g_{bi}A^l_{\;jk}+\psi_{abil}A^l_{\;jk}\\
    &=g_{ai}\varphi_{bjk}-g_{bi}\varphi_{ajk}+\psi_{abil}\varphi_{njk}g^{nl}.
\end{split}\end{equation}
On the other hand, using (\ref{aquelala}) there follows
\begin{equation}\begin{split}
    g(\partial_a\times\partial_b,\partial_i\times(\partial_j\times\partial_k)&=g(A^m_{\;ab}\partial_m,-g_{ij}\partial_k+g_{ik}\partial_j-\psi_{ijkl}g^{ln}\partial_n)\\
    &=-g_{ij}g_{mk}A^m_{\;ab}+g_{mj}g_{ik}A^m_{\;ab}-\psi_{ijkl}g^{ln}A^m_{\;ab}g_{mn}\\
    &=-g_{ij}\varphi_{kab}+g_{ik}\varphi_{jab}-\psi_{ijkl}g^{ln}\varphi_{nab}.
\end{split}\end{equation}
Hence, one may write the expression
\begin{equation}
T_{ijkab}=g_{ia}\varphi_{jkb}-g_{ib}\varphi_{jka}+g_{ij}\varphi_{abk}-g_{ik}\varphi_{abj}+\varphi_{jkn}\psi_{abil}g^{nl}+\varphi_{abn}\psi_{ijkl}g^{ln}=0\end{equation}
It is not enlightening to present explicitly but one can see that from
\begin{equation}
T_{ijkab}+T_{ajkbi}+T_{bijka}-T_{kijab}-T_{jkabi}=0\end{equation}
it follows the desired result.

By a similar a reasoning, calculating \begin{equation}
g(\partial_a\times(\partial_b\times\partial_c),\partial_i\times(\partial_j\times\partial_k))\end{equation}
first using eqn (\ref{aquelala}) and then (\ref{fundamentalrel2}) yields
\begin{equation}\begin{split}
\psi_{abcd}\psi_{ijkl}g^{dl}&=-\varphi_{ajk}\varphi_{ibc}-\varphi_{iak}\varphi_{jbc}-\varphi_{ija}\varphi_{kbc}\\
&\;\;\;\;+g_{ia}g_{jb}g_{kc}+g_{bi}g_{ak}g_{jc}+g_{ci}g_{ja}g_{bk}\\
&\;\;\;\;-g_{ia}g_{jc}g_{kb}-g_{bi}g_{ja}g_{ck}-g_{ci}g_{ak}g_{jb}\\
&\;\;\;\;+g_{ja}\psi_{bcki}+g_{ai}\psi_{jkbc}+g_{ak}\psi_{ijbc}\\
&\;\;\;\;-g_{ab}\psi_{ijkc}+g_{ac}\psi_{ijkb},
\end{split}\end{equation}
which can be further contracted with $g^{ck}$ and $g^{bj}$, resulting in the two last identities.
\end{proof}
\section{$G_2$-Splittings of $\Omega(M)$}
Let $(M,\varphi)$ be a $G_2$-structure and fix $g$ its associated metric. In this section, an orthogonal $G_2$-splitting of the space of differential forms \begin{equation}\Omega(M)=\displaystyle\bigoplus_{k=1}^7\Omega^k(M)\end{equation}
the over the 7-dimensional manifold $M$ is analyzed. Such decomposition shall be $G_2$-invariant and will be proven to be a great tool to what follows. These results are extracted from \cite{intro,spirocontas}.

Any tensor defined by means of $\varphi$ shall be $G_2$-invariant and, therefore, also will the ones defined by $\psi$, $\star$ and the associated metric $g$. When $k=0,1,6$ or $7$, then $\Omega^k(M)$ is irreducible, but when $k=2,3,4$ or $5$, there exist non-trivial decompositions. Since there holds
\begin{equation}\Omega^k(M)=\star\Omega^{7-k}(M),\end{equation} 
it suffices to understand the decompositions of $\Omega^2(M)$ and $\Omega^3(M)$ and then take the Hodge star to understand their $4$ and $5$ dimensional counterparts. In what follows, one writes $\Omega^k_l$ to mean the $l$-dimensional part of $\Omega^k(M)$ according to this splitting.

\begin{theorem}
Let $(M,\varphi)$ be a $G_2$-structure. There is a $G_2$-invariant splitting of the space of $2$-forms given by
\begin{equation}
    \Omega^2(M)=\Omega^2_7\oplus\Omega^2_{14},
\end{equation}
where
\begin{align}
    \Omega^2_7&=\{\beta\in\Omega^2(M)\;:\;\star(\varphi\wedge\beta)=2\beta\}=\{ X\iprod\varphi\;:\;X\in\mathfrak{X}(M)\}\\
    \Omega^2_{14}&=\{\beta\in\Omega^2\;:\;\star(\varphi\wedge\beta)=-\beta\}=\{\beta\in\Omega^2\;:\;\beta\wedge\psi=0\}.
\end{align}
\end{theorem}
The proof is presented in what follows: indeed, one may define a map \begin{equation}R:\Omega^2(M)\rightarrow\Omega^2(M)\end{equation} by setting
\begin{equation}
    R(\beta)=\star(\varphi\wedge\beta).
\end{equation}
Then, in local coordinates one may write $\beta=\frac{1}{2}\beta_{ij}dx^i\wedge dx^j$ and $R(\beta)=\frac{1}{2}(R(\beta))_{ab}dx^a\wedge dx^b$. Calculating, it comes
\begin{equation}\begin{split}
    R(\beta)=\star(\varphi\wedge\beta)&=\frac{1}{2}\beta_{ij}\star(dx^i\wedge dx^j\wedge\varphi)=\frac{1}{2}\beta_{ij}g^{il}\partial_l\iprod\star(dx^j\wedge\varphi)\\
    &=-\frac{1}{2}\beta_{ij}g^{il}g^{jm}\partial_l\iprod\partial_m\iprod\psi=-\frac{1}{2}\beta_{ij}g^{il}g^{jm}(\frac{1}{2}\psi_{mlab}dx^a\wedge dx^b)\\
    &=\frac{1}{4}\beta_{ij}\psi_{lmab}g^{il}g^{jm}dx^a\wedge dx^b,
\end{split}\end{equation}
where the anti-symmetry of $2$-forms and eqn (\ref{tudoa01}) were used twice. It then follows that
\begin{equation}
    (R(\beta))_{ab}=-\frac{1}{2}\psi_{abcd}g^{ci}g^{dj}\beta_{ij}.
\end{equation}
Now, one can see that $R$ is self-adjoint, so that it splits $\Omega^2(M)$ orthogonally. Indeed, notice that
\begin{equation}\begin{split}
    (R^2(\beta))_{ab}&=\frac{1}{2}\psi_{abcd}g^{ci}g^{dj}(R(\beta))_{ij}=\frac{1}{4}\psi_{abcd}\psi_{ijst}g^{ci}g^{dj}g^{sp}g^{tq}\beta_{pq}\\
    &=\frac{1}{4}(4g_{as}g_{bt}-4g_{at}g_{}bs+2\psi_{abst})g^{sp}g^{tq}\beta_{pq}\\
    &=\beta_{ab}-\beta_{ba}+\frac{1}{2}\psi_{abst}g^{sp}g^{tq}\beta_{pq}\\
    &=2\beta_{ab}+(R(\beta))_{ab}.
\end{split}\end{equation}
Then $R^2=2\text{Id}+R$, so that $(R-2\text{Id})(R+\text{Id})=0$. It follows that the eigenvalues for $R$ are precisely $+2$ and $-1$, which gives the presented splitting for $\Omega^2(M)$.

Now, in order to derive the second description of each part $\Omega^2_k$ notice that if $X\in\mathfrak{X}(M)$, then \begin{equation}X\iprod \varphi= X^k\varphi_{ijk}.\end{equation} 
Besides, the condition \begin{equation}\psi\wedge\beta=0\end{equation} is equivalent by eqn (\ref{tudoa01}) to
\begin{equation}\beta_{ij}g^{il}g^{jm}\varphi_{lmk}=0.\end{equation}
\begin{proposition}\label{contracao2}
Let $\beta=\frac{1}{2}\beta_{ij}dx^i\wedge dx^j\in\Omega^2$. Then,
\begin{equation}\begin{split}
    \beta\in\Omega^2_7\;\;\;&\iff\;\;\;\beta_{ij}g^{il}g^{jm}\psi_{lmab}=4\beta_{ab}\;\;\;\iff\;\;\;\beta_{ij}=X^k\varphi_{ijk},\\
    \beta\in\Omega^2_{14}\;\;\;&\iff\;\;\;\beta_{ij}g^{il}g^{jm}\psi_{lmab}=-2\beta_{ab}\;\;\;\iff\;\;\;\beta_{ij}g^{il}g^{jm}\varphi_{lmk}=0.
\end{split}\end{equation}
\end{proposition}
\begin{proof}
In order to prove the statement one may use several contractions from Theorem \ref{theoremcontrac}. Starting with $\Omega^2_7$, suppose that \begin{equation}\beta_{ij}g^{il}g^{jm}\psi_{lmab}=4\beta_{ab}.\end{equation} 
Multiplying it by $\varphi_{rij}$ yields 
\begin{equation}\begin{split}
    \beta_{ij}g^{il}g^{jm}\psi_{lmab}\varphi_{rij}&=4\beta_{ab}\varphi_{rij}\\
    \beta_{ij}(4\varphi_{rab})&=4\beta_{ab}\varphi_{rij}\\
    \beta_{ij}\varphi_{rab}\varphi_s^{\;\;ab}&=\beta_{ab}\varphi_{rij}\varphi_s^{\;\;ab}\\
    \beta_{ij}(6g_{rs})&=\beta_{ab}\varphi_s^{\;\;ab}\varphi_{rij}\\
    \beta_{ij}&=\frac{1}{42}\beta_{ab}\varphi^{rab}\varphi_{rij}.
\end{split}\end{equation}
Then, just set $X^r=\frac{1}{42}\beta_{ab}\varphi^{rab}$. It follows that
\begin{equation}
\beta_{ij}=X^r\varphi_{rij},\end{equation}
which proves the first implication. Conversely, suppose $\beta_{ij}=X^k\varphi_{ijk}$ for some $X\in\mathfrak{X}(M)$. Then,
\begin{equation}\begin{split}
    \beta_{ij}g^{il}g^{jm}\psi_{lmab}&=X^k\varphi_{ijk}g^{il}g^{jm}\psi_{ablm}\\
    &=4X^k\varphi_{kab}\\
    &=4\beta_{ab},
\end{split}\end{equation}
as wanted.

For the $\Omega^2_{14}$ part, suppose \begin{equation}\beta_{ij}g^{il}g^{jm}\psi_{lmab}=-2\beta_{ab}.\end{equation}
Then, multiplying by $\varphi_{sij}$ yields
\begin{equation}\begin{split}
    \beta_{ij}\psi_{ablm}\varphi_{sij}g^{il}g^{jm}&=-2\beta_{ab}\varphi_{sij}\\
    4\beta_{ij}\varphi_{sab}&=-2\beta_{ab}\varphi_{sij},
    \end{split}\end{equation}
    and then further multiplying both sides by $g^{is}g^{ja}$ gives 
    \begin{equation}\begin{split}
    \beta_{ij}g^{is}g^{ja}\varphi_{sab}&=-\frac{1}{2}\beta_{ab}\varphi_{sij}g^{is}g^{ja}\\
    &=0.
\end{split}\end{equation}
Finally, if $\beta_{ij}g^{il}g^{jm}\varphi_{lmk}=0$ then since 
\begin{equation}\psi_{lmab}=-g_{la}g_{mb}+g_{lb}g_{ma}+\varphi_{lmr}\varphi_{ab}^{\;\;\;r}\end{equation}
there holds
\begin{equation}\begin{split}
    \beta_{ij}g^{il}g^{jm}\psi_{lmab}&=\beta_{ij}g^{il}g^{jm}(-g_{la}g_{mb}+g_{lb}g_{ma}+\varphi_{lmr}\varphi_{ab}^{\;\;\;r})\\
    &=-\beta_{ab}+\beta_{ba}\\
    &=-2\beta_{ab},
\end{split}\end{equation}
which completes the proof.
\end{proof}
\begin{preremark}\upshape 
Notice that since $G_2\subset \text{SO}(7)$ then, in the Lie algebra level, one has
\begin{equation}
    \mathfrak{g}_2\subset \mathfrak{so}(7)\simeq\Omega^2(M),
\end{equation}
in such a way that, in fact, it is possible to see that the $14$-dimensional part $\Omega^2_{14}$ of $\Omega^2(M)$ has
\begin{equation}
    \mathfrak{g}_2\simeq \Omega^2_{14}
\end{equation}
as Lie algebras. Since the splitting is orthogonal, there holds
\begin{equation}
    \lp\mathfrak{g}_2\rp^{\perp}=\Omega^2_{7}
\end{equation}
with respect to the associated metric $g$.
\end{preremark}
\begin{lemma}
If $\beta\in\Omega^2_{14}$ then
\begin{equation}\label{intro212}
    \beta_{ab}g^{bl}\varphi_{lpq}=\beta_{ql}g^{lm}\varphi_{map}-\beta_{pl}g^{lm}\varphi_{maq}.
\end{equation}
\end{lemma}
\begin{proof}
 By the last proposition, since $\beta\in\Omega^2_{14}$ one has $\beta_{ab}=-\frac{1}{2}\beta_{ij}g^{im}g^{jn}\psi_{mnab}$. Calculating, it comes
 \begin{equation}\begin{split}
     \beta_{ab}g^{bl}\varphi_{lpq}&=-\frac{1}{2}\lp\beta_{ij}g^{im}g^{jn}\psi_{mnab}\rp\varphi_{lpq}g^{bl}\\
     &=-\frac{1}{2}\beta_{ij}g^{im}g^{jn}\Big{(}-g_{pm}\varphi_{qna}-g_{pn}\varphi_{mqa}-g_{pa}\varphi_{mnq}\\
     &\;\;\;\;+g_{mn}\varphi_{pna}+g_{nq}\varphi_{mpa}+g_{aq}\varphi_{mnp}\Big{)}\\
     &=-\frac{1}{2}\lp\beta_{pj}g^{jn}\varphi_{qna}+\beta_{ip}g^{im}\varphi_{mqa}-\beta_{qj}g^{jn}\varphi_{pnq}-\beta_{iq}g^{im}\varphi_{mpa}\rp\\
     &=\beta_{ql}g^{lm}\varphi_{map}-\beta_{pl}g^{lm}\varphi_{maq},
 \end{split}\end{equation}
 where $\beta_{ij}g^{il}g^{jm}\varphi_{lmk}=0$ from last proposition was also used.
\end{proof}

\begin{proposition}
The space $\Omega^2_{14}\subset\Omega^2(M)$ is a Lie algebra with respect to the commutator
\begin{equation}
    [\beta,\mu]_{ij}=\beta_{il}g^{lm}\mu_{mj}-\mu_{il}g^{lm}\beta_{mj}.
\end{equation}
\end{proposition}
\begin{proof}
Since $\mathfrak{so}(7)\simeq\Omega^2(M)$ already has a Lie algebra structure, it suffices to show that the commutator is closed in $\Omega^2_{14}$. By the previous results, it is known that $[\beta,\mu]\in\Omega^2_{14}$ if and only if
\begin{equation}
[\beta,\mu]_{ij}g^{ia}g^{jb}\varphi_{abc}=0.\end{equation}
Hence, using eqn (\ref{intro212}) it follows that
\begin{equation}\begin{split}
    [\beta,\mu]_{ij}g^{ia}g^{jb}\varphi_{abc}&=\beta_{il}g^{lm}\mu_{mj}g^{ia}g^{jb}\varphi_{abc}-\mu_{il}g^{lm}\beta_{mj}g^{ia}g^{jb}\varphi_{abc}\\
    &=g^{lm}\mu_{mj}g^{jb}\lp\beta_{cr}g^{rs}\varphi_{slb}-\beta_{br}g^{rs}\varphi_{slc}\rp-\mu_{il}g^{lm}\beta_{mj}g^{ia}g^{jb}\varphi_{abc}\\
    &=-\beta_{br}g^{rs}\varphi_{slc}g^{lm}\mu_{mj}g^{jb}-\mu_{il}g^{lm}\beta_{mj}g^{ia}g^{ib}g^{jb}\varphi_{abc}\\
    &=-\varphi^r_{lc}\beta^j_r\mu^l_j+\varphi^b_{ac}\beta^l_j\mu^a_l\\
    &=0,
\end{split}\end{equation}
as desired.
\end{proof}

In order to analyze the space of $3$-forms $\Omega^3(M)$ let $(A^i_j)\in\text{M}(7,\re)=\mathfrak{gl}(7)$ be a real matrix. Then, $e^{tA}\in\text{GL}(7,\re)$ and one can consider the action
\begin{equation}
    e^{tA}\cdot \varphi=\frac{1}{6}\varphi_{ijk}(e^{tA}dx^i)\wedge(e^{tA}dx^j)\wedge(e^{tA}dx^k).
\end{equation}
Then, it follows that
\begin{equation}
\frac{d}{dt}{\Bigg |}_{t=0}e^{tA}\cdot\varphi=\frac{1}{6}(A^l_i\varphi_{ljk}+A^l_j\varphi_{ilk}+A^l_k\varphi_{ijl})dx^i\wedge dx^j\wedge dx^k.
\end{equation}
It is then possible to use the associated metric $g$ to identify the matrix $A\in\Gamma(T^*M\otimes TM)$ with a bilinear form $A=(A_{ij})=(A^l_ig_{lj})$. Notice that the space of sections $\Gamma(T^*M\otimes T^*M)$ of bilinear forms can be decomposed by
\begin{equation}
    \Gamma(T^*M\otimes T^*M)\simeq \mathcal{S}^2(M)\oplus \Omega^2(M),
\end{equation}
where as usual \begin{equation}\mathcal{S}^2(M)=\Gamma\lp Sym^2\lp T^*M\rp\rp.\end{equation}
The \textbf{trace} defined by $g$ 
\begin{equation}
\tr_g(A)=A_{ij}g^{ij}
\end{equation}
may be considered as well. Now, if $h\in \mathcal{S}^2(M)$ one can write its traceless part as \begin{equation}h_0= h - \frac{1}{7}\lp \text{Tr}_gh\rp g,\end{equation} yielding a decomposition
\begin{equation}
    \mathcal{S}^2(M)\simeq \Omega^0(M)\oplus \mathcal{S}^2_0(M)
    \end{equation}
    where $\mathcal{S}^2_0(M)$ corresponds to sections of traceless symmetric bilinear forms over $M$. Considering the already obtained decomposition of $\Omega^2(M)$, one can see that
    \begin{equation}
        \Gamma(T^*M\otimes T^*M)\simeq \Omega^0(M)\oplus \mathcal{S}^2_0(M)\oplus\Omega^2_{7}\oplus\Omega^2_{14}.
    \end{equation}
    Then, one may write
    \begin{equation}
        A=\frac{1}{7}(\tr A)g+A_0+A_7+A_{14},
    \end{equation}
    with $A_0$ traceless symmetric and $A_i\in\Omega^2_i$ for $i=7$ or $14$. Then, the application \begin{equation}F:\Gamma(T^*M\otimes T^*M)\rightarrow \Omega^3(M)\end{equation} given by
     \begin{equation}F(A) = \frac{d}{dt}{\Bigg |}_{t=0}e^{tA}\cdot\varphi\end{equation}
is a linear map between $\Omega^0\oplus \mathcal{S}^2_0(M)\oplus\Omega^2_{7}\oplus\Omega^2_{14}$ and $\Omega^3(M)$. The next result gives the $G_2$ splitting for $\Omega^3(M)$, as follows.
    
    \begin{theorem}
    Let $F: \Omega^0\oplus \mathcal{S}^2_0(M)\oplus\Omega^2_{7}\oplus\Omega^2_{14}\rightarrow \Omega^3$ be as previously defined. Then, its kernel is equal to $\Omega^2_{14}$ and the parts $\Omega^0$, $\mathcal{S}^2_0(M)$ and $\Omega^2_7$ are isomorphically mapped, respectively, onto $\Omega^3_1$, $\Omega^3_{27}$ and $\Omega^3_7$, which are given by
    \begin{align}
        \Omega^3_1&=\{ f\varphi\;:\;f\in C^{\infty}(M)\},\\
        \Omega^3_7&=\{ X\iprod \psi\;:\;X\in\mathfrak{X}(M)\},\\
        \Omega^3_{27}&=\{ h_{ij}g^{jl}dx^i\wedge(\partial_l\iprod\varphi)\;:\;h_{ij}=h_{ji},\;\tr_g(h)=0\}.
    \end{align}
    \end{theorem}
    \begin{proof}
    Since $G_2$ is the group preserving $\varphi$ then by definition $\mathfrak{g}_2=\ker F$. By dimensional count it suffices to show that $\Omega^2_{14}$ is inside the kernel. One may then write for $\beta\in\Omega^2(M)$ the decomposition
    \begin{equation}
    \beta_{ij}=\lp\beta_7\rp_{ij}+\lp\beta_{14}\rp_{ij},
    \end{equation}
    for which, as before, there holds
    \begin{equation}
    \lp\beta_{14}\rp_{ij}=\frac{1}{2}\lp\beta_{14}\rp_{ab}g^{ap}g^{bq}\psi_{pqij}.\end{equation}
    Then, using the contraction between $\varphi$ and $\psi$ one has
    \begin{equation}\begin{split}
    \lp F\lp\beta_{14}\rp\rp_{ijk}&=\lp\beta_{14}\rp^l_{i}\varphi_{ljk}+\lp\beta_{14}\rp^l_j
\varphi_{ilk}+\lp\beta_{14}\rp^l_k\varphi_{ijl}\\
&=2\lp\lp\beta_{14}\rp^l_{i}\varphi_{ljk}+\lp\beta_{14}\rp^l_j
\varphi_{ilk}+\lp\beta_{14}\rp^l_k\varphi_{ijl}\rp,\\
&=2\lp F\lp\beta_{14}\rp\rp_{ijk},
\end{split}\end{equation}
and then $F\lp\beta_{14}\rp=0$ so that $\Omega^2_{14}\subset \mathfrak{g}_2=\ker F$, as wanted. In addition, it follows that $F$ is injective in $\Omega^0\oplus\mathcal{S}^2_0(M)\oplus\Omega^2_7$.  

Maintaining the notation $\beta_{ij}=\lp\beta_7\rp_{ij}+\lp\beta_{14}\rp_{ij}$ for $\beta\in\Omega^2$, one may now analyze the image of $\Omega^2_7$ by $F$. From Proposition \ref{contracao2} it follows that
\begin{equation}
    (\beta_7)_{ij}=\beta^k\varphi_{kij},
\end{equation}
where \begin{equation}\beta^k=\frac{1}{42}(\beta_7)_{ij}\varphi_{abc}g^{kc}g^{ia}g^{jb}.\end{equation}
It then follows that
\begin{equation}\begin{split}
    (F(\beta_7))_{ijk}&=\frac{1}{42}((\beta_7)^n\varphi_{nil}g^{lm}\varphi_{mjk}+(\beta_7)^n\varphi_{njl}g^{lm}\varphi_{imk}+(\beta_7)^n\varphi_{nkl}g^{lm}\varphi_{ijm})\\
    &=\frac{1}{42}(\beta_7)^n(g_{nj}g_{ik}-g_{nk}g_{ij}+\psi_{nijk}+g_{nk}g_{ji}-g_{ni}g_{jk}-\psi_{njik}+g_{ni}g_{kj}-g_{nj}g_{ki}+\psi_{nkij})\\
    &=\frac{3}{42}(\beta_7)^n\psi_{nijk}\\
    &=X^n\psi_{nijk},
\end{split}\end{equation}
where $X^n=\frac{1}{14}(\beta_7)^n$. 
One may therefore conclude that \begin{equation}F(\Omega^2_7)=\{X\iprod\psi\;:\;X\in\mathfrak{X}(M)\},\end{equation}
which is denoted by $\Omega^3_7$.

The image through $F$ of the symmetric part $\mathcal{S}^2(M)=\Omega^0(M)\oplus\mathcal{S}^2_0(M)$ can then be perceived. Obviously, there holds \begin{equation}F(\Omega^0(M))=\{f\varphi\;:\;f\in C^{\infty}(M)\},\end{equation}
which is denoted $\Omega^3_1$. Besides, if $h_{ij}\in\mathcal{S}^2_0(M)$ then
\begin{equation}\begin{split}
    F(h_{ij})&=\frac{1}{6}(h^l_i\varphi_{ljk}+h^l_j\varphi_{ilk}+h^l_k\varphi_{ijl})dx^i\wedge dx^j\wedge dx^k\\
    &=\frac{1}{2}(h^l_i\varphi_{ljk})dx^i \wedge dx^j\wedge dx^k\\
    &=h^l_idx^i\wedge(\partial_l\iprod\varphi)\\
    &=h_{ij}g^{jl}dx^i\wedge(\partial_l\iprod\varphi),
    \end{split}\end{equation}
    and then \begin{equation}F(\mathcal{S}^2_0(M))=\{ h_{ij}g^{jl}dx^i\wedge(\partial_l\iprod\varphi)\;:\;h_{ij}=h_{ji},\;\tr_g(h)=0\}\end{equation}
    is the $\Omega^3_{14}$-factor, as wanted.
\end{proof}
\begin{preremark}\upshape
It follows that, in a $G_2$-structure $(M,\varphi)$, a $3$-form $\eta\in\Omega^3(M)$ is completely characterized by the data given by a vector field $X\in\mathfrak{X}(M)$ and a symmetric $2$-tensor $h$ (which encompasses all of $\mathcal{S}=\Omega^0\oplus\mathcal{S}_0$). It reads
\begin{equation}\begin{split}
    \eta&=h_{ij}g^{jl}dx^i\wedge(\partial_l\iprod\varphi)+X^l\partial_l\iprod\psi\\
    &=\frac{1}{2}h^l_i\varphi_{ljk}dx^i\wedge dx^j\wedge dx^k+\frac{1}{6}X^l\psi_{lijk}dx^i\wedge dx^j\wedge dx^k.
\end{split}\end{equation}
Furthermore, since $h_{ij}=\frac{1}{7}\tr_g(h)g_{ij}+h^0_{ij}$, where $h^0_{ij}$ corresponds to the traceless part of $h_{ij}$, it follows that
\begin{equation}\begin{split}
    F(h_{ij})&=\frac{1}{2}h^l_i\varphi_{ljk}dx^i\wedge dx^j\wedge dx^k\\
    &=\frac{3}{7}\tr_g(h)\varphi+\frac{1}{2}(h^0)^l_i\varphi_{ljk} dx^i\wedge dx^j \wedge dx^k,
\end{split}\end{equation}
which explicitly depicts the $\Omega^3_1$ and $\Omega^3_{27}$ components.
\end{preremark}

A $G_2$-structure $\varphi$ over $M$ determines a Riemannian metric $g$ and therefore one may consider the Levi-Civita connection $\nabla$ \footnote{The notation $\nabla^g$ is dropped for simplicity, since more general connections $\nabla$ are not considered in this section.}. One may then analyze the tensor field \begin{equation}\nabla\varphi\in\Gamma\lp T^*M\otimes\Lambda^3\lp T^*M\rp\rp.\end{equation}
In the Riemannian manifold case one was interested in the metric-compatibility property (which was seen to be equivalent to $\nabla g=0$). Then, one may define a similar notion for the $G_2$-structure case.

\begin{definition}
Let $(M,\varphi)$ be a $G_2$-structure and consider the tensor field $\nabla\varphi\in\Gamma\lp T^*M\otimes\Lambda^3\lp T^*M\rp\rp$. If
\begin{equation}
    \nabla\varphi=0,
\end{equation}
then $\varphi$ is called a \textbf{torsion-free} $G_2$-structure.
\end{definition}

\begin{theorem}
Let $X$ be a vector field over $M$. Then, $\nabla_X \varphi$ lies in the subspace $\Omega^3_7$ of the $G_2$ splitting of $\Omega^3(M)$. It follows that the covariant derivative $\nabla\varphi$ is a smooth section of $T^*M\otimes\Lambda^3_7\lp T^*M\rp$.
\end{theorem}
\begin{proof}
Since any $3$-form $\eta$ can be written as $\eta=F(A)$ for an unique $A=h+A_7$, where $h\in\mathcal{S}^2(M)$ and $A_7\in\Omega^2_7$, it follows that
\begin{equation}\label{torsionn}
\begin{split}
    \langle F(A),\nabla_X\varphi)&=\frac{1}{6}(F(A))_{ijk}(\nabla_X\varphi)_{abc}g^{ia}g^{jb}g^{kc}\\
    &=\frac{1}{6}(A^l_i\varphi_{ljk}+A^l_j\varphi_{ilk}+A^l_k\varphi_{ijl})X^m\nabla_m\varphi_{abc}g^{ia}g^{jb}g^{kc}\\
    &=\frac{1}{2}A^l_i\varphi_{ljk}X^m\nabla_m\varphi_{abc}g^{ia}g^{jb}g^{kc}\\
    &=\frac{1}{2}A^{la}X^m\varphi_{ljk}\nabla_m\varphi_{abc}g^{jb}g^{kc}.
    \end{split}
\end{equation}
Now, Theorem \ref{theoremcontrac} gives $\varphi_{ijk}\varphi_{abc}g^{jb}g^{kc}=6g_{ia}$. Taking the covariant derivative $\nabla_m$ and using the compatibility with $g$, it comes
\begin{equation}
    (\nabla_m\varphi_{ljk})\varphi_{abc}g^{jb}g^{kc}=-\varphi_{ljk}(\nabla_m\varphi_{abc})g^{jb}g^{kc}.
\end{equation}
Therefore, eqn (\ref{torsionn}) is anti-symmetric in the indices $l$ and $a$. Hence, the symmetric part of $A^{la}$ does not contribute to the expression, which gives the result.
\end{proof}
\begin{preremark}\upshape
The last result shows that 
\begin{equation}\nabla\varphi\in\Gamma\lp T^*M\otimes\Lambda^3_7\lp T^*M\rp\rp,\end{equation}
so that for each $X\in\mathfrak{X}(M)$ there holds 
\begin{equation}\nabla_X\varphi\in\Omega^3_7=\{Y\iprod\psi\;:\;Y\in\mathfrak{X}(M)\}.\end{equation}
With such characterization in mind, one may consider the following definition.
\end{preremark}
\begin{definition}
Let $(M,\varphi)$ be a $G_2$-structure. The \textbf{torsion tensor} of the $G_2$-structure is given by $T\in\Gamma(T^*M\otimes T^*M)$ such that
\begin{equation}\label{grig51}
    \nabla_X\varphi=2T(X)\iprod\psi,
\end{equation}
for each $X\in\mathfrak{X}(M)$.
\end{definition}
In index notation, one has
\begin{equation}
    \nabla_m\varphi_{ijk}=2T_{mp}g^{pq}\psi_{qijk},
\end{equation}
    and contracting with $\psi_{nabc}g^{ai}g^{bj}g^{ck}$ yields
    \begin{equation}
        T_{mn}=\frac{1}{48}\nabla_m\varphi_{ijk}\psi_{nabc}g^{ia}g^{jb}g^{kc}.
    \end{equation}
    It follows that the $G_2$-structure satisfies $\nabla\varphi=0$ if an only if $T=0$ and a classic result on torsion-free
$G_2$-structures is given by    
    \begin{corollary}\cite{fernandezgray,intro}
    The $G_2$-structure $\varphi$ over $M$ is torsion free if and only if $d\varphi=0$ and $d\psi=0$.
    \end{corollary}
     Moreover, since \begin{equation}T\in\Gamma(T^*M\otimes T^*M)\simeq \Omega^0(M)\oplus\mathcal{S}^2_0(M)\oplus\Omega^2_7\oplus\Omega^2_{14},\end{equation} one may decompose the torsion into four independent parts through
    \begin{equation}
        T=T_1+T_0+T_7+T_{14},
    \end{equation}
    where $T_1=\frac{1}{7}\tr_g(T)g$ and $T_0$ is traceless symmetric. 
Considering the vanishing or nonvanishing of each of its parts, a number of $2^4=16$ distinct torsion classes of $G_2$-structures emerge from this splitting. The torsion can be seen to connect with the curvature tensor of the underlying space and, in fact, some results are known depending on the class. For instance, if one considers the scalar case where all parts vanish but $T_1\neq 0$, then the induced metric $g$ can be shown to be positive Einstein with \begin{equation}R_{ij}=\frac{3}{8}\lambda^2 g_{ij}\end{equation}
and there holds $d\varphi=\lambda\psi$ \cite{spirocontas,intro}. More details on torsion classes of $G_2$-structures can be found in \cite{intro,spirocontas,introref8,spirotudo}.

\section{Octonion Bundle}

Given a $G_2$-structures one may present a generalization of the octonion algebra over a 7-dimensional manifold $M$, called the \textbf{octonion bundle} $\oct M$, as seen in \cite{grigorian,grigoriantorsion}. Fix, from now on, the $G_2$-structure $(M,\varphi)$ with associated metric $g$ and volume form $\text{vol}$.
\begin{definition}
The octonion bundle $\oct M$ over $M$ is the rank $8$ vector bundle
\begin{equation}\oct M=\Lambda^0(M)\oplus TM,\end{equation}
where $\Lambda^0(M)=M\times \re$ is the trivial line bundle and for each $p\in M$
\begin{equation}\oct_p M=\re\oplus T_p M.\end{equation}
\end{definition}
This bundle encompasses the real/imaginary decomposition of an octonion. A section $A\in\Gamma(\oct M)$ will be simply called an octonion. There are globally defined projections
\begin{equation}\begin{split}
    \text{Re}\;\;&:\;\;\Gamma(\oct M)\rightarrow \Omega^0(M),\\
    \text{Im}\;\;&:\;\;\Gamma(\oct M)\rightarrow \mathfrak{X}(M),
\end{split}\end{equation}
 and the octonion $A$ can generally be written as 
\begin{equation}A=\ree(A)+\text{Im}(A)=\lp\ree(A),\imm(A)\rp=\begin{pmatrix}
\ree(A)\\ \imm(A)\end{pmatrix}.\end{equation}
As before, the conjugation can also be defined by means of the equation
\begin{equation}
\bar{A}=\lp\ree(A),-\imm(A)\rp.\end{equation}
The metric $g$ over $M$ may induce a metric over $\oct M$, called the \textbf{octonion metric}. Namely, for $A=(a,\alpha)\in\Gamma(\oct M)$ such metric is taken as
\begin{equation}\begin{split}
    \Vert A\Vert^2&=\langle A,A\rangle=a^2+g(\alpha,\alpha)\\
    &=a^2+|\alpha|^2.
\end{split}\end{equation}

\begin{definition}
Given a $G_2$-structure $(M,\varphi)$ the \textbf{vector cross product} $\times_{\varphi}$ with respect to $\varphi$ can be defined by the expression
\bege
\langle \alpha\times_{\varphi}\beta,\gamma\rangle=\varphi(\alpha,\beta,\ga),
\enge
for every vector fields $\alpha,\beta,\ga\in\mathfrak{X}(M)$.
\end{definition}
This vector cross product obviously satisfy all properties obtained in the first section. For now on, whenever it is clear as to which 3-form $\varphi$ the definition of the cross product takes use, it shall be simply denoted by $\times$.

\begin{definition}
Let $A,B\in\Gamma(\oct M)$ be octonions with $A=(a,\alpha)$ and $B=(b,\beta)$. Then, the octonions product $A\circ_{\varphi} B$ with respect to $\varphi$ is defined by
\begin{equation}
    A\circ_{\varphi}B=\begin{pmatrix}
    ab-\langle\alpha,\beta\rangle\\ a\beta+b\alpha+\alpha\times_{\varphi}\beta
    \end{pmatrix}\in \Gamma(\oct M).
\end{equation}
\end{definition}
\begin{preremark}\upshape
Notice that this definition mimics eqn (\ref{intro360}). In fact, the $G_2$-structure globally provides with the information needed to define a cross vector $\times$ over the tangent bundle, which is the most important ingredient when defining a normed division algebra product. Whenever it is clear, the octonion product is simply denoted by juxtaposition $AB$ and it obviously has the expected properties from the division algebra $\oct$ developed in the last Chapter.
\end{preremark}
As before, the commutator and associator operations can also be considered: let $A,B$ and $C$ be octonions over $M$, with $A$ and $B$ as before and $C=(c,\gamma)$. Then,
\begin{align}
    \begin{split}
        [A,B]&=AB-BA\\
        &=2\alpha\times\beta\\
        &=2\varphi(\alpha,\beta,\cdot)^{\sharp},
    \end{split}
\end{align}
and
\begin{align}
    \begin{split}
        [A,B,C]&=(AB)C-A(BC)\\
        &=2\psi(\alpha,\beta,\ga,\cdot)^{\sharp}.
    \end{split}
\end{align}
This construction shows that given a $G_2$-structure over a 7-manifold, it is possible to fully transfer the octonion algebra structure to $\oct M$. Some useful identities in this configuration can be perceived as follows \cite{grigorian}.

\begin{lemma}
Let $A=(0,\alpha)$ be a pure imaginary octonion. Then, its exponential $e^A=\displaystyle\sum_{k=0}^{\infty}\frac{1}{k!}A^k$ is given by
\bege
e^A=\cos\lp\Vert\alpha\Vert\rp+\alpha\frac{\sin\lp\Vert\alpha\Vert\rp}{\Vert\alpha\Vert}.
\enge
\end{lemma}
\begin{proof}
It follows directly from the definition of octonion multiplication that
\begin{equation}\begin{split}
    A&=\alpha,\\
    A^2&=-\Vert\alpha\Vert^2,\\
    A^3&=-\Vert\alpha\Vert^2\alpha,\\
    A^4&=\Vert\alpha\Vert^4,
\end{split}\end{equation}
and so on. It follows that,
\begin{equation}\begin{split}
    e^A&=(1-\frac{1}{2}\Vert\alpha\Vert^2+\frac{1}{4!}\Vert\alpha\Vert^4-\ldots)+(\Vert\alpha\Vert-\frac{1}{3!}\Vert\alpha\Vert^3+\frac{1}{5!}\Vert\alpha\Vert^5-\ldots)\frac{\alpha}{\Vert\alpha\Vert}\\
    &=\cos(\Vert\alpha\Vert)+\alpha\frac{\sin(\Vert\alpha\Vert)}{\Vert\alpha\Vert}.
\end{split}\end{equation}
\end{proof}
\begin{corollary}
Let $B=(b,\beta)\in\Gamma(\oct M)$ be a nonzero octonion. Then, for every $k\in\mathbb{Z}$ there holds
\begin{equation}
    B^k=\Vert B\Vert^k\lp\cos\lp k\theta\rp+\hat{\beta}\frac{\sin\lp k\theta\rp}{\sin\lp\theta\rp}\rp,
\end{equation}
where $\hat{\beta}=\frac{\beta}{\Vert B\Vert}$ and $\theta\in\re$ with $\cos\lp \theta\rp =\frac{b}{\Vert B\Vert}$ and $\sin\lp\theta\rp=\Vert\hat{\beta}\Vert$.
\end{corollary}

\begin{lemma}\label{l39} For each octonion $A,B,C\in\Gamma(\oct M)$ and $k\in\mathbb{N}$ there holds
\begin{enumerate}
\item $[A,B,C]=-[\bar{A},B,C]$,
\item $[A^k,A,C]=0$,
\item $A[A,B,C]=[A,B,C]\bar{A}$,
\item $[A,A^kB,C]=\bar{A}^k[A,B,C]$
\item $[A,BA^k,C]=[A,B,C]\bar{A}^k$,
\item $[A^{k+1},B,C]=[A^k,B,C]\bar{A}+[A,B,C]A^k$.
\end{enumerate}
In particular, the last equation gives for $k=1$ and $k=2$ the relations
\begin{enumerate}
\item $[A^2,B,C]=[A,B,C](A+\bar{A})$,
\item $[A^3,B,C]=[A,B,C](\bar{A}^2+\Vert A\Vert^2+A^2)$.
\end{enumerate}
\end{lemma}
\begin{preremark}\upshape

Let $B\in\Gamma(\oct M)$ and consider the \textbf{right} and \textbf{left translations} \begin{equation}R_B,\;L_B:\Gamma(\oct M)\rightarrow\Gamma(\oct M)\end{equation} 
respectively given by
\begin{equation}\begin{split}
    R_B A&=AB\\
    L_B A&=BA.
\end{split}\end{equation}
When $B\neq 0$ these maps are invertible with $(R_B)^{-1}=R_{B^{-1}}$ and similarly for $L_B$. Besides, as already seen, they satisfy
\end{preremark}
\begin{lemma}\label{grig311}
Let $A,B,C\in\Gamma(\oct M)$. There holds
\begin{align}
    \langle R_BA,C\rangle&=\langle A,R_{\overline{B}}C\rangle\\
    \langle L_B A,C\rangle&=\langle A,L_{\overline{B}}C\rangle.
\end{align}
\end{lemma}

\section{Isometric $G_2$-Structures}
One would like to known if, given a $G_2$-structure $(M,\varphi)$ with associated metric $g_{\varphi}$, there is another $G_2$-structure $(M,\tilde{\varphi})$ with associated metric $g_{\tilde{\varphi}}$ such that \begin{equation}g_{\varphi}=g_{\tilde{\varphi}}\end{equation}
all over $M$. It turns out that the answer is positive so that for a fixed $G_2$-structure $\varphi$ there is a family parameterized by $S^7/\mathbb{Z}_2\simeq \re\mathbb{P}^7$ of other $G_2$-structures with the same associated metric $g_{\varphi}$ \cite{bryant2}. Indeed, notice that in general this can be analyzed by looking into the quotient $SO(7)/G_2$, which is $7$-dimensional, and can be shown to be diffeomorphic to the projective space $\re\mathbb{P}^7$. This notion is investigated in this section and its relation with the octonion bundle $\oct M$ is considered.
\begin{definition}
Let $V\in\Gamma(\oct M)$ be a non-vanishing octonion. Then, the \textbf{adjoint map} $\emph{\ad}_V:\Gamma(\oct M)\rightarrow \Gamma(\oct M)$ is defined by the expression
\bege
\emph{\ad}_V(A)=VAV^{-1},
\enge
for each $A\in\Gamma(\oct M)$.
\end{definition}
\begin{preremark}\upshape
Notice that the adjoint map is invertible, since $\ad_{V^{-1}}=(\ad_V)^{-1}$. Also, there holds \begin{equation}\ad_{\lambda V}=\ad_V,\end{equation}
so that one may assume that $V$ is unitary, without loss of generality. Besides, it preserves the octonion metric which can be seen by the straightforward computation
\begin{equation}\begin{split}
\langle \ad_V(A),\ad_V(B)\rangle&=\langle VAV^{-1},VBV^{-1}\rangle\\
&=\frac{1}{\Vert V\Vert^2}\langle VA\overline{V},VBV^{-1}\rangle\\
&=\langle \overline{V}\lp\overline{V}^{-1}\rp A,BV^{-1}V\rangle\\
&=\langle A,B\rangle.
\end{split}\end{equation}
Therefore, $\ad_V\in \text{O}(8)$. Moreover, $\ad_V$ preserves the real part of $\oct$, so that it maps imaginary octonions to imaginary octonions. Therefore, it restricts to pure imaginary octonions, with restriction denoted by \begin{equation}\ad_V|_{\imm\oct}\in \text{O}(7).\end{equation}
\end{preremark}
For simplicity, denote $\ad_V|_{\imm\oct}=\ad_V$. Let also $\beta$ be a pure imaginary octonion and $V=(v_0,v)$. There holds
\begin{align}
    \begin{split}
        \ad_V\lp \beta\rp &=V\beta V^{-1}\\
        &=\frac{1}{\Vert V \ \Vert^2}\lp v_0+v\rp \beta\lp v_0-v\rp \\
        &=\frac{1}{\Vert V \ \Vert^2}\lp v_0+v\rp \lp \langle v,\beta\rangle+v_0\beta+v\times\beta\rp \\
        &=\frac{1}{\Vert V \ \Vert^2}\lp v_0^2\beta+2v_0v\times\beta+v\langle v,\beta\rangle +v\times\lp v\times\beta\rp \rp \\
        &=\frac{1}{\Vert V \ \Vert^2}\lp \lp v_0^2-\Vert v\Vert^2\rp \beta+2v_0v\times\beta+2v\langle v,\beta\rangle\rp .
    \end{split}
\end{align}
It follows that, in index notation:
\begin{equation}
\Big{(}\ad_V|_{\imm\oct}\Big{)}^a_b=\frac{1}{||V||^2}\Big{(}\Big{(}v_0^2-|v|^2\Big{)}\delta^a_b-2v_0(v\iprod\varphi)^a_b+2v^av_b\Big{)}.    
\end{equation}
Furthermore, it may be seen that $\det\Big{(}\ad_V|_{\imm\oct}\Big{)}=+1$, so that $\ad_V|_{\imm\oct}\in \text{SO}(7)$ \cite{grigorian}. Since $\ad_V|_{\ree\oct}=+1$, there follows $\ad_V\in\text{SO}(8)$. The adjoint map also satisfy the following identities:

\begin{lemma}\label{l42}
Let $V$ be a nowhere-vanishing octonion. Then, for every $A,B\in\Gamma(\oct M)$ there holds
\begin{enumerate}
    \item $(VA)(BV^{-1})=\emph{\ad}_V(AB)+[A,B,V^{-1}](V+\overline{V})$,
    \item $(AV^{-1})(VB)=AB+[A,B,V^{-1}]V$.
\end{enumerate}
\end{lemma}
\begin{proposition}\label{grig43}
Let $(M,\varphi)$ be a $G_2$-structure and suppose $V$ is a nowhere-vanishing octonion. Then, for every $A,B\in\Gamma(\oct M)$ there holds
\begin{equation}
    \lp\emph{\ad}_V\lp A\rp\rp\lp\emph{\ad}_V(B)\rp=\emph{\ad}_V(AB)+[A,B,V^{-1}]\lp V+\bar{V}+\frac{1}{\Vert V\Vert^2}V^3\rp.
\end{equation}
Moreover, there holds
\begin{align}
    \emph{\ad}_{V^{-1}}\lp\lp\emph{\ad}_V(A)\rp\lp\emph{\ad}_V(B)\rp\rp&=AB+[A,B,V^{-3}]V^3\\
    &=\lp AV^{-3}\rp\lp V^3B\rp.
\end{align}
\end{proposition}
\begin{proof}
By using the second identity from Lemma \ref{l42} and the ones found in Lemma \ref{l39}, it follows that
\begin{equation}
    \begin{split}
        \lp\ad_V(A)\rp\lp\ad_V(B)\rp&=\lp VAV^{-1}\rp\lp VBV^{-1}\rp\\
        &=\lp VA\rp\lp BV^{-1}\rp+[VA,BV^{-1},V^{-1}]V\\
        &=\lp VA\rp\lp BV^{-1}\rp+\bar{V}[A,B,V^{-1}]\frac{V^2}{\Vert V\Vert^2}\\
        &=\lp VA\rp\lp BV^{-1}\rp+[A,B,V^{-1}]\frac{V^3}{\Vert V\Vert^2}.
    \end{split}
\end{equation}
Now, the first equation from Lemma \ref{l42} can be used to derive
\begin{equation}
    \begin{split}
        \lp\ad_V(A)\rp\lp\ad_V(B)\rp&=\ad_V(AB)+[A,B,V^{-1}]\lp\bar{V}+V\rp+[A,B,V^{-1}]\frac{V^3}{\Vert V\Vert^2}\\
        &=\ad_V(AB)+[A,B,V^{-1}]\lp\bar{V}+V+\frac{V^3}{\Vert V\Vert^2}\rp,
    \end{split}
\end{equation}
which proves the first identity. Now, noting that the subalgebra generated by the two elements $V$ and $[A,B,V^{-1}]$ is associative and applying $\ad_{V^{-1}}$ to the last equation yields
\begin{equation}
    \begin{split}
        \ad_{V^{-1}}\lp\lp\ad_V(A)\rp\lp\ad_V(B)\rp\rp&=AB+V^{-1}\lp[A,B,V^{-1}]\lp\bar{V}+V\frac{V^3}{\Vert V\Vert^2}\rp\rp V\\
        &=AB+\lp V^{-1}[A,B,V^{-1}]\rp\lp\lp\bar{V}+V+\frac{V^3}{\Vert V\Vert^2}\rp V\rp.
    \end{split}
\end{equation}
Using the identities from Lemma \ref{l39}, there holds
\begin{equation}
    \begin{split}
        \ad_{V^{-1}}\lp\lp\ad_V(A)\rp\lp\ad_V(B)\rp\rp&=AB-\lp[A,B,V]V\rp\lp\lp\bar{V}+V\frac{V^3}{\Vert V\Vert^2}\rp\frac{V}{\Vert V\Vert^4}\rp\\
        &=AB-[A,B,V]\lp\lp\bar{V}+V\frac{V^2}{\Vert V\Vert^4}\rp\rp\\
        &=AB-[A,B,V]\lp\bar{V}^2+\Vert V\Vert^2+V^2\rp\frac{V^3}{\Vert V\Vert^6}.
    \end{split}
\end{equation}
From the last equation in Lemma \ref{l39} and the second one in Lemma \ref{l42}, one can finally see that
\begin{equation}
    \begin{split}
        \ad_{V^{-1}}\lp\lp\ad_V(A)\rp\lp\ad_V(B)\rp\rp&=AB-\Vert V\Vert^{-6}[A,B,V^{3}]V^{3}\\
        &=AB+\Vert V\Vert^{-6}[A,B,\bar{V}^3]V^3\\
        &=AB+[A,B,V^{-3}]V^3\\
        &=\lp AV^{-3}\rp\lp V^{3}B\rp,
    \end{split}
\end{equation}
as wanted.
\end{proof}
Now, given a non-vanishing octonion $V$ one can then define a new octonion product $\circ_{V^3}$ given by
\begin{equation}\label{octprrro}
    A\circ_{V^3} B=\ad_V((\ad_V(A))(\ad_V(B)))=(AV^3)(V^{-3}B).
\end{equation}
It would be then natural to ask what kind of three form $\varphi_{V^3}$ would define such product. Notice that for every pure imaginary octonions $A,B$ and $C$ there holds
\begin{equation}
\begin{split}
    \varphi\lp\ad_{V^{-1}}(A),\ad_{V^{-1}}(B),\ad_{V^{-1}}(C)\rp&=\langle \ad_{V^{-1}}(A)\times \ad_{V^{-1}}(B),\ad_{V^{-1}}(C)\rangle\\
    &=\langle (\ad_{V^{-1}}(A)) (\ad_{V^{-1}}(B)),\ad_{V^{-1}}(C)\rangle.
    \end{split}
\end{equation}
since the adjoint restricts to the imaginary part. Obviously $\ad_V$ is self-adjoint so that
\begin{equation}
\begin{split}
\varphi\lp\ad_{V^{-1}}(A),\ad_{V^{-1}}(B),\ad_{V^{-1}}(C)\rp&=    \langle \ad_{V}(\ad_{V^{-1}}(A)\ad_{V^{-1}}(B)),C\rangle\\
&=\langle A\circ_{V^3} B,C\rangle\\
&=\varphi_{V^3}(A,B,C).
\end{split}
\end{equation}
Therefore, 
\bege\label{sigma}\varphi_{V^3}(A,B,C)=\varphi(\ad_{V^{-1}}(A),\ad_{V^{-1}}(B),\ad_{V^{-1}}(C)).\enge
\begin{preremark}\upshape 
Notice that since $\ad_V$ is invertible, then eqn (\ref{sigma}) shows that $\varphi_{V^3}$ is in the $\text{GL}(7,\re)$ orbit of the original $G_2$-structure $\varphi$. Moreover, since $\ad_V$ preserves the metric associated with $\varphi$, it follows that $\varphi_{V^3}$ has the same metric associated as $\varphi$. In order to better understand these relations and the octonion product defined by means of eqn (\ref{octprrro}), one can define the following map:
\end{preremark}
\begin{definition}
Let $(M,\varphi)$ be a $G_2$-structure. Then, for each non-vanishing octonion $V=(v_0,v)$ define the map of 3-forms $\sigma_V:\Omega^3(M)\rightarrow\Omega^3(M)$ given by
\begin{equation}
\sigma_V(\varphi)=\frac{1}{\Vert V \Vert^2}\lp\lp v_0^2-\Vert v\Vert^2\rp\varphi-2v_0v\iprod\psi+2v\wedge\lp v\iprod\varphi\rp\rp.    
\end{equation}
\end{definition}
\begin{theorem}\label{grig45}
Let $(M,\varphi)$ be a $G_2$-structure. Then, for any nowhere-vanishing octonion $V$ there holds
\begin{equation}
    \sigma_{V^3}(\varphi)(\cdot,\cdot,\cdot)=\varphi(\ad_{V^{-1}}(\cdot),\ad_{V^{-1}}(\cdot),\ad_{V^{-1}}(\cdot)).
\end{equation}
\end{theorem}
\begin{proof}
For pure imaginary octonions $A,B$ and $C$ there holds
\begin{equation}
\begin{split}
\varphi_{V^3}\lp A,B,C\rp&=\varphi\lp\ad_{V^{-1}}(A),\ad_{V^{-1}}(B),\ad_{V^{-1}}(C)\rp\\
&=\langle \ad_{V^{-1}}(A)\ad_{V^{-1}}(B),\ad_{V^{-1}}(C)\rangle\\
&=\langle \ad_V(\ad_{V^{-1}}(A)\ad_{V^{-1}}(B)),C\rangle\\
&=\langle AB+[A,B,V^3]V^{-3},C\rangle\\
&=\varphi(A,B,C)+\langle[A,B,V^3]V^{-3},C\rangle.
\end{split}
\end{equation}
Now, let $V^3=(w_0,w)$ and so $\Vert V^3\Vert^2=w_0^2+|w|^2=W$. Then, there holds 
\bege
\begin{split}
[A,B,V^3]V^{-3}&=[A,B,w]\frac{(w_0,-w)}{ W}\\
&=\frac{w_0}{W}[A,B,w]-\frac{1}{W}[A,B,w]\times w.
\end{split}
\enge
In order to expand in index notation one may write $w=w^de_d$ and use $\psi$ in order to express the associator. It follows that 
\begin{equation}
\begin{split}
    \langle[e_a,e_b,w],e_c\rangle&=\langle 2g^{ij}\psi_{jabd}w^de_i,e_c\rangle\\
    &=2g^{ij}\psi_{jabd}w^d\langle e_i,e_c\rangle\\
    &=2g^{ij}g_{ic}\psi_{jabd}w^d\\
    &=2\psi_{cabd}w^d.
    \end{split}
\end{equation}
Furthermore, 
\bege
\begin{split}
    \langle [e_a,e_b,w]\times w,e_c\rangle&=\langle 2\psi^m_{abd}w^dw^n(e_m\times e_n),e_c\rangle\\
    &=2\psi^m_{abd}w^dw^n\langle e_m\times e_n,e_c\rangle\\
    &=2\psi^m_{abd}\varphi_{mnc}w^dw^n.
\end{split}
\enge
Considering the previous relations, it follows that
\bege\label{grig416}
\lp\varphi_{V^3}\rp_{abc}=\varphi_{abc}+\frac{2w_0}{W}\psi_{cabd}w^d-\frac{2}{W}\varphi_{cmn}\psi^m_{\;\;abd}w^dw^n.
\enge
Now, (\ref{intro426}) can be used in the form
\begin{equation}
    \varphi_{abc}\psi_{mnp}^{\;\;\;\;\;\;\;c}=-3\lp g_{a[m}\varphi_{np]b}-g_{b[m}\varphi_{np]a}\rp,
\end{equation}
yielding
\begin{equation}
\varphi_{cmn}\psi^m_{\;abd}w^dw^n=\Vert u\Vert^2\varphi_{abc}-3w_{[a}\varphi_{bc]m}w^m,\end{equation}
which turns eqn (\ref{grig416}) into
\begin{equation}
\lp\varphi_{V^3}\rp_{abc}=\lp 1-\frac{2}{W}\vert w\vert^2\rp\varphi_{abc}+\frac{2w_0}{W}\psi_{cabd}w^d+\frac{6}{W}w_{[a}\varphi_{bc]m}w^m,\end{equation}
which in coordinate-free notation is given by
\begin{equation}
\varphi_{V^3}=\frac{1}{W}\lp\lp w_0^2-\Vert w\Vert^2\rp\varphi-2w_ow\iprod\psi+2w\wedge\lp w\wedge\varphi\rp\rp,\end{equation}
as claimed.
\end{proof}
\begin{preremark}\upshape
The last result shows that $\circ_{V^3}=\circ$ if and only if $V^3$ is real. Since one may assume $V$ to be unitary, then the octonion product is preserved by $\ad_V$ if an only if $V^6=1$. Furthermore, from Proposition \ref{grig43}, the octonion product defined by the $G_2$-structure $\sigma_V(\varphi)$ for a nonvanishing octonion $V$ is given, for $A,B\in\Gamma(\oct M)$, by
\begin{equation}\label{grig420}
    A\circ_V B=A\circ_{\sigma_V(\varphi)} B=AB+[A,B,V]V^{-1}=(AV)(V^{-1}B).
\end{equation}
\end{preremark}
\begin{lemma}
Let $U$ and $V$ be nonvanishing octonions. Then,
\begin{equation}
    U\circ_{\varphi} V=U\circ_V V 
\end{equation}
\end{lemma}
\begin{proof}
Let $V=(v_0,v)$. Then, the result comes directly from the octonion product definition and from the calculation
\begin{equation}\begin{split}
    v\iprod \sigma_V(\varphi)&=\frac{1}{V}v\iprod\lp\lp v_0^2-\Vert v\Vert^2\rp\varphi-2v_0v\iprod\psi+2v\wedge\lp v\iprod\varphi\rp\rp\\
    &=\frac{1}{V}\lp\lp v_0^2-\Vert v\Vert^2\rp v\iprod\varphi+2\Vert v\Vert^2 v\iprod\varphi\rp\\
    &=v\iprod \varphi.
\end{split}\end{equation}
Therefore, multiplying by $V$ using the product $\circ_V$ defined by $\sigma_V(\varphi)$ is the same as using the product $\circ$ defined by $\varphi$. One may then write the expression $UV$ without specifying which octonion product is being taken. 
\end{proof}
\begin{preremark}\upshape
Since $\sigma_V(\varphi)$ defines a new product $\circ_{V}$ then given $A,B,C\in\Gamma(\oct M)$ one may denote their associator with respect to $\circ_V$ by
\begin{equation}
[A,B,C]_V=(A\circ_V B)\circ_V C-A\circ_V(B\circ_V C).\end{equation}
\end{preremark}
\begin{theorem}\label{grig48}
Let $(M,\varphi)$ be a $G_2$-structure. Then, given nowhere-vanishing octonions $U$ and $V$ there holds
\begin{equation}\label{grig419}
    \sigma_U(\sigma_V(\varphi))=\sigma_{UV}(\varphi).
\end{equation}
\end{theorem}

\section{Octonion Covariant Derivative}

Given a $G_2$-structure $(M,\varphi)$ one may analyze the relation between the octonion product over $\oct M$ and the (Levi-Civita) connection given by the associated metric $g$, being the torsion naturally introduced. As usual, the connection $\nabla$ satisfies the Leibniz rule between the product of two vector fields, but extending it to the octonion bundle one may investigate how $\nabla$ behaves with respect to it.

Let $(M,\varphi)$ be a $G_2$-structure and fix from now on the Levi-Civita connection $\nabla$ of the metric $g$ associated with $\varphi$. If $A=(a,\alpha)\in\Gamma(\oct M)$ then one may define the extension
\begin{equation}\label{grig61}
    \nabla_X A=\lp\nabla_Xa,\nabla_X\alpha\rp
\end{equation}
for each $X\in\mathfrak{X}(M)$.
\begin{proposition}
Let $A,B\in\Gamma(\oct M)$. Then, for every $X\in\mathfrak{X}(M)$ there holds
\begin{equation}\label{grig62}
    \nabla_X\lp AB\rp=\lp\nabla_X A\rp B+A\lp\nabla_X B\rp-[T(X),A,B],
\end{equation}
where $T(X)=(0,X\iprod T)$ and $T$ is the torsion of the $G_2$-structure $\varphi$.
\end{proposition}
\begin{proof}
One may write $A=(a,\alpha)$ and $B=(b,\beta)$. Then, using the octonion product definition there holds
\begin{equation}\label{aquiagora}
    \nabla_X\lp AB\rp=\nabla_X\begin{pmatrix}
    ab-\langle \alpha,\beta\rangle\\ a\beta+b\alpha+\varphi(\alpha,\beta,\cdot)^{\sharp}\end{pmatrix}=\begin{pmatrix}
    \lp\nabla_X a\rp b+a\lp\nabla_X b\rp-\nabla_X\lp\langle\alpha,\beta\rangle\rp\\
    \nabla_X\lp a\beta+b\alpha\rp+\nabla_X\lp\varphi(\alpha,\beta,\cdot)^{\sharp}\rp.
    \end{pmatrix},
\end{equation}
whereas
\begin{equation}
    \lp \nabla_X A\rp B=\begin{pmatrix}
    \lp\nabla_X a\rp b-\langle \nabla_X\alpha,\beta\rangle\\
    \lp\nabla_X a\rp\beta+b\lp\nabla_X\alpha\rp+\lp\nabla_X\alpha\rp\times\beta
    \end{pmatrix},
\end{equation}
and similarly for $A\lp\nabla_X B\rp$. Then, notice that since $\nabla$ is metric-compatible and satisfies the Leibniz rule, it follows
\begin{equation}
    \nabla_X\lp AB\rp-\lp\nabla_X A\rp B-A\lp\nabla_X B\rp=\begin{pmatrix}
    0\\
    \nabla_X\lp\varphi(\alpha,\beta,\cdot)^{\sharp}\rp-(\nabla_X\alpha)\times\beta -\alpha\times(\nabla_X\beta)
    \end{pmatrix}.
\end{equation}
However,
\begin{equation}\begin{split}
    \lp\nabla_X \varphi\rp(\alpha,\beta,\cdot)^{\sharp}&=\nabla_X \lp\varphi(\alpha,\beta,\cdot)^{\sharp}\rp-\varphi(\nabla_X \alpha,\beta,\cdot)^{\sharp}-\varphi(\alpha,\nabla_X \beta,\cdot)^{\sharp}\\
    &=\nabla_X \lp\varphi(\alpha,\beta,\cdot)^{\sharp}\rp-\lp\nabla_X\alpha\rp\times\beta-\alpha\times\lp\nabla_X\beta\rp,
    \end{split}\end{equation}
    and hence
\begin{equation}
    \nabla_X(AB)=\lp\nabla_XA\rp B+A\lp\nabla_X B\rp-\begin{pmatrix}
    0\\\lp\nabla_X(\varphi)\rp(\alpha,\beta,\cdot)^{\sharp}
    \end{pmatrix}.
\end{equation}
From eqn (\ref{grig51}) one has $\nabla_X(\varphi)=2 T(X)\iprod\psi$, and therefore
\begin{equation}\begin{split}
    \lp\nabla_X(\varphi)\rp(\alpha,\beta,\cdot)^{\sharp}&=2\psi\lp T(X),\alpha,\beta,\cdot\rp^{\sharp}\\
    &=[T(X),\alpha,\beta],
\end{split}\end{equation}
which gives the result.

\end{proof}

\begin{preremark}\upshape 
It is straightforward to see that if either $A$ or $B$ is real then the associator vanishes and one may recover the standard Leibniz rule for $\nabla$. Also, notice that 
\begin{equation}[T(X),A,B]=0\end{equation}
identically for all $X\in\mathfrak{X}(M)$ and $A,B\in\Gamma(\oct M)$ if and only if $T=0$, that is, the Levi-Civita connection is compatible with octonion multiplication if and only if the $G_2$-structure is torsion-free.
\end{preremark}
It is possible to adapt the covariant derivative in order to make it compatible with octonion multiplication \cite{grigorian}. Note that the torsion tensor $T$ may be considered as a pure octonion-valued $1$-form over $M$, that is
\begin{equation}
T\in\Gamma(T^* M\otimes\imm(\oct M))=\Omega^1(\imm(\oct M)),\end{equation}
with
\begin{equation}
\begin{pmatrix} 0 \\ X\iprod T\end{pmatrix}=T(X)\in\Gamma(\imm(\oct M)).\end{equation}
\begin{definition}
Define for each vector field $X\in\mathfrak{X}(M)$ the \textbf{octonion covariant derivate}  
\begin{equation}
D_X:\Gamma(\oct M)\rightarrow\Gamma(\oct M)\end{equation}
by the relation
\begin{equation}\label{grig64}
    D_X A=\nabla_XA-AT(X),
\end{equation}
for each $A\in\Gamma(\oct M)$.
\end{definition}
\begin{preremark}\upshape By straightforward computation there holds
\begin{equation}\label{grig65}
    D_X 1=-T(X),
\end{equation}
for every $X\in\mathfrak{X}(M)$. One may see that this derivation satisfies a quasi-derivation property with respect to the octonion product, as follows.
\end{preremark}
\begin{proposition}
Let $A,B\in\Gamma(\oct M)$ and $X\in\mathfrak{X}(M)$. It follows that
\begin{equation}\label{grig66}
    D_X\lp AB\rp=\lp\nabla_XA\rp B+A\lp D_XB\rp
\end{equation}
\end{proposition}
\begin{proof}
Using Definition (\ref{grig64}) directly, it comes
\begin{equation}
    D_X(AB)=\nabla_X(AB)-(AB)T(X),
\end{equation}
and then using Proposition (\ref{grig61}) and associator properties there follows
\begin{equation}\begin{split}
    D_X(AB)&=\lp \nabla_XA\rp B+A\lp\nabla_XB\rp-[T(X),A,B]-(AB)T(X)\\
    &=\lp\nabla_XA\rp B+A\lp\nabla_XB\rp-[A,B,T(X)]-(AB)T(X)\\
    &=\lp\nabla_XA\rp B+A\lp\nabla_XB\rp-A(BT(X))+(AB)T(X)-(AB)T(X)\\
    &=\lp\nabla_XA\rp B+A\lp D_XB\rp.
\end{split}\end{equation}
\end{proof}
One may also show that $D$ has a kind of metric-compatibility with respect to the $\oct M$ extended metric, namely
\begin{equation}
g(A,B)=\ree(A)\ree(B)+g(\imm(A),\imm(B)),
\end{equation}
where $g$ in the right-hand side denotes the original associated metric over $M$.
\begin{proposition}
Let $A,B\in\Gamma(\oct M)$ and $X\in\mathfrak{X}(M)$. There holds
\begin{equation}\label{grig68}
    \nabla_X(g(A,B))=g(D_X A,B)+g(A,D_X B).
\end{equation}
\end{proposition}
\begin{proof}
One may see that
\begin{equation}\begin{split}
    g(D_XA,B)&=g(\nabla_X A-AT(X),B)\\
    &=g(\nabla_XA,B)-g(AT(X),B)\\
    &=g(\nabla_XA,B)-g(T(X),\overline{A}B),
\end{split}\end{equation}
where Lemma \ref{grig311} was used. Similarly
\begin{equation}\begin{split}
    g(A,D_XB)&=g(A,\nabla_XB)-g(A,BT(X))\\
    &=g(A,\nabla_XB)-g(T(X),\overline{B}A).
\end{split}\end{equation}
Combining the previous relations yields
\begin{equation}
    g(D_XA,B)+g(A,D_XB)=g(\nabla_X A,B)+g(A,\nabla_XB)-g(T(X),\overline{A}B-\overline{B}A).
\end{equation}
Now, $T(X)$ is pure imaginary whereas $\overline{A}B+\overline{B}A$ is real, so that their inner product is zero. Therefore,
\begin{equation}
    g(D_XA,B)+g(A,D_XB)=g(\nabla_X A,B)+g(A,\nabla_XB)=\nabla_X(g(A,B)).
\end{equation}
\end{proof}
One may now consider a change of reference given by a nonvanishing octonion $V$ by means of $\sigma_V(\varphi)$ in eqn (\ref{grig420}) yielding a new octonion product $\circ_V$, as previously analyzed.
\begin{lemma}\label{grig71}
Let $V$ be a nonvanishing octonion. Then, for every $A,B\in\Gamma(\oct M)$ and any $X\in\mathfrak{X}(M)$ there holds
\begin{equation}\label{grig72}
    \nabla_X\lp A\circ_V B\rp=\lp\nabla_X A\rp\circ_V B+A\circ_V\lp\nabla_X B\rp-[\emph{\ad}_V (T(X))+V(\nabla_X V^{-1}),A,B]_V.
\end{equation}
\end{lemma}
\begin{preremark}\upshape
Denote by $T_V$ the torsion of the $G_2$-structure $\sigma_V(\varphi)$. Then, by eqn (\ref{grig62}) there follows
\begin{equation}\label{grig73}
    \nabla_X(A\circ_V B)=(\nabla_X A)\circ_V B+A\circ_V(\nabla_X B)-[T_V(X),A,B]_V.
\end{equation}
Comparing the last equation with (\ref{grig72}) gives the following result.
\end{preremark}
\begin{theorem}
Let $(M,\varphi)$ be a $G_2$-structure with torsion $T\in\Omega^1(\emph{\imm}( \oct M))$. Then, the torsion $T_V$ of $\sigma_V(\varphi)$ for some non-vanishing octonion $V\in\Gamma(\oct M)$ is given by
\begin{equation}\label{grig74}
    T_V=\emph{\imm}(\emph{\ad}_V T+V(\nabla V^{-1})).
\end{equation}
Furthermore, if $V$ has constant norm then
\begin{equation}
    T_V=-(DV)V^{-1}.
\end{equation}
\end{theorem}
\begin{proof}
This comes directly from eqn (\ref{grig72}). Since it is defined for every $A,B\in\Gamma(\oct M)$, comparing with (\ref{grig73}) yields that the imaginary parts of $T_V$ and $\ad_V T+V(\nabla V^{-1})$ must be the same. However, since $T_V$ is pure imaginary, the result follows.

However, in general $\ree(\ad_V T+V(\nabla V^{-1}))\neq 0$. Notice first that since $V V^{-1}=1$ and since $[A,V,V^{-1}]=0$ for every $A\in\Gamma(\oct M)$ then one has
\begin{equation}
\begin{split}
    \nabla(VV^{-1})&=0\\
    V(\nabla V^{-1})+(\nabla V)V^{-1}&=0\\
    V(\nabla V^{-1})=-(\nabla V)V^{-1}.
    \end{split}
\end{equation}
It follows that
\begin{equation}\begin{split}
    \ree(\ad_V T+V(\nabla V^{-1})&=\langle \ad_V T+V(\nabla V^{-1}),1\rangle\\
    &=\langle V(\nabla V^{-1}),1\rangle\\
    &=-\langle(\nabla V)V^{-1},1\rangle\\
    &=-\frac{1}{\Vert V\Vert^2}\langle \nabla V, V\rangle\\
    &=-\frac{1}{2}\frac{1}{\Vert V\Vert^2}\nabla \Vert V\Vert^2\\
    &=-\nabla \ln \Vert V\Vert.
\end{split}\end{equation}
In particular, if $\Vert V \Vert$ is constant then the real part vanishes and therefore
\begin{equation}\begin{split}
    T_V&=\ad_V T+V(\nabla V^{-1})\\
    &=VTV^{-1}-(\nabla V)V^{-1}\\
    &=-(\nabla V-VT)V^{-1}\\
    &=-(DV)V^{-1}.
\end{split}\end{equation}
\end{proof}

\section{Spinor Bundle}
To conclude this chapter, it is possible to relate this description of $G_2$-structures with one emerging from the so-called spinor bundle over the $7$-dimensional manifold $M$, as one may see in \cite{grigorian}. The general construction of the spinor bundle using Clifford algebras are briefly introduced and an equivalence between the two descriptions on the level of affinely connected spaces is presented. 

 Let $V$ be a finite $n$-dimensional real vector space and consider the space of alternating $k$-multilinear transformations $\Lambda^k(V)$. Such space gives rise to the exterior algebra $\Lambda(V)$ by means of the well-known wedge product. If $\psi,\phi\in \Lambda^k(V)$ then one can define the \textbf{reversion} given by
\begin{equation}
    \widetilde{\psi\phi}=\tilde{\phi}\tilde{\psi},
\end{equation}
which is an algebra anti-automorphism. There also holds
\begin{equation}
    \tilde{\psi}=(-1)^{k(k-1)/2}\psi.
\end{equation}
On the other hand, the \textbf{graded involution} is an automorphism given by
\begin{equation}
    \hat{\psi}=(-1)^k\psi,
\end{equation}
and the \textbf{conjugation}, which is the composition of the reversion and the graded involution is denoted by
\begin{equation}
    \overline{\psi}=(-1)^{k(k+1)/2}\psi.
\end{equation}
One can also define the projection $\langle\;\cdot\;\rangle_i$ on the $i$-vector part . If $\psi=\psi_1+\cdots+\psi_n$, where $\psi_j\in\Lambda^j(V)$ for each $j$, then
\begin{equation}
    \langle\psi\rangle_i=\psi_i.
\end{equation}
It is possible to further define the projection on the $i$ and $j$ part of $\psi$, given by
\begin{equation}
    \langle\psi\rangle_{i\oplus j}=\psi_i+\psi_j,
\end{equation}
and so on.

Now, let $(V,g)$ be a quadratic space ($g$ is a non-degenerate symmetric bilinear form over $V$). The Clifford algebra associated to $(V,g)$ is denoted by $\cl(V,g)$ and can be perceived as a deformation of the exterior algebra $\Lambda(V)$. It is a $\mathbb{Z}_2$-graded associative algebra with unity $1\in\cl(V,g)$ and it is isomorphic to the exterior algebra as a vector space and therefore it inherits the previously mentioned (anti)automorphisms and projections. Denote by $\cl^+(V,g)$ its even subalgebra. One may also write
\begin{equation}
    \cl(V,g)=\bigoplus^n_{i=0}\Lambda^i(V),
\end{equation}
so that the multivector structure of the exterior algebra can also be considered. The Clifford algebra is endowed with a product $"\cdot"$ which is defined by the so-called Clifford identity, which for each $u,v\in V$ reads
\begin{equation}
    u\cdot v+v\cdot u=2g(u,v).
\end{equation}
In fact, if $I$ is the ideal in the tensor algebra $T(V)$ generated by all elements
\begin{equation}
v\otimes v-g(v,v),
\end{equation}
where $v\in V$, then the Clifford algebra is given by the quotient
\begin{equation}
\cl(V,g)= T(V)/I,
\end{equation}
so that one can see that the Clifford product satisfies
\begin{equation}
    u\cdot v= u\wedge v+g(u,v).
\end{equation}
In the case $V=\re^n$, then Sylvester's Law of Inertia for the quadratic space $(\re^n,g)$ says that there is an orthogonal (with respect to $g$) basis $\{e_1,\ldots,e_p,e_q,\ldots,e_{p+q}\}$ for $V$, where $p+q=n$, and such that
\begin{equation}
    g(e_i,e_j)=\left\{
    \begin{array}{ll}
         \;1 & \;\text{if}\;1\leq i\leq p,  \\
        -1 & \;\text{if}\;p+1\leq i \leq q. \\
    \end{array}
    \right.
\end{equation}
Then one may denote $(\re^{n},g)=\re^{p,q}$ where $(p,q)$ is called the \textbf{signature} of such quadratic space. The Clifford algebra for $\re^{p,q}$ is then denoted by $\cl_{p,q}$.
\begin{example}\upshape
In the trivial case $V=\{0\}$ the tensor algebra is just $T(V)=\re$ and so $\cl_{0,0}=\re$. Consider now the quadratic space $V=\re^{0,1}$. Then, there is a vector $e_1\in\re$ such that
\begin{equation}
    g(e_1,e_1)=-1,
\end{equation}
where $g$ is the quadratic form associated with $\re^{0,1}$. A basis for the Clifford algebra is then given by $\{1,e_1\}$ and it is straightforward to see that
\begin{equation}
    \cl_{0,1}\simeq\com.
\end{equation}
Similarly, the quadratic space $V=\re^{0,1}$ has
\begin{equation}
    \cl_{1,0}\simeq\re\oplus\re.
\end{equation}
\end{example}
\begin{example}\upshape
Moving forward one may consider the space $V=\re^{0,2}$, for which there is an orthonormal basis $\{e_1,e_2\}$ such that 
\bege g(e_1,e_1)=g(e_2,e_2)=-1,\;\;\;\;\;\;g(e_1,e_2)=g(e_2,e_1)=0.\enge
An arbitrary element in $\cl_{0,2}$ is in the form
\bege \cl_{0,2} \ni  a+be_1+ce_2+de_1e_2,\enge
with $a,b,c,d\in\re$ and 
\bege e_1\cdot e_1=e_2\cdot e_2=-1,\;\;\;\;\;\;e_1\cdot e_2=e_2\cdot e_1=0.\enge
It follows that
\bege(e_1\cdot e_2)^2=-1.\enge
Therefore, the set $\{1,e_1,e_2,e_1\cdot e_2\}$ is a basis for $\cl_{0,2}$ and one can see that it is isomorphic to the quaternion algebra $\quat$. An explicit isomorphism $\rho:\cl_{0,2}\rightarrow\quat$ can be defined, for instance by
\bege \rho(1)=1,\;\;\rho(e_1)=i,\;\;\rho(e_2)=j,\;\;\rho(e_1\cdot e_2)=k,\enge
where $i,j$ and $k$ are the imaginary units in $\quat$, such that $i^2=j^2=k^2=-1$ and $ij=-ji=k,\;jk=-kj=i$ and $ki=-ik=j$. Therefore,
\bege \cl_{0,2}\simeq \quat.\enge
In a similar way, one can see that 
\begin{equation}\cl_{2,0}\simeq\cl_{1,1}\simeq\text{M}(2,\re).\end{equation}
In the general case $V=\re^{p,q}$ one need only to construct explicit isomorphisms up to $\dim V=8$, since there holds
\end{example}
\begin{theorem}[Atiyah-Bott-Shapiro Periodicity Theorem]For every quadratic space $\re^{p,q}$ it follows that $\cl_{p,q+8}\simeq\mathcal{M}(16,\re)\otimes\cl_{p,q}$.
\end{theorem}
In the light of the last Theorem, one may define the following table:
\begin{table}[H]\label{tabela}
\centering
\resizebox{\textwidth}{!}{\begin{tabular}{|c|c|c|c|c|}
\hline
$p-q\mod 8$ & 0                     & 1                                               & 2                     & 3                  \\ \hline
$\cl_{p,q}$ & $\mathcal{M}(2^{[n/2]},\re)$     & $\mathcal{M}(2^{[n/2]},\re)\oplus \mathcal{M}(2^{[n/2]},\re)$         & $\mathcal{M}(2^{[n/2]},\re)$     & $\mathcal{M}(2^{[n/2]},\com)$ \\ \hline
$p-q\mod 8$ & 4                     & 5                                               & 6                     & 7                  \\ \hline
$\cl_{p,q}$ & $\mathcal{M}(2^{[n/2]-1},\quat)$ & $\mathcal{M}(2^{[n/2]-1},\quat)\oplus \mathcal{M}(2^{[n/2]-1},\quat)$ & $\mathcal{M}(2^{[n/2]-1},\quat)$ & $\mathcal{M}(2^{[n/2]},\quat)$ \\ \hline
\end{tabular}}
\caption{Clifford Algebra Classification ($n=p+q$ and $[\;\cdot\;]$ is the floor function) \cite{roldao}.}
\end{table}

Now, a natural group sitting inside of $\cl_{p,q}$ is the subset of invertible elements, namely
\bege\cl_{p,q}^*=\{a\in\cl_{p,q}\;:\;\exists a^{-1}\in\cl_{p,q}\},\enge
and a prominent subgroup in Clifford theory is the so-called \textbf{twisted Clifford-Lipschitz group} given by
\bege 
\Gamma_{p,q}=\{a\in\cl^*_{p,q}\;:\;\hat{a}v a^{-1}\in\re^{p,q},\;\forall v\in\re^{p,q}\}.\enge
Defining the application 
\bege \sigma:\Gamma_{p,q}\rightarrow \text{Aut}(\cl_{p,q}),\enge
given by
\bege
\sigma(a)(v)=\hat{a}v a^{-1}
\enge
then one can prove the
\begin{theorem}\label{clif}
Let $\cl_{p,q}$ be the Clifford algebra for the quadratic space $\re^{p,q}$ and let $\sigma:\Gamma_{p,q}\rightarrow \text{Aut}(\cl{p,q})$ be as previously defined. Also, let $\Gamma_{p,q}^+=\Gamma_{p,q}\cap\cl_{p,q}^+$. Then, 
\bege\begin{split}
\sigma(\Gamma_{p,q})&\simeq\emph{\text{O}}(p,q),\\
\sigma(\Gamma^+_{p,q})&\simeq\emph{\text{SO}}(p,q).
\end{split}\enge
Moreover, there holds
\begin{equation}
    \ker \sigma=\re^*,
\end{equation}
where $\re^*=\re\backslash\{0\}$.
\end{theorem}
Now, one may consider a norm $N:\cl_{p,q}\rightarrow\re$ defined for each $a\in\cl_{p,q}$ by
\begin{equation}
N(a)=|\langle\widetilde{a} a\rangle_0|.
\end{equation}
This norm satisfies the relation
\bege
N(a\cdot b)=N(a)N(b),
\enge
and it may be used to define the $\text{Pin}(p,q)$ subgroup of the twisted Clifford-Lipschitz group, given by
\begin{equation}
    \text{Pin}(p,q)=\{a\in\Gamma_{p,q}\;:\;N(a)=1\}.
\end{equation}
Then, the $\text{Spin}(p,q)$ group is just
\begin{equation}
    \text{Spin}(p,q)=\text{Pin}(p,q)\cap\cl^+_{p,q}.
\end{equation}
It then follows from Theorem \ref{clif} that
\bege\begin{split}
\textnormal{Pin}(p,q)/\mathbb{Z}_2\simeq\textnormal{O}(p,q),\\
\textnormal{Spin}(p,q)/\mathbb{Z}_2\simeq\textnormal{SO}(p,q).
\end{split}\enge
Then, the restriction $\sigma:\text{Spin}(p,q)\rightarrow\text{SO}(p,q)$ can be seen as a $2$-fold covering of the space $\text{SO}(p,q)$. Elements of an irreducible representation of the group $\text{Spin}(p,q)$ are called (classical) \textbf{spinors}.

This construction can be transported to an oriented manifold $M$ as follows: one considers a $\text{Spin}(p,q)$-principal bundle $\pi_{s}:P_{\text{Spin}(p,q)}(M)\rightarrow M$ with a $2$-fold application \cite{roldao202040}
\begin{equation}
    s:P_{\text{Spin}(p,q)}(M)\rightarrow P_{\text{SO}(p,q)}(M),
\end{equation}
for which there holds
\bege
s(p\phi)=s(p)\sigma(\phi),
\enge
for every $p\in P_{\text{Spin}(p,q)}$ and $\phi\in \text{Spin}(p,q)$. Then, in order to define spinor fields over $M$ one must first set the notion of a \textbf{spinor bundle}. Namely, it is given by the vector bundle 
\bege
\mathcal{S}(M)=P_{\text{Spin}(p,q)}(M)\times_{\rho}S_{p,q},
\enge
where $\rho:\text{Spin}(p,q)\rightarrow \text{End}(S_{p,q})$ is a representation of the Spin group and $S_{p,q}$ is a left module for $\cl_{p,q}$. This yields the following description with respect to the Clifford algebra classification \cite{roldao2020,roldao}:
\begin{table}[H]
\centering

\resizebox{\textwidth}{!}{\begin{tabular}{|c|c|c|c|c|}
\hline
$p-q\mod 8$ & 0                                                      & 1                         & 2                      & 3                         \\ \hline
$S_{p,q}$ & $\re^{2^{[(n-1)/2]}}\oplus\re^{2^{[(n-1)/2]}}$         & $\re^{2^{[(n-1)/2]}}$     & $\com^{2^{[(n-1)/2]}}$ & $\quat^{2^{[(n-1)/2]-1}}$ \\ \hline
$p-q\mod 8$ & 4                                                      & 5                         & 6                      & 7                         \\ \hline
$S_{p,q}$ & $\quat^{2^{[(n-1)/2]-1}}\oplus\quat^{2^{[(n-1)/2]-1}}$ & $\quat^{2^{[(n-1)/2]-1}}$ & $\com^{2^{[(n-1)/2]}}$ & $\re^{2^{[(n-1)/2]}}$     \\ \hline
\end{tabular}}
\caption{Spinor classification ($n=p+q$ and $[\;\cdot\;]$ is the floor function) \cite{roldao}.}
\end{table}
Then, one says a \textbf{spinor field} is precisely a section in $\mathcal{S}$. More details on the definitions and properties of the Clifford algebras can be found in \cite{lou,roldao} whereas for spinor bundles one can see \cite{lawson}. Further applications of spinors and their emergence in mathematical-physics can be also seen in \cite{Fabbri:2017lvu, BBR,lou,Ablamowicz:2014rpa,Fabbri:2016msm, Mosna:2003am, Cra, 123, Bonora:2017oyb, brito, 1, Gursey:1983yq}.

Let now $(M,\varphi)$ be a $G_2$-structure and $A\in\Gamma(\oct M)$ be an octonion over $M$. Then, consider the algebra of left translations $L_A:\Gamma(\oct M)\rightarrow \Gamma(\oct M)$, with $L_A(V)=AV$ for every $V\in\Gamma(\oct M)$. Notice that for every $A,B,V\in\Gamma(\oct M)$ there holds
\begin{equation}
    \begin{split}
        L_AL_B(V)+L_BL_A(V)&=A(BV)+B(AV)\\
        &=AB(V)+[A,B,V]+(BA)V+[B,A,V]\\
        &=(AB+BA)V,
    \end{split}
\end{equation}
so that if $A$ and $B$ are pure imaginary then
\begin{equation}
    L_AL_B+L_BL_A=-\langle A,B\rangle \text{Id},
\end{equation}
which is precisely the defining identity of the Clifford algebra. Therefore, the octonion algebra gives rise to a Clifford algebra (which is associative, since composition is associative) by means of left translations. This construction is called the \textbf{enveloping algebra} of the octonion algebra $\oct$. Notice that in general $L_AL_B\neq L_{AB}$.

\begin{lemma}\label{grig81}
Let $(M,\varphi)$ be a $G_2$-structure over $M$ and $A,B,C\in\Gamma(\oct M)$. Then,
\begin{equation}
    A(BC)=(A\circ_C B)C,
\end{equation}
where $\circ_C$ is the octonion product defined by means of $\sigma_C(\varphi)$. In particular,
\begin{equation}
    L_AL_BC=L_{A\circ_C B}C.
\end{equation}
\end{lemma}
\begin{proof}
Direct computation yields
\begin{equation}\begin{split}
    A(BC)&=(AB)C+[A,B,C]\\
    &=(AB+[A,B,C]C^{-1})C\\
    &=(A\circ_C B)C,
\end{split}\end{equation}
where the relation $A\circ_C B=AB+[A,B,C]C^{-1}$ was used.
\end{proof}

Let now $\mathcal{S}(M)=\mathcal{S}$ be the spinor bundle over the $7$-manifold $M$ and denote by $\langle \cdot,\cdot\rangle_{\mathcal{S}}$ its inner product. Also, one may denote by $\langle\cdot,\cdot\rangle_{\oct}$ the octonion metric with respect to the $G_2$-structure $(M,\varphi)$. Then, a nowhere-vanishing unit spinor $\xi\in\Gamma(\spb)$ over $M$ also defines a $G_2$-structure via the expression \cite{agricola2}
\begin{equation}\label{grig82}
    \varphi_{\xi}(\alpha,\beta,\gamma)=-\langle\xi,\alpha\cdot(\beta\cdot(\gamma\cdot\xi))\rangle_{\spb},
\end{equation}
where $\alpha,\beta,\gamma\in\mathfrak{X}(M)$.
\begin{lemma}
Let $\alpha,\beta,\gamma\in\Gamma(\emph{\imm}\oct M)$ and $V\in\Gamma(\oct M)$ an unit octonion. Then, there holds
\begin{equation}\label{grig83}
    \lp\sigma_V\varphi\rp(\alpha,\beta,\gamma)=-\langle V,\alpha\lp\beta\lp\gamma V\rp\rp\rangle_{\oct}.
\end{equation}
\end{lemma}
\begin{proof}
From Lemma \ref{grig81} it comes
\begin{equation}\begin{split}
    \alpha\lp\beta\lp\gamma V\rp\rp&=\alpha\lp\lp\beta\circ_V \gamma\rp V\rp\\
    &=\lp\alpha\circ_{V}\lp\beta\circ_V \gamma\rp \rp  V.
\end{split}\end{equation}
Then, since $\Vert V\Vert=1$ it follows that
\begin{equation}\begin{split}
    \langle V,\alpha\lp\beta\lp\gamma V\rp \rp \rangle_{\oct}&=\langle V,\lp\alpha\circ_V\lp\beta\circ_V\gamma\rp \rp V\rangle_{\oct}\\
    &=\langle 1,\alpha\circ_V\lp\beta\circ_V\gamma\rp \rangle_{\oct}\\
    &=-\langle\alpha,\beta\circ_V \gamma\rangle_{\oct}.
\end{split}\end{equation}
Therefore, there holds
\begin{equation}\begin{split}
    \langle V,\alpha\lp\beta\lp\gamma V\rp \rp \rangle_{\oct}&=-\langle\alpha,\beta\circ_V\gamma\rangle_{\oct}\\
    &=-\lp\sigma_V\varphi\rp \lp\alpha,\beta,\gamma\rp .
\end{split}\end{equation}
\end{proof}

Since the Clifford algebra product only depends on the metric, notice that the difference between eqns (\ref{grig82}, \ref{grig83}) is that the first assumes a choice of metric whereas the second assumes a choice of $G_2$-structure $\varphi_{\xi}$. One may then define the linear map \bege j_{\xi}:\Gamma(\spb)\rightarrow \Gamma(\oct M)\enge
given by
\begin{align}
    j_{\xi}(\xi)&=1\\
    j_{\xi}(V\cdot \eta)&=V\circ_{\varphi_{\xi}}j_{\xi}(\eta),
\end{align}
for every octonion $V$ and spinor field $\eta$. Notice that if $\eta=A\cdot \xi$ for some octonion $A\in\Gamma(\oct M)$ then
\begin{equation}
    j_{\xi}(\eta)=j(A\cdot \xi)=A.
\end{equation}

If one fixes a nowhere-vanishing spinors $\xi$ then there is a pointwise decomposition of $\spb$ as $\re\cdot\xi\oplus\{ X\cdot\xi\;:\;X\in\re^7\}$, so that every spinor $\eta$ can be writen as $\eta=A\cdot\xi$ for some octonion $A$. Therefore, $j_{\xi}$ is a pointwise isomorphism of real vector space from spinors to octonions.
\begin{lemma}\label{grig83l}
The map $j_{\xi}$ preserves the inner products, namely
\begin{equation}
    \langle\eta_1,\eta_2\rangle_{\spb}=\langle j_{\xi}(\eta_1),j_{\xi}(\eta_2)\rangle_{\oct}.
\end{equation}
\end{lemma}
\begin{proof}
Indeed, let $V_1$ and $V_2$ be octonions such that $\eta_1=V_1\cdot\xi$ and $\eta_2=V_2\cdot \xi$ with $V_1=(a_1,v_1)$ and $V_2=(a_2,v_2)$. Then,
\begin{equation}\begin{split}
    \langle \eta_1,\eta_2\rangle_{\spb}&=\langle V_1\cdot\xi,V_2\cdot\xi\rangle_{S}\\
    &=a_1a_2\Vert\xi\Vert^2+\langle v_1\cdot V_1,v_2\cdot V_2\rangle_{\spb}\\
    &=a_1a_2\Vert \xi\Vert^2+\langle v_1,v_2\rangle\Vert\xi\Vert^2\\
    &=\langle V_1,V_2\rangle_{\oct}\\
    &=\langle j_{\xi}(\eta_1),j_{\xi}(\eta_2)\rangle_{\oct}.
\end{split}\end{equation}
\end{proof}

With respect to the change of reference, one can see that if one fixes a non-vanishing unit spinor $\xi$ then under $j_{\xi}$ by the last result there holds
\begin{equation}\begin{split}
    \varphi_{\xi}(\alpha,\beta,\gamma) &=-\langle\xi,\alpha\cdot\lp \beta\cdot\lp \gamma\cdot\xi\rp \rp \rangle_{\spb}\\
    &=-\langle j_{\xi}\lp \xi\rp ,\alpha\lp \beta\lp \gamma\lp j_{\xi}\lp \xi\rp \rp \rp \rp \rangle_{\oct}\\
    &=-\langle 1,\alpha\lp \beta\lp \gamma\rp\rp \rangle\\
    &=\langle\alpha,\beta\gamma\rangle,
\end{split}\end{equation}
as expected. Then, if $\eta=A\cdot \xi$ by Lemma \ref{grig83l} it follows that
\begin{equation}\begin{split}
    \varphi_{\eta}(\alpha,\beta,\gamma)&=-\langle\eta,\alpha\cdot\lp \beta\cdot\lp \gamma\cdot\eta\rp \rp \rangle_{\spb}\\
    &=-\langle j_{\xi}\lp \eta\rp ,\alpha\lp \beta\lp \gamma\lp j_{\xi}\lp \eta\rp \rp \rp \rp \rangle_{\oct}\\
    &=-\langle A,\alpha\lp \beta\lp \gamma\lp A\rp \rp\rp \rangle,
\end{split}\end{equation}
where the octonion product is given with respect to $\varphi_{\xi}$. It follows from eqn (\ref{grig83}) that
\begin{equation}\label{grig87}
    \varphi_{A\cdot \xi}=\sigma_A(\varphi_{\xi}).
    \end{equation}

\begin{corollary}
Let $\xi$ be a nonvanishing unit spinor on a $7$-dimensional manifold $M$ and let $\varphi_{\xi}$ be its associated $G_2$-structure. Then, for any unit octonions $U$ and $V$ there holds
\begin{equation}
    \varphi_{U\cdot(V\cdot\xi)}=\varphi_{(UV)\cdot\xi}.
\end{equation}
\end{corollary}
\begin{proof}
Theorem \ref{grig48} asserts that
\begin{equation}
    \sigma_U(\sigma_V\varphi_{\xi})=\sigma_{UV}\varphi_{\xi}.
\end{equation}
On the other hand, from eqn (\ref{grig87}) it comes
\begin{equation}\begin{split}
    \sigma_U(\sigma_V\varphi_{xi})&=\sigma_U(\varphi_{V\cdot\xi})=\varphi_{U\cdot (V\cdot \xi)}\\
    \sigma_{UV}\varphi_{\xi}&=\varphi_{(UV)\cdot\xi},
\end{split}\end{equation}
which gives the desired result.
\end{proof}

Furthermore, one may endow $\spb$ with a connection $\nabla^{\spb}$ lifted from the Levi-Civita connection $\nabla$ over $M$. It has the property that if $\eta=A\cdot\xi$ then
\begin{equation}
    \nabla^{\spb}_X \eta=(\nabla_X A)\cdot\xi+A\cdot\nabla_X^{\spb}\xi.
\end{equation}
It follows that there is an endomorphism $T^{(\xi)}:\mathfrak{X}(M)\rightarrow\mathfrak{X}(M)$ such that \cite{agricola}
\begin{equation}\label{grig89}
    \nabla_X^{\spb}\xi=-T^{(\xi)}(X)\cdot \xi,
\end{equation}
where $T^{\xi}$ is called the torsion tensor of $\varphi_{\xi}$. It then follows
\begin{theorem}
Let $\xi\in\Gamma(\spb)$ be a nonvanishing unit spinor on a $7$-dimensional manifold $M$ and $\varphi_{\xi}$ its associated $G_2$-structure. Then, for every $\eta\in\Gamma(\spb)$ there holds
\begin{equation}
    j_{\xi}(\nabla^{\spb}_X \eta)=D_X^{(\xi)}(j_{\xi}(\eta)),
\end{equation}
where $D^{(\xi)}$ is the octonion covariant derivative with respect to the $G_2$-structure $\varphi_{\xi}$.
\end{theorem}
\begin{proof}
From (\ref{grig89}), it comes
\begin{equation}\begin{split}
    j_{\xi}(\nabla^{\spb}_X\xi)&=-T^{(\xi)}(X)\\
    &=D^{(\xi)}_X 1\\
    &=D_X^{(\xi)}j_{\xi}(\xi).
\end{split}\end{equation}
Then, for $\eta=A\cdot \xi$ the identity
\begin{equation}
    \nabla^{\spb}_X \eta=(\nabla_X A)\cdot\xi+A\cdot\nabla_X^{\spb}\xi
\end{equation}
allied with the defining relations for $j_{\xi}$ yield
\begin{equation}\begin{split}
    j_{\xi}(\nabla_X^{\spb}\eta)&=(\nabla_X A)\cdot j_{\xi}(\xi)+A\cdot j_{\xi}(\nabla_X^{\spb}\xi)\\
    &=\nabla_X A-AT^{(\xi)}(X)\\
    &=D^{(\xi)}_X A\\
    &=D_X^{(\xi)}j_{\xi}(\eta).
\end{split}\end{equation}
\end{proof}

Therefore, the isomorphism $\spb\simeq\oct M$ provided by $j_{\xi}$ for a choice of nonvanishing unit spinor $\xi$ gives an isometric relation which maps the spin bundle connection $\nabla^{\spb}$ to the octonion covariant derivative $D^{\xi}$. However, since the Clifford algebra is associative, the octonion algebra contains more information. The octonion product can be further defined in terms of projections of Clifford products as seen in \cite{roldaovaz} and the product deformation as in eqn (\ref{grig420}) can be perceived therein, and its relation with spinor fields over the $7$-sphere $S^7$ with non-vanishing torsion can be scrutinized \cite{aquerold,ced}. Furthermore, with the study of $G_2$-structures we can extend the formalism herein introduced, emulating $S^7$ spinors into current algebras and Kac-Moody algebras, as in \cite{ced, Cederwall:1994ep}.

\newpage
\phantomsection
\addcontentsline{toc}{chapter}{\;\;\;\;\;\textsc{conclusion}}
\chapter*{\textsc{conclusion}}
\vspace{1cm}
\vspace{1cm}

The objective of this work was to formulate more general descriptions of geometries on manifolds which could be further considered within the framework of theoretical physics. In order to do that, affine connections with non-vanishing torsion were analyzed, being those key ingredients to work with the Kaluza-Klein supergravity theories in $7$ dimensions. Besides, the octonion product was integrated to $7$-dimensional manifolds and octonionic fields in this context were seen to relate to the important notion of spinor fields in physics, which may yield generalizations in the parallelizable $7$-sphere $S^7$, as in \cite{aquerold,ced,Cederwall:1994ep}.

A brief introduction on vector bundles and Riemannian geometry was given, and affine connections over a manifold were considered in a way that some useful properties of the Levi-Civita connection could still be perceived by allowing the notion of totally anti-symmetric contorsion. The Riemannian metric was shown to be a great tool in order to define normal coordinates and was introduced only when needed.

Then, geodesic loops were constructed onto affinely connected spaces and their fundamental tensors considered. The tangent space was endowed with the $W$-algebra operations and the fundamental tensor were related to the underlying notions of torsion and curvature. This apparatus was then considered in the context of the Kaluza-Klein $d=11$ spontaneous compactification theory of supergravity, where the equations of motion were seen to yield geometric constraints over the ground state by using the techniques of geodesic loops heretofore scrutinized.

In addition, the normed division algebras were analyzed and their properties exposed. A treatment on the algebra $\oct$ was given and by means of the $3$-form $\varphi$ a $G_2$-structure over the vector space $\re^7$ was considered. Further on, the notion of $G_2$-structures over $7$-dimensional manifolds was analyzed, enabling such space to be endowed with an octonion-like product, yielding interesting relations with the spinor bundle and its covariant derivative.

Apart from the above-mentioned possible generalizations for spinor fields emerging from the octonion bundle, it is also possible to consider more general connections when discussing the octonion covariant derivative. For instance, the results found in \cite{grigorian} on this matter may be generalized for deformations of the Levi-Civita connection by a totally anti-symmetric contorsion, since it still satisfies the metric-compatibility property, which was extensively used. The relation between geodesic (and local) loops and more general global (Lie) loops and the topological constraints to their existence may be further considered. Finally, we believe there may be a link between the torsion of a $G_2$-structure and the underlying torsion of a connection which can be perceived locally by the fundamental tensors of the geodesic loop.

\backmatter

\addcontentsline{toc}{chapter}{Bibliography}


\clearpage
\phantomsection
\addcontentsline{toc}{chapter}{Índice Remissivo}
\printindex
\end{document}

%% file: capa.tex
\thispagestyle{empty}
\includegraphics[height=2.7cm, keepaspectratio=true]{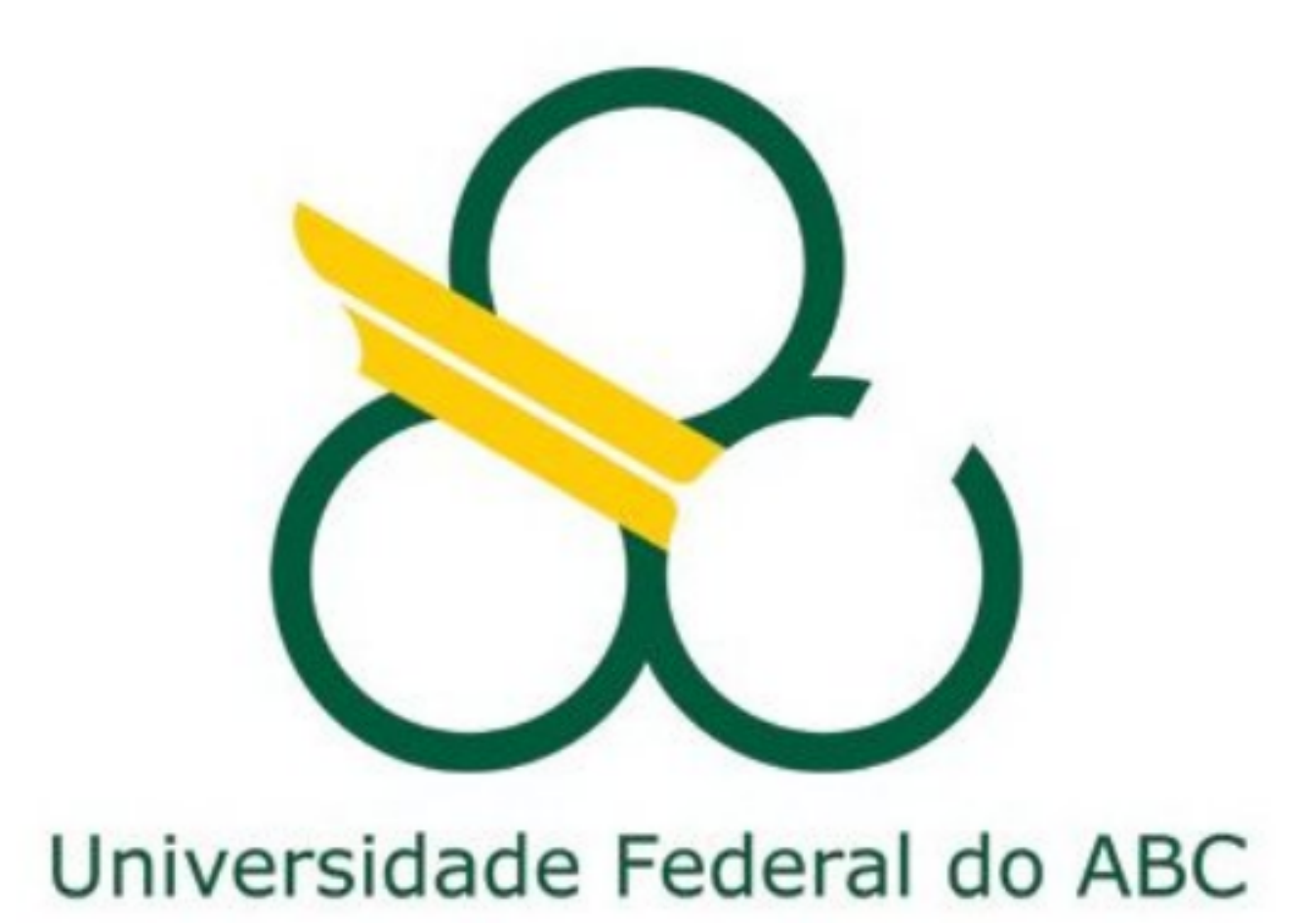} \hfill

\vspace*{2.8cm}
\begin{center}
  {\Large \scshape \autor}
\end{center}
\vspace{5cm}
\begin{center}
  {\huge \scshape \bfseries \titulo}
\end{center}
\vfill
\begin{center}
  \vspace{0.8cm}
  {\bfseries Santo  André, \ano}
\end{center}

%% file: folha-de-rosto.tex
\thispagestyle{empty}
\includegraphics[height=2.7cm, keepaspectratio=true]{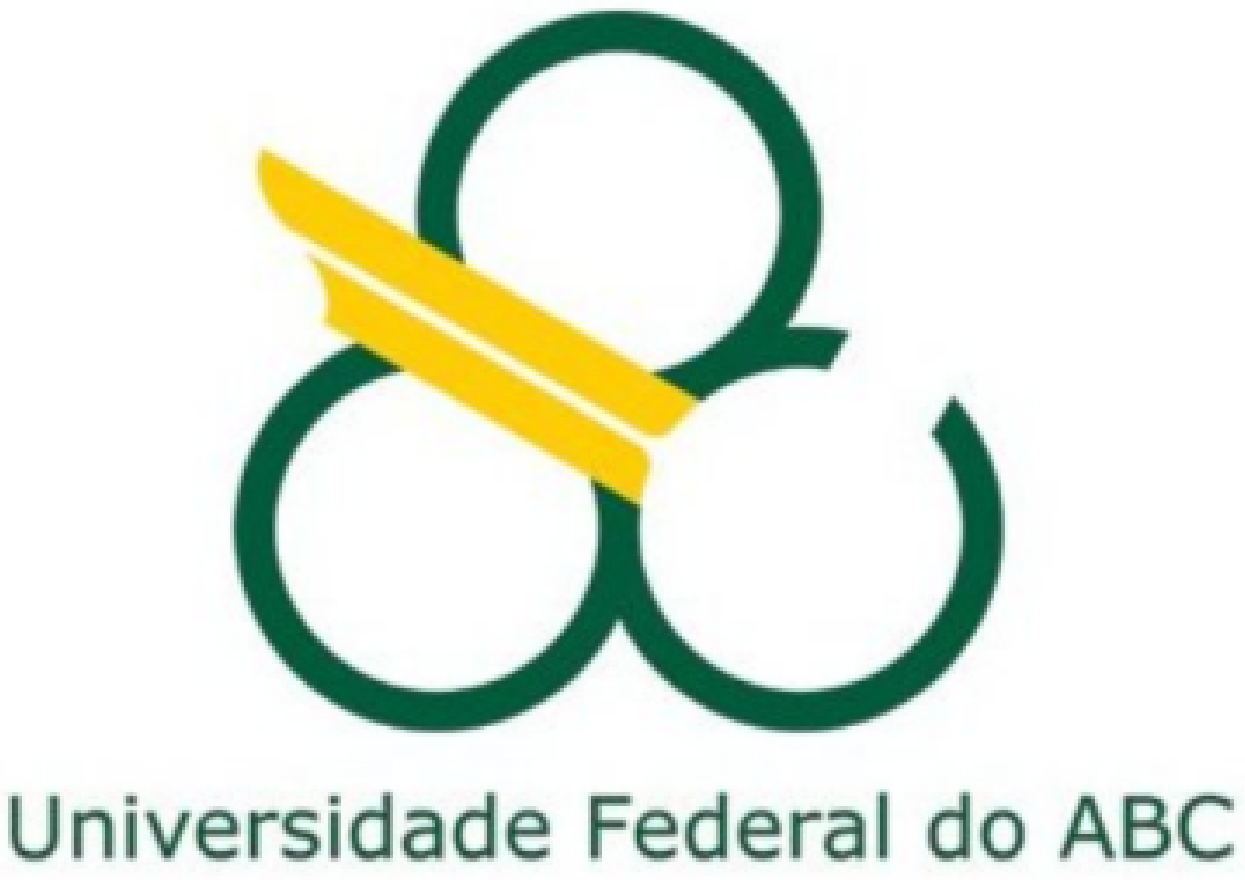} \hfill

\vspace{1.37cm}
\begin{center}
  {\large \scshape \bfseries Universidade Federal do ABC
  \vspace{.5cm}

\centro}
\end{center}
\vspace{.7cm}
\begin{center}
  {\large \scshape \bfseries \autor}
\end{center}
\vspace{.7cm}
\begin{center}
  {\LARGE \scshape \bfseries \titulo}
\end{center}
\vspace{1cm}
{\bfseries
\noindent
Orientador\ifx\femaleOrientador\undefined
\else
a\fi: Prof\ifx\femaleOrientador\undefined
\else
a\fi. Dr\ifx\femaleOrientador\undefined
\else
a\fi. \orientador
\vspace{.25cm}

\ifx\coorientador\undefined
\else
\noindent
Coorientador\ifx\femaleCoorientador\undefined
\else
a\fi: Prof\ifx\femaleCoorientador\undefined
\else
a\fi. Dr\ifx\femaleCoorientador\undefined
\else
a\fi. \coorientador
\fi
}

\vspace{1.4cm}
\begin{flushright}
  \begin{minipage}[c]{.6\textwidth}
    \begin{flushleft}
      \ifx\mestrado\undefined
      \noindent Tese de doutorado
      \else
    \noindent   Dissertação de mestrado
      \fi
      apresentada ao \centro para \\ \noindent  obtenção do título de
      \titulacao
    \end{flushleft}
  \end{minipage}
\end{flushright}
\vspace{1.3cm}
  \ifx\versaofinal\undefined
\noindent 
{\footnotesize \scshape
Esta é a versão original da tese, tal como\\
submetida à Comissão Julgadora.\\
}
\else
\noindent 
{\footnotesize \scshape
Este exemplar corresponde à versão final da 
\ifx\mestrado\undefined
tese
\else
dissertação
\fi \\
defendida 
\ifx\femaleAuthor\undefined
pelo aluno 
\else
pela aluna
\fi
\autor,\\
e orientada pel\ifx\femaleOrientador\undefined
o\else
a\fi{} Prof\ifx\femaleOrientador\undefined
\else
a\fi. Dr\ifx\femaleOrientador\undefined
\else
a\fi. \orientador.
}
\fi

\vspace{1.2cm}

\vfill
\begin{center}
  {\small \scshape \bfseries Santo André, \ano}
\end{center}

%% file: dedicatoria.tex
\chapter*{}
\vspace{3cm}
{This study was financed by grant 2018/10367-2 - São Paulo Research Foundation (FAPESP) and in part by the Coordenação de Aperfeiçoamento de Pessoal de Nível Superior - Brasil (CAPES) - Finance Code 001.}

\vspace{12cm}

\begin{flushleft}
{As opiniões, hipóteses e conclusões ou recomendações expressas neste material são de responsabilidade do(s) autor(es) e não necessariamente refletem a visão da FAPESP.}
\end{flushleft}

%% file: agradecimentos.tex
\chapter*{Agradecimentos}

As forças que impulsonaram a conclusão deste trabalho vêm de diversos lugares. Desta forma, agredeço primeiramente a todos aqueles que passaram por minha vida durante esses anos e os passados, implicando-se, assim, necessários ao resultado atual.

Ao guia: agradeço ao meu orientador Roldão, sem o qual esse trabalho não seria possível. Agradeço-o imensuravelmente pela sua dedidação, foco e visão, assim como a disponibilidade, quase sempre imediata, para conversar sobre qualquer questão. Estas diretrizes me acompanharão em minha vida acadêmica futura, com a maior das certezas.

Aos amigos: Matheus Martins, por ser meu terceiro irmão para tudo que existe; ao André Gomes, pelas conversas infinitas sobre matemática e tudo o mais, revelando muitas incertezas e (portanto) muitas maravilhas; ao Filipe Marçal, pela suprema humanidade, fraternidade e também pelas excelentes indicações musicais; a todo o Squadex, pelas risadas e companhias de madrugadas, enfim, a todos meus amigos que entenderam minha ausência em certos eventos nesses últimos anos.

Ao amor: gostaria de agradecer à minha esposa Ariel pela absoluta paciência e entendimento nos dias de reclusão para que eu pudesse compor este trabalho. Sem sua ajuda e carinho eu certamente não conseguiria realizá-lo do melhor modo. Agradeço pelas aventuras passadas (tanto ao vivo, quanto em Azeroth) e as que virão. Sabra!

Por final, às raízes: agradeço aos meus pais, Vilma e Aquerman, pelo ambiente extremamente livre que criaram em nossa casa, pelas conversas e por, nelas, ouvirem nossas opiniões e ideias. Aos meus irmãos, Raphael e Pedro, por serem as melhores pessoas para conversar sobre absolutamente qualquer coisa. Às minhas avós, Vilma e Irene, por terem cuidado de nós; ao meu avô, Carlos, por ter me ensinado a tabuada, que guardava num compartimento especial atrás da cabeça e ao meu avô, Vicente, por ter me demonstrado que sempre é uma boa hora para uma música. Obrigado.

%% file: epigrafe.tex
\chapter*{}
\vfill
\begin{flushright}
"Nothing ever exists entirely alone;\\
 everything is in relation to everything else."\\
 -- Bukkyo Dendo Kyonkai\\
 (The teachings of the Buddha)

\end{flushright}

%% file: resumo.tex
\chapter*{Resumo}

Investigaremos deformações do produto octoniônico advindas da torção paralelizável sobre a $7$-esfera $S^7$, obtendo uma família de geometrias que surge como novas soluções de equações de movimento no formalismo Lagrangiano. Isso é feito ao se considerar a compactificação espontânea $M_4\times S^7$, onde $M_4$ denota uma variedade Lorentziana $4$-dimensional. Além da geometria Riemanniana convencional e das duas geometrias propostas por Cartan e Schouten, soluções em geometrias com torção e em espaços de sete dimensões mais gerais são obtidas. Tal formalismo será ulteriormente também derivado na $7$-esfera $S^7$ com torção paralelizável, dada localmente pelas constantes de estruturas de um \textit{loop} geodésico não-associativo. Estruturas $G_2$ em variedades de sete dimensões serão ainda investigadas, com a introdução dos produto e fibrado octoniônicos $\oct M$. Neste cenário, seções deste fibrado sobre tais espaços podem ser interpretados como campos espinoriais sob uma identificação isometrica levando a conexão espinorial à derivada covariante octoniônica relacionada ao produto definido sobre $\oct M$.

\vspace{.5cm}
\textbf{Palavras-chave}: \palavraschaves

%% file: abstract.tex
\chapter*{ Abstract}
\selectlanguage{english}

We investigate octonion product deformations coming from the parallelizable torsion of the 7-sphere $S^7$, obtaining a family of geometries from solutions of the Lagrangian formalism movement equations. This can be achieved by analyzing the spontaneous compactification $M_4\times S^7$, where $M_4$ is a Lorentzian $4$-dimensional manifold. Besides the usual Riemannian geometry and two others proposed by Cartan and Schouten, solutions in geometries with torsion and more general seven-dimensional spaces are obtained. Such formalism may by subsequently derived over the 7-sphere $S^7$, locally given by the structure constants of a nonassociative geodesic loop. Furthermore, $G_2$-structures are investigated, giving rise to the octonion product and bundle $\mathbb{O} M$ over a seven-dimensional manifold $M$. Then, sections of this bundle over such space can be perceived as spinor fields in an isometric identification mapping the spin connection to an octonion covariant derivative preserving the octonion product defined over $\oct  M$.

\vspace{.5cm}
\textbf{Keywords}: \keywords

\selectlanguage{english}